\documentclass[11pt]{article}

\usepackage{catchfilebetweentags}
\usepackage{geometry}
\input{macros}

\begin{document}

\begin{center}
    {{\LARGE{ \mbox{High-dimensional estimation with missing data: }\\
    \mbox{Statistical and computational limits}}}}

    	\vspace*{.2in}
	
	{\large{
			\begin{tabular}{ccc}
                Kabir Aladin Verchand$^{\dagger}$, Ankit Pensia$^{\circ}$,
				Saminul Haque$^{\star}$, and Rohith Kuditipudi$^{\star}$
			\end{tabular}
	}}
    
	\vspace*{.2in}
    
    	\begin{tabular}{c}
		$^{\dagger}$ Department of Data Sciences and Operations, University of Southern California \\
		$^{\circ}$ Department of Statistics and Data Science, Carnegie Mellon University\\
        $^{\star}$ Department of Computer Science, Stanford University
	\end{tabular}

    \vspace*{.2in}
	\today
    
	\vspace*{.2in}
    \begin{abstract}
    We consider computationally-efficient estimation of population parameters when observations are subject to missing data.  In particular, we consider estimation under the realizable contamination model of missing data in which an $\epsilon$ fraction of the observations are subject to an arbitrary (and unknown) missing not at random (MNAR) mechanism.  When the true data is Gaussian, we provide evidence towards statistical-computational gaps in several problems.  For mean estimation in $\ell_2$ norm, we show that in order to obtain error at most $\rho$, for any constant contamination $\epsilon \in (0, 1)$, (roughly) $n \gtrsim d e^{1/\rho^2}$ samples are necessary and that there is a computationally-inefficient algorithm which achieves this error.  On the other hand, we show that any computationally-efficient method within certain popular families of algorithms
    requires a much larger sample complexity of (roughly) $n \gtrsim d^{1/\rho^2}$ and that there exists a polynomial time algorithm based on sum-of-squares which (nearly) achieves this lower bound.  For covariance estimation in relative operator norm, we show that a parallel development holds.  Finally, we turn to linear regression with missing observations and show that such a gap does not persist.  Indeed, in this setting we show that minimizing a simple, strongly convex empirical risk nearly achieves the information-theoretic lower bound in polynomial time.   \renewcommand\thefootnote{}\footnote{Authors are listed in random order.}
                \addtocounter{footnote}{-1}
    \end{abstract}
\end{center}

\tableofcontents
\newpage
\section{Introduction} \label{sec:intro}
Statistical procedures are typically defined under the idealized assumption that each observation $\{X_i\}_{i \in [n]}$ in a sample of size $n$ is an independent outcome from a population distribution $P$.  In practice, it is often the case that observations depart in some way from this idealized scenario.  One common departure is that of missing data, in which each observation may only be partially revealed.  Of key importance when handling missing data is the nature of the mechanism by which data is missing.  Typically, these ``missingness mechanisms'' are classified into one of three categories in increasing flexibility: Missing completely at random (MCAR), Missing at random (MAR), and Missing not at random (MNAR).  In brief, the MCAR assumption ensures that the missingness mechanism is independent of the underlying observations; the MAR assumption ensures that the missingness mechanism depends only on the observed data\footnote{See~\cite{seaman2013what} for a precise definition.}; and the MNAR assumption places no restriction on the missingness mechanism.  For a detailed exposition of each of these three categories, we refer the reader to the book~\cite{little2014statistical}.  

While the former two assumptions often enable identification of population parameters, they are (i.) often too strong to hold in practice and (ii.) impossible to test for~\cite{gill1997coarsening}.  On the other hand, the MNAR assumption is typically too flexible to enable identification of the population parameters~\cite{manski2003partial}.  In particular, when data are subject to missing not at random mechanisms, the resulting observations may be significantly biased.  Motivated by this discrepancy,~\cite{ma2024estimation} recently introduced the realizable contamination framework which considers Huber-style contamination tailored to the setting of missing data.  In this paper, we adopt this model, which we describe next.  

Throughout the main text we will consider a simplified \emph{all-or-nothing} setting in which observations have either no missing data or are fully missing.  Later, in Appendix~\ref{sec:multiple-patterns}, we show how our results naturally extend to the setting of multiple missingness patterns.  
\begin{figure}[h!]
    \centering
    \centering

    \begin{subfigure}[t]{0.48\textwidth}
        \centering
        \begin{tikzpicture}[x=0.8cm,y=0.8cm]

\foreach \r in {1,...,7} {
  \foreach \c in {1,...,6} {
    \pgfmathtruncatemacro{\miss}{
      (\r==2) || (\r==7)
    }

    \pgfmathsetmacro{\x}{\c-1}
    \pgfmathsetmacro{\y}{7-\r}

    \ifnum\miss=1
      \draw (\x,\y) rectangle ++(1,1);
      \node at (\x+0.5,\y+0.5) {$\star$};
    \else
      \fill[gray!35] (\x,\y) rectangle ++(1,1);
      \draw (\x,\y) rectangle ++(1,1);
    \fi
  }
}

\end{tikzpicture}
        \caption{All-or-nothing missingness.}
        \label{fig:all-or-nothing}
    \end{subfigure}
    \hfill
    \begin{subfigure}[t]{0.48\textwidth}
        \centering
        \begin{tikzpicture}[x=0.8cm,y=0.8cm]

\foreach \r in {1,...,7} {
  \foreach \c in {1,...,6} {
    \pgfmathtruncatemacro{\miss}{
      ((\r==1 || \r==3 || \r==7) && (\c==1 || \c==2)) ||
      ((\r==2 || \r==6) && (\c==5 || \c==6)) ||
      ((\r==4 || \r==5) && (\c==3 || \c==4))
    }

    \pgfmathsetmacro{\x}{\c-1}
    \pgfmathsetmacro{\y}{7-\r}

    \ifnum\miss=1
      \draw (\x,\y) rectangle ++(1,1);
      \node at (\x+0.5,\y+0.5) {$\star$};
    \else
      \fill[gray!35] (\x,\y) rectangle ++(1,1);
      \draw (\x,\y) rectangle ++(1,1);
    \fi
  }
}

\end{tikzpicture}
        \caption{Multiple missingness patterns.}
        \label{fig:multiple-pattern}
    \end{subfigure}
    \caption{Types of missing data patterns.  Each row indicates a single sample, where the
    entries in gray indicate an observed value and the $\star$ entries indicate missingness.  In order to simplify the results in the main text, we focus on the all-or-nothing patterns described in panel (a), deferring the extension to the more general setting to Appendix~\ref{sec:multiple-patterns}.}
    \label{fig:both}
\end{figure}

To concretely define the model, let $P$ denote a probability distribution on $\mathbf{R}^d$ and suppose that we are interested in estimating some quantity $\theta(P)$ from incomplete observations.  Following~\cite{ma2024estimation}, we define the following sets of distributions, specialized to the all-or-nothing setting:  
\begin{subequations}
\begin{align}
\mathsf{MCAR}_{(P, q)} &:= \Bigl\{\law\bigl(X \ostar \Omega\bigr):\, X \sim P \text{ and } \Omega \in \{(0)^d, (1)^d\}, \; \Omega \indep X, \; \mathbb{P}(\Omega = (1)^{d}) = q \Bigr\} \label{def:MCAR-aon}\\
\mnar_P &:= \Bigl\{\law\bigl(X \ostar \Omega\bigr):\, X \sim P \text{ and } \law(\Omega) \in \mathcal{P}\bigl(\{(0)^d, (1)^d\}\bigr)\Bigr\}. \label{def:MNAR-aon}
\end{align}
\end{subequations}
We point the reader to the notation section at the end of the introduction for a precise definition of the modified Hadamard product $\ostar$.  We emphasize that these distributions are on the extended space $\R^d \cup \{\star^d\}$.  In words, $\mathsf{MCAR}_{(P, q)}$ describes the missing completely at random set of distributions in which each observation is seen with probability $q$.  On the other hand, $\mathsf{MNAR}_P$ describes the set of missing not at random distributions which can be obtained by applying any missingness mechanism to the observations.

Since consistent estimation under the assumption-lean setting of $\mathsf{MNAR}_P$ is in general impossible, we consider the \emph{realizable contamination model} introduced by~\cite{ma2024estimation}.  In particular, given a contamination parameter $\epsilon \in [0, 1]$, the realizable set of distributions is given by 
\begin{align} \label{def:realizable-model}
\mathcal{R}(P, \epsilon, q) := (1- \epsilon) \mcar_{(P, q)} + \epsilon \mnar_P.
\end{align}
When $q=1$, we omit it, i.e., $\mathcal{R}(P, \epsilon) := \mathcal{R}(P, \epsilon, 1)$.
The following straightforward lemma (which is a corollary of~\cite[Proposition 2]{ma2024estimation}) characterizes the set of realizable distributions by upper and lower bounds on likelihood ratios.  
\begin{lemma} \label{lem:all-or-nothing-characterization}
The distribution $Q \in \mathcal{R}(P, \epsilon, q)$ if and only if, for all $z \in \mathbf{R}^d$,
\[
q (1-\epsilon) \leq \frac{\mathrm{d} Q}{\mathrm{d} P}(z) \leq q (1 - \epsilon) + \epsilon.
\]
\end{lemma}
We emphasize that, restricted to $\mathbf{R}^d$, this lemma implies that $Q$ is a sub-probability measure whose density is upper and lower bounded by that of $P$.  From this relation, we can read off upper and lower bounds on the conditional distribution $Q_{\bR}$ of the observations, given that they were observed.  In particular, note that 
\[
q (1-\epsilon) \leq Q(\mathbf{R}^d) \leq q (1 - \epsilon) + \epsilon.
\]
This in turn implies that if $Q \in \mathcal{R}(P, \epsilon, q)$, then
\begin{align} \label{ineq:sandwich-realizability}
1 - \frac{\epsilon}{q (1 - \epsilon) + \epsilon} \leq \frac{\mathrm{d} Q_{\bR}}{\mathrm{d} P}(z) \leq 1 + \frac{\epsilon}{q (1 - \epsilon)} \quad \text{ for all } z \in \mathbf{R}^d.
\end{align}
Most of our algorithms in the sequel will be based on this characterization.  We will occasionally use the notation $\mathcal{R}_{\R}(P, \epsilon, q)$ to denote the corresponding set of conditional distributions (e.g. on $\R^d$).  We note in passing that in this special case, the condition is similar to familiar sensitivity conditions considered in the causal inference literature, e.g.,~\cite{rosenbaum87sensitivity,zhao2019sensitivity} as well as models of sampling bias~\cite{sahoo2022learning}.  In particular, if we let $\Gamma = 1 + \frac{\eps}{q(1-\eps)}
\geq 1$, this is equivalent to a biased sampling model in which \emph{all} of the observations are biased with likelihood ratios bounded above and below by $1/\Gamma$ and $\Gamma$, respectively.  Indeed, we note that under an appropriate change of variables, all of our algorithms (and lower bounds) apply in this setting. 

We will assume throughout the remainder of this text that $q=1$.
This is without loss of generality by an appropriate rescaling of $\epsilon$ and $n$.  In particular, if both $q$ and $\epsilon$ are known, one can remove a $1-q'$-fraction of the completely missing observations and apply an algorithm designed for observations from the set $\mathcal{R}(P, \epsilon, 1)$.  The following lemma makes this precise.

\begin{lemma}\label{lem:gen-q-conversion}
  Let $\epsilon \in [0,1)$ and $q \in [0, 1]$ and using these define $\epsilon' = \frac{\epsilon}{\epsilon + q(1-\epsilon)}$ and $q' = \eps + q(1-\epsilon)$. Then, 
  \begin{align*}
    \cR(P, \epsilon, q) = \Bigl\{q' \cdot Q' + (1-q')\delta_{\{\star^d\}}:\; Q' \in \cR(P, \epsilon', 1)\Bigr\}.
  \end{align*}
\end{lemma}
We provide the proof of Lemma~\ref{lem:gen-q-conversion} to Section~\ref{sec:proof-lem-gen-q-conversion}.

\subsection{Contributions} 
We consider three particular estimation problems and overview our main results in each.  

\paragraph{Mean estimation.} We consider the family of base distributions $\mathcal{P}_{\R^d}$, where for each $\theta \in \R^d$, $P_{\theta} = \mathsf{N}(\theta, \sigma^2 I_d)$ (with known $\sigma > 0$) for $\theta \in \R^d$.
We show that in the realizable contamination model~\eqref{def:realizable-model}, any estimator $\widehat{\theta}$ of $\theta$ from $n$ observations which succeeds with probability at least $1 - \delta$ must satisfy
\[
\bigl \| \widehat{\theta} - \theta \bigr \|_2 \gtrsim \frac{\sigma \log(\frac{1}{1-\epsilon})}{\sqrt{\log \Bigl(1 + \frac{\epsilon^2}{1-\epsilon} \cdot \frac{n}{d+\log(\frac{1}{\delta})}\Bigr)}},
\]
and we provide an estimator which achieves this error.  See Theorem~\ref{thm:mean-est-it-ub-lb} for a precise statement.  We note that in terms of sample complexity, if $\delta, \sigma, q, \epsilon$ are fixed constants, then this implies that to reach error $\| \widehat{\theta} - \theta \|_{2} \leq \rho$, we require (roughly) $n \gtrsim d e^{1/\rho^2}$ samples.  

Turning to computationally-efficient rates, we show that, for $\epsilon$ a small enough constant, a polynomial time algorithm based on the sum-of-squares method achieves error $\rho$ as long as (roughly) $n \gtrsim d^{1/\rho^2}$ observations are available, where we hide multiplicative factors depending on $\rho$.  See Theorem~\ref{thm:mean-est-sos-upper-bound} for a precise statement.  We complement this, in Theorem~\ref{thm:sq-lower-bound-mean} with a nearly-matching computational lower bound against statistical query (SQ) algorithms, low-degree polynomial tests, sum-of-squares hierarchy, and polynomial threshold functions.
For the remainder of this section, we focus on statistical query algorithms for simplicity.
Together, these results provide evidence towards a large statistical--computational gap.  We summarize this state of affairs in Figure~\ref{fig:mean-gap}.

\begin{figure}[h!]
    \centering
    \begin{tikzpicture}[
  font=\large,
  tick/.style={black, line width=0.9pt},
]

\definecolor{sqhard}{RGB}{252,237,202}      %
\definecolor{easy}{RGB}{209,243,203}        %
\definecolor{impossible}{RGB}{242,211,206}  %

\def\xL{0}
\def\xR{12}
\def\yB{0}
\def\yT{0.7}

\def\xA{2.6}   %
\def\xB{7.0}   %

\fill[impossible] (\xL,\yB) rectangle (\xA,\yT);
\fill[sqhard]     (\xA,\yB) rectangle (\xB,\yT);
\fill[easy]       (\xB,\yB) rectangle (\xR,\yT);

\node at ({(\xL+\xA)/2},{(\yB+\yT)/2}) {Impossible};
\node at ({(\xA+\xB)/2},{(\yB+\yT)/2}) {SQ hard};
\node at ({(\xB+\xR)/2},{(\yB+\yT)/2}) {Easy};

\draw[black, line width=1.0pt] (\xL,\yB) -- (\xR,\yB);

\draw[tick] (\xA,\yB) -- (\xA,-0.35);
\draw[tick] (\xB,\yB) -- (\xB,-0.35);

\node[below, align=center] at (\xA,-0.35) {$d e^{1/\rho^2}$};
\node[below, align=center] at (\xB,-0.35) {$d^{1/\rho^2}$};

\node[left] at (\xL-0.15,\yB+0.18) {$n$};

\end{tikzpicture}
    \caption{Sample complexity phase diagram for mean estimation with  $\epsilon$ a fixed constant.  In order to achieve $\ell_2$ norm error $\rho$, it is information-theoretically necessary and sufficient to take $n \asymp de^{1/\rho^2}$ many samples.  On the other hand, any statistical query algorithm must take (roughly) $d^{1/\rho^2}$ many samples and a polynomial time algorithm (nearly) saturates this lower bound.
    }
    \label{fig:mean-gap}
\end{figure}
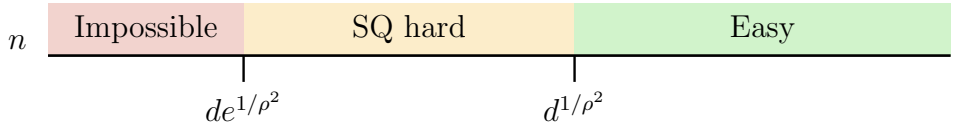

\paragraph{Covariance estimation.} We show that a parallel set of claims holds in the covariance estimation setting.
In particular, we consider the family $\mathcal{P}_{\rm cov}$, where for each $\Sigma \in \mathcal{C}^d_{++}$, $P_{\Sigma} = \mathsf{N}(0, \Sigma)$, we show that any estimator $\widehat{\Sigma}$ that succeeds with probability at least $1-\delta$ must satisfy 
\[
\bigl \| \Sigma^{-1/2}\widehat{\Sigma}\Sigma^{-1/2} - I_d \bigr \|_{\mathrm{op}} \gtrsim   \sqrt{\frac{d + \log(\frac{1}{\delta})}{n(1-\eps)}} + \frac{\log(\frac{1}{1-\epsilon})}{\log\left(1 + \frac{\epsilon n}{d+\log(\frac{1}{\delta})}\right)},
\]
and that a computationally-inefficient algorithm achieves nearly the same rate.  See Theorem~\ref{thm:cov-est-it-ub-lb} for an exact statement.  On the other hand, we show that for any fixed $\tau$ a small enough constant, in order to achieve error $\rho$ in relative operator norm, roughly $n \gtrsim d^{1/\rho}$ samples are sufficient for a polynomial time sum-of-squares based estimator to succeed (Theorem~\ref{thm:cov-est-sos-upper-bound}) and that nearly the same sample complexity is required of any statistical query algorithm (Theorem~\ref{thm:sq-lower-bound-cov}).

\paragraph{Linear regression.} In the setting of linear regression with missing observations in which the missingness may depend arbitrarily on both the response $Y$ as well as the covariate $X$, we show that a large statistical--computational gap does not persist.  In particular, we consider the class $\mathcal{P}_{\mathrm{LR}}(\sigma^2)$, which consists of covariate response pairs $(X, Y)$ distributed as $X \sim \mathsf{N}(0, I_d)$ and $Y \mid X \sim \mathsf{N}\bigl(X^{\top} \theta, \sigma^2\bigr)$.
In this setting, we show that any estimator $\widehat{\theta}$ of the coefficients $\theta$ which succeeds with probability at least $1 - \delta$ must incur error at least
\[
\bigl \| \widehat{\theta} - \theta \bigr \|_2 \gtrsim   \frac{\sigma \cdot \log(1/(1-\epsilon))}{\sqrt{\log\left(1 + \frac{\epsilon^2}{1-\epsilon} \cdot \frac{n}{d + \log(\frac{1}{\delta})}\right)}},
\]
and that there exists a simple polynomial time estimator based on minimizing a strongly convex empirical risk which achieves nearly matching error
\[
\bigl \| \widehat{\theta} - \theta \bigr \|_2 \lesssim \frac{\sigma \cdot \frac{\epsilon}{1 - \epsilon} \cdot \log\log\Bigl(\frac{n(1-\epsilon)}{d + \log(\frac{1}{\delta})}\Bigr)}{\sqrt{\log\left(1 + \frac{\epsilon^2}{1-\epsilon} \cdot \frac{n}{d + \log(\frac{1}{\delta})}\right)}}.
\]
Observe that for $\epsilon$ less than a sufficiently small constant, the upper bounds and lower bounds differ only by a mild $\log\log$ term.  See Theorem~\ref{thm:lin-reg-ub-lb} for details.

\subsection{Further related work} \label{sec:related}
The model and methods considered in this paper lie at the intersection of the study of missing data and robust estimation, and here we describe the relation to the most related results and models in each of these fields.

\paragraph{Missing data.} Missing data is a mature field and we refer the interested reader to the book~\cite{little2014statistical} for many of the classical models and results.  The realizable contamination model which we study in this manuscript was recently introduced by~\cite{ma2024estimation}, although we note that similar models have been proposed in the literature in restricted settings by~\cite{horowitz1995identification,bonvini2022sensitivity}.  As mentioned in the discussion following the sandwich relation in~\eqref{ineq:sandwich-realizability}, in the all-or-nothing setting, there is a direct connection between realizable contamination and biased sampling~\cite{sahoo2022learning,aronow2013interval} and sensitivity analysis in causal inference~\cite{rosenbaum87sensitivity}.

An advantage of the realizable contamination model is that it does not preclude consistent estimation even when a constant fraction of the observations are contaminated.  We note that considering constrained forms of MNAR which similarly allow identification of the population parameters have been considered in the literature, we direct the interested reader to, e.g.,~\cite{malinsky2022semiparametric,mohan2013graphical,rotnitzky1998semiparametric}, and the references therein for a small subset of this literature.  We emphasize that the realizable contamination model encapsulates all of these modeling choices when $\epsilon = 1$ as it places no assumptions on the contamination component other than that it masks the original data.  On the other hand, we note that stronger forms of contamination (which can mask and alter the base distribution) have been considered in the missing data literature.  For instance,~\cite{hu2021robust} consider an analogy of the strong contamination model~\cite{DiaKan22-book} and derive nearly optimal mean estimators in this setting.  In the setting of linear regression,~\cite{diakonikolas2025linear} consider adversarial contamination and provide optimal and computationally-efficient algorithms for estimating the regression coefficients.  We emphasize that a key distinction between these stronger forms of contamination and the realizable setting considered here is that our restricted model often enables identification in settings where the strong contamination model does not.  

Our focus is on computationally-efficient methods.  This has been a focus of the literature on truncated statistics over the last several years (as we discuss in the next paragraph).  We remark here that even in the more benign MCAR setting, missingness can induce computational difficulties.  For instance,~\cite{loh2012high} consider sparse linear regression with MCAR covariates and provide a computationally-efficient algorithm for estimation despite nonconvexity of a natural plug-in estimator.  Moving beyond MCAR, there has been a flurry of work on the computational aspects of estimation from truncated statistics which we detail next. 

\paragraph{Truncated statistics.} Estimation from truncated samples is a classical problem of estimating a distribution, or its parameters, when given i.i.d.\ samples conditioned on the samples falling in an observation set $S$. 
Recent literature~\cite{DasTZ18} re-visited this classical problem---under the assumption of a Gaussian base distribution---with an eye towards computationally-efficient algorithms.  These authors show that in many learning settings, when given oracle access to the truncation set $S$, computationally-efficient estimation of population parameters is possible; by contrast, the same authors show that without such oracle access, it is information-theoretically impossible to perform estimation without further assumptions on the set $S$.  Motivated by this, subsequent work~\cite{kontonis2019efficient} further explored the unknown truncation set $S$, under the additional assumption that $S$ is known to have additional regularity such as belonging to a class of bounded VC dimension or Gaussian surface area.  For instance, these authors show that if the truncation set $S$ belongs to a collection $\mathcal{C}$, then it is possible to estimate the underlying mean with a sample complexity scaling as $\widetilde{O}(\mathrm{VC}(\mathcal{C})/\rho + d^2/\rho^2)$, albeit with an inefficient algorithm.  
Following this,~\cite{diakonikolas24statistical} show that there exist truncation sets with small VC dimension and Gaussian surface area such that a superpolynomial sample complexity is required for efficient algorithms in the statistical query model. 

The setting with unknown truncation is more closely related to the all-or-nothing setting considered in our paper, where we can understand the contamination set as the set of all possible truncated distributions, that is, without any structural assumptions on $S$.  Given the similarity between the two settings, it is worthwhile to open a small parenthesis to discuss the differences, which are crucial and lead to different types of algorithms.
In the truncated statistics problem, the assumptions imposed on $S$ permit truncating the entirety of the tails, so the tail behaviour within the model cannot be used to estimate the distribution.
By contrast, in our MNAR setting, we allow the likelihood ratio $\frac{\mathrm{d}Q}{\mathrm{d}P}$ to vary arbitrarily within $[1-\epsilon, 1]$, independently across points.
When $\epsilon$ is close to 1, this renders the bulk of the distribution unhelpful for estimation, as $Q$ can hide all the potential variation of $\frac{\mathrm{d}P'}{\mathrm{d}P}(z)$ for $P'$ in a neighborhood of $P$.  However, when $P$ belongs to a known class of distributions with rapidly decaying tails (e.g. Gaussian), a small perturbation can lead to a large likelihood ratio in the tails.   It is this feature of the tails that we exploit in our problems of interest.
Thus, while estimation from truncated samples setting is achieved by focusing on the bulk of the distribution, estimation from MNAR samples is achieved by focusing on the tails of the distribution.
Despite this contrast, the computationally-efficient sample complexities of mean and covariance estimation are similar in the two settings.

One major difference is that in linear regression, we obtain a computationally-efficient algorithm that requires only linear in $d$ many samples, while in the truncated sample setting, existing algorithms incur either $d^{O(1/\rho^2)}$ samples~\cite{LeeMeZa24} or require truncation which can only depend on the response~\cite{KouridakisMeKaCa26}.  

\paragraph{Robust estimation.} For a thorough background on robust estimation, we refer the reader to the books \cite{HubRon09} (for its historical treatment) and \cite{DiaKan22-book} (for its algorithmic aspects).

The prototypical contamination model considered in robust statistics is the Huber contamination model,
where the statistician observes
$Q = (1- \epsilon_{\mathrm{Huber}})P + \eps_{\mathrm{Huber}} R$, where $R$ is an arbitrary distribution and $\eps_{\mathrm{Huber}} \in [0,1/2)$ is the contamination rate.
Even when $P$ belongs to a nice family of distributions such as Gaussians,
the Huber contamination model is strong enough to preclude consistent estimation of parameters such as mean, covariance,
and linear regression.
Furthermore, realizable contamination can be seen as a special case of Huber contamination: the first inequality in
\Cref{ineq:sandwich-realizability} implies that the conditional distribution $Q_\R$ of $Q \in \cR(P,\eps,1)$
is a valid Huber contamination as long as $\eps_{\mathrm{Huber}} \geq \epsilon$.
However, realizable contamination contains additional structural information (the second inequality in \Cref{ineq:sandwich-realizability}) that leads to consistency.

Our work could also be seen as part of a broader research agenda that studies practical restrictions of Huber contamination model, 
which lead to improved statistical rates.
We discuss some related works below.

In the context of mean estimation,
multiple recent works have studied the mean-shift contamination model~\cite{Li23-oblivious,KotGao25,DiaIKP25,KalKLZ26-sparse-fft,DiaIKL26}, which places a different kind of constraint on the contamination distribution. In the basic setting, one observes samples from $(1-\eps) \mathsf N(\mu,I_d) \, + \,\eps R \,\ast\, \mathsf N(0,I_d)$, where $R$ is an arbitrary distribution and $\ast$ denotes the convolution operator.
These works have shown that consistent mean estimation is possible in this model: for any constant $\eps \in (0,1)$, the sample complexity to get error $\rho$  is roughly ${d}/{\rho^2} + e^{\Theta(1/\rho^2)}$, and furthermore there is no substantial statistical--computational gap.
While this rate is similar to that of realizable contamination for $d=1$, 
the statistical sample complexity for realizable contamination is much higher for $d>1$, scaling as $d e^{\Theta(1/\rho^2)}$(and the computational sample complexity is significantly higher, scaling as $d^{\Theta(1/\rho^2)}$).

In the context of linear regression, the mean-shift contamination model has the following analog:
\begin{align}
    X &\sim \mathsf N(0,I_d),\qquad{\rm and}\qquad  y \mid X \sim (1-\eps) \cdot \mathsf{N}(X^\top \theta_\star,\sigma^2) + \eps \cdot R_X \ast D_{X^\top \theta_\star}\,,
    \label{eq:related-work-adaptive}
\end{align}
where $R_X$ is a univariate distribution that may depend on the covariate $X$.
When the contaminating distribution, $R_X$, is oblivious of $X$, i.e., $R_X = R$, then it is termed an \emph{oblivious} contamination model.
For the oblivious contamination model, multiple works have shown that consistent estimation is possible~\cite{Tsakonas14, JaiTK14,BhatiaJK15,BhatiaJKK17, 
SugBRJ19,Steurer21Outliers}, and there exist (at most quadratic) statistical--computational gaps~\cite{DiaGKLP25-oblivious}. 
When the distribution $R_X$ is allowed to depend on $X$, it is termed \emph{adaptive} contamination.
The forthcoming work \cite{DiaGKPX26-adaptive} studies this adaptive contamination model, both from computational and statistical perspective.
While \Cref{eq:related-work-adaptive} can be seen as a more general contamination model than realizable contamination with missing responses (see Eqs.~\eqref{eq:model-lr}), 
the statistical rates surprisingly coincide; this follow by 
combining our lower bounds (which builds on their work) and their upper bounds.
By contrast, the computational rates differ widely, pointing to the computational benefits of realizable contamination over adaptive contamination.
To elaborate, our work gives a computationally-efficient algorithm with nearly-matching rate for realizable contamination, and \cite{DiaGKPX26-adaptive} establishes a statistical query lower bound for adaptive contamination of $d^{1/\rho^2}$ for any constant $\epsilon$. 
While their SQ lower bound of $d^{1/\rho^2}$ for \Cref{eq:related-work-adaptive} looks similar to our SQ lower bound for realizable mean estimation, it is unclear if there is a deeper connection in terms of a formal reduction.

\subsubsection*{Notation}
We let $\R, \R_{+}, \R_{++}$ denote the set of real numbers, non-negative real numbers, and positive real numbers, respectively.  We let $\mathbf{N}$ denote the set of natural numbers.  We let $\mathbf{S}^{d-1} = \{u \in \R^d: \; \| u \|_2 = 1\}$ denote the unit sphere.
We let $\mathcal{C}^d_+ = \{X \in \bR^{d \times d}:\, X = X^{\top} \text{ and } X \succeq 0\}$
and $\mathcal{C}^d_{++} = \{X \in \bR^{d \times d}:\, X = X^{\top} \text{ and } X \succ 0\}$ denote the cones of positive semidefinite and positive definite matrices, respectively.  We use $\1_{A}$ and $\1\{A\}$ interchangeably to denote the indicator function on the set $A$.

We let $\mathbf{R}_{\star} = \mathbf{R} \cup \{\star\}$ denote an extended space where the value $\star$ denotes a missing entry, and let $\mathbf{R}^k_{\star} = \underbrace{\mathbf{R}_{\star} \times \ldots \times \mathbf{R}_{\star}}_{k \text{ times}}$. We say that the support of a binary vector $\omega \in \{0, 1\}^k$, denoted by $\mathrm{supp}(\omega)$, is the set of coordinates on which $\omega$ takes the value $1$. We define the operation $\ostar: \R^k \times \{0, 1\}^k \rightarrow \mathbf{R}^k_{\star}$, where the $j$th component of $x\ostar\omega$ is defined by
\begin{align*}
    (x\ostar\omega)_j \coloneqq \begin{cases}
        x_j \quad &\text{if } \omega_j = 1\\
        \star &\text{if } \omega_j = 0,
    \end{cases}
\end{align*}
for $j \in [k]$.  We use $S \pm T$ between multisets $S$ and $T$ to, respectively, denote multiset union and difference.  For a multiset $S$, we use $\E_{x \sim S}[f(x)]$ to denote the expectation of $f$ with respect to the empirical distribution of $S$.

For $n \in \mathbf{N}$, we use the notation $f(n) \lesssim g(n)$ (or $f(n) = O(g(n))$) to mean $\lvert f(n) \rvert \leq C \lvert g(n) \rvert$ for some universal, positive constant $C$.  Analogously, we use the notation $f(n) \gtrsim g(n)$ (or $f(n) = \Omega(g(n))$) to mean $\lvert f(n) \rvert \geq c \lvert g(n) \rvert$ for some universal, positive constant $c$.  If $f(n) \lesssim g(n)$ and $f(n) \gtrsim g(n)$, we denote this relationship by $f(n) \asymp g(n)$ (or $f(n) = \Theta(g(n))$).
We use $\widetilde{O}, \widetilde{\Omega}$, and $\widetilde{\Theta}$ to hide terms that are poly-logarithmic in the parameters.

We let $\mathsf{N}(\mu, \Sigma)$ denote the Gaussian distribution with mean $\mu$ and covariance $\Sigma$.
We let $\phi_d(\cdot; \mu, \Sigma)$ to denote its density.
We omit the subscript from the $\phi_d$ when $d=1$.
We often omit the parameter $\Sigma$ from $\phi_d$ when it is the identity matrix and omit $\mu$ when it is the zero vector,
so that $\phi_d(\cdot)$ denotes the density of the standard $d$-dimensional Gaussian.
We let $\chi^2_d$ denote the chi-squared distribution with $d$ degrees of freedom.

\section{Background} \label{sec:background}
In this section, we provide background on some of the tools we use throughout the paper.

\subsection{Sum-of-squares algorithms}
We briefly review sum-of-squares proofs and pseudo-expectations, referring to \cite{BarSte16-sos-notes,FleKP19-sos} for a more complete overview.
Throughout, we will refer to a polynomial as a sum of squares if it is equal to a sum of squared polynomials. We use $\R[X]$ to denote the set of polynomials over the indeterminates/variable $X = (X_1,\dots,X_N)$ and $\R[X]_{\leq t}$ to denote those polynomials with degree at most $t$. For $p \in \R[X]$, we use $\norm{p}_2$ to denote the Euclidean norm of the associated coefficient vector (in the monomial basis).

Consider a system of polynomial inequalities $\mathcal{A} = \{q_i(X) \geq 0\}_{i=1}^m$ over the $N$-dimensional indeterminates $X = (X_1,\dots,X_N)$.

\begin{definition}
    (Sum-of-squares proofs)  Let $p,p' \in \R[X]$. The inequality $p(X) \geq p'(X)$ has a \emph{sum-of-squares proof} given $\mathcal{A}$ if there exist sum of squares polynomials $p_S \in \R[X]$ for $S \subset [m]$ satisfying
\begin{align*}
    p(X) = p'(X) + \sum_{S \subset [m]} p_S(X) \Pi_{i \in S} q_i(X),
\end{align*}
and we say the proof has degree $t$ if each summand has degree at most $t$; 
we denote this by $\mathcal{A} \sos{X}{t} p \geq p'$ (we may omit the variable when it is clear from the context).
\end{definition}

\begin{definition}
    (Pseudo-expectations) A degree-$t$ pseudo-expectation $\pE : \R[X]_{\leq t} \to \R$ is a linear map satisfying $\pE[1] = 1$ and $\pE[p(X)^2] \geq 0$ for any $p \in \R[X]_{\leq t/2}$. 
    Given a set of polynomial inequalities $\mc A$, 
    we say a degree-$t$ pseudo-expectation $\pE$ satisfies $\mc{A}$ if $\pE[p(X)] \geq 0$ for all $p \in \R[X]_{\leq t}$ such that $p(X) = s(X)^2 \Pi_{i \in S} q_i(X)$ for some $S \subset [m]$ and $s \in \R[X]$. 
    Under the same conditions, we say $\pE$ approximately satisfies $\mc{A}$ if
$\pE[p(X)] \geq -2^{-N^{\Omega(t)}} \norm{s}_2 \Pi_{i \in S}\norm{q_i}_2$.

\end{definition}

Under mild conditions on the bit complexity of the constraint set and boundedness of the feasible region, we can efficiently find an approximately satisfying pseudoexpectation. Moreover, this pseudoexpectation will approximately satisfy all polynomial inequalities that have a sum of squares proof (of sufficiently small degree), again under mild conditions on the bit complexity of the proof. These conditions will hold for all the problems we consider in this paper with negligible approximation error, so for simplicity we take the following assumptions as facts throughout.

\begin{assumption}
    We can compute a degree-$d$ pseudo-expectation satisfying $\mc{A}$ (or refute its satisfiability) in time $(m+N)^{O(d)}$.
\end{assumption}

\begin{assumption}
    Let $\mc{A} \sos{}{t'} p \geq 0$ for $p \in \R[X]$ and let $\pE$ be a degree-$t$ pseudo-expectation that satisfies $\mc{A}$. Let $h \in \R[X]$ be a sum of squares with $\mathrm{deg}(h) + t' \leq t$. Then $\pE[h \cdot p] \geq 0$.
\end{assumption}

Finally, we will use $x \in (1 \pm y) z$ as shorthand to denote the intersection of the constraints $x \leq (1+y)z$ and $x \geq (1-y) z$. We will sometimes consider sum-of-squares proofs involving matrix-valued variables. For a matrix-valued polynomial $A$, we use the constraint $A \preceq B$ to denote $A = B - CC^T$, where $C$ is an implicit dummy matrix-valued indeterminate  (and likewise for $A \succeq B$).

\paragraph{Quantifier elimination.} 
Let $Z = (Z_1,\dots,Z_{d})$ be indeterminates.
Let $\mc A'$ be a set of $m'$ polynomial inequalities over $Z$.
We will frequently work with a set of constraints $\mc A$ over indeterminates $X$ that include the following constraint on the variable $X$:
\begin{align}
\label{eq:axiom-quantifier}
  \mc{A}'  \sos{Z}{t} p(X,Z) \geq 0,
\end{align}
 which may be read as: there exists a sum-of-squares proof under the axioms $\mc A'$ in the indeterminates $Z$ that $p(X,Z) \geq 0$.
 For example, we will routinely use the constraint of the following form: $\{\sum_i Z_i^2=1\} \sos{Z}{t} \sum_{j=1}^n \langle Z, y-X\rangle^{2t} \leq B$.
 Even though \Cref{eq:axiom-quantifier} does not seem to be representable as $\mc A = \{q_i(X) \geq 0\}_{i=0}^m$ for a small $m$, it turns out that it can be compactly represented.
 In particular, if $p$ is a polynomial of degree at most $t$ and $Z$ has dimension $d$, then we can encode it as 
$ \cA$ over indeterminates $X$ and $X'= (X'_1,\dots,X'_{N'})$,
where $|N'| \lesssim |O(N  + d + m')|^{O(t)}$
and $|\cA| = (O(N' + m+d))^{O(t)}$.
We refer the reader to \cite{KotSte17} for further details (see also \cite[Appendix A.2]{DiaKKPP22-colt}).

\subsection{SQ lower bounds}
\label{sec:prelim-sq-lower-bound}
In this section, we give a brief overview of the statistical query (SQ) framework~\cite{Kearns98,FelGRVX17}. In the sequel, we shall show that our algorithms for mean and covariance estimation is qualitatively optimal in the SQ framework.
While our focus is on SQ algorithms, our structural results also rule out polynomial-time algorithms from other popular families of algorithms: sum-of-squares hierarchies~\cite{DiaKPP24-sos}, low-degree polynomial tests~\cite{BreBHLS21}, and polynomial threshold functions~\cite{DiaKLP25-ptf}.

An SQ algorithm does not take in samples from a distribution $Q$, instead, it interacts with the underlying distribution $Q$ through an oracle of the following form:
\begin{definition}[STAT Oracle] 
	Let $P$ be a distribution on the domain $\cX$.
    A statistical query is a bounded (measurable)
	function $f : \cX \to [-1, 1]$.
	For a tolerance $\SQtolerance \in (0,1)$, the $\STAT_{P,\SQtolerance}$ oracle 
  responds to the query $f$ with a value $v = \STAT_{P,\SQtolerance}(f)$
	such that $|v - \E_{X\sim P} [f(X)]| \leq \SQtolerance$.
  We call $\SQtolerance$ the tolerance of the SQ oracle.
\end{definition}
Many popular algorithms for statistical estimation use samples only to approximate $\E_P[f_1(X)],$ $\dots$, $\E_P[f_m(X)]$ for a sequence of (adaptively chosen) functions $f_i$'s. 
An SQ algorithm could directly obtain these approximations using the $\STAT_{P,\SQtolerance}$ oracle above, up to error $\SQtolerance$.
Observe that $\SQtolerance$ here corresponds to the sampling error, and thus $\poly(1/\SQtolerance)$ is interpreted as the ``effective sample complexity'' of the SQ algorithm.

Our focus will be on showing the hardness of testing problems; Standard arguments imply that the corresponding estimation task is no easier (either statistically or computationally).
We will consider a testing problem of the following kind:
\begin{definition}[Generic Testing Problem]
\label{def:generic-testing-problem}
    Let $P_{\mathrm{null}}$ be a distribution and let $\cA$ be a set of distributions.
    Suppose the tester has access (either through i.i.d.\ samples $(X_1,\dots,X_n)\sim {Q}^{\otimes n}$ or an SQ oracle $\mathrm{STAT}_{Q,\SQtolerance}$) to the distribution $Q$. 
    Consider the following testing problem:
    \[
    H_0: Q = P_{\mathrm{null}} \quad \text{ vs. } \quad H_1: Q \in \mathcal{A}\,.
    \]
    \end{definition}
    We say an SQ algorithm solves the testing problem above with $q$ queries and tolerance $\SQtolerance$ if there exists an SQ algorithm that interacts with an $\mathrm{STAT}_{Q,\SQtolerance}$ oracle by making at most $q$ adaptive queries and outputs a decision $\widehat{\phi}$ such that, with probability at least $2/3$, $\widehat{\phi} = \mathrm{1}_{Q \in \cA}$.

An SQ hardness result is parameterized by two parameters, a query parameter $q_0$ and a tolerance parameter $\SQtolerance_0$, and is of the following flavor: any SQ algorithm that makes less than $q_0$ queries and successfully solves the inference task of interest must use at least a single query with tolerance less than $\SQtolerance_0$.
This is interpreted as an information-computation gap stating that every successful SQ algorithm must use either $\Omega(q_0)$ ``runtime'' or $\poly(1/\SQtolerance_0)$ ``samples''.
We refer the reader to~\cite[Chapter 8]{DiaKan22-book} for further details.

The influential work of \cite{DiaKS17}
proved SQ hardness results for 
a generic hypothesis testing problem, known as Non-Gaussian Component Analysis (NGCA).
NGCA asks us to distinguish whether $Q = \mathsf{N}(0,I)$ or whether $Q$ has a non-Gaussian component distribution $A$ in a hidden direction, defined below.

\begin{definition}[High-Dimensional Hidden Direction Distribution]\label{def:hidden-direction-dist}
    Let $A$  be a univariate distribution on $\R$; we use $a(x)$ for the density of $A$ at $x$.
    For a unit vector $v$, we denote by $P_{A,v}$ the distribution with the density $P_{A,v}(x) := a(v^\top x) \phi_{v^\perp}(x)$, where $\phi_{v^\perp}(x) := \exp\left(-\|x - (v^\top x)v\|_2^2/2\right)/(2\pi)^{(d-1)/2}$, i.e., the distribution that coincides with $A$ on the direction $v$ and is an independent standard Gaussian in every orthogonal direction.  
\end{definition}

\begin{definition}[NGCA]
    \label{def:ngca}
    Let $A$ be a given univariate distribution.
    Consider the distribution testing problem (see \Cref{def:generic-testing-problem}) given access to a distribution $Q$ over $\R^d$:
\[
    H_0: Q = \mathsf N(0,I_d) \quad \text{ vs. } \quad H_1: Q \in  \{P_{A,v} \mid v \in \bS^{d-1}\}\,.
    \]
\end{definition}
\cite{DiaKS17} established the following SQ lower bounds, showing that if $A$ matches $m$ moments with $\mathsf{N}(0,1)$, then any SQ algorithm either needs at least exponentially many queries (``runtime'') or inverse tolerance (``samples'')  at least $d^{\Omega(m)}$. 
\begin{lemma}[SQ Lower Bounds for NGCA~\cite{DiaKS17}]
\label{lem:ngca-lower-bound}
Suppose $A$ matches $m$ moments with $\mathsf{N}(0,1)$ and $\chi^2(A,\mathsf{N}(0,1)) < \infty$.
Then any SQ algorithm that solves the testing problem in \Cref{def:ngca} must use either
\begin{itemize}
    \item number of queries at least $\frac{2^{\Omega(d^{\Omega(1)})}}{d^{O(m)}}$, or
    \item tolerance less than $\SQtolerance \lesssim \frac{\sqrt{\chi^2(A,\mathsf{N}(0,1))}}{\dimension^{\Omega(m)}}$.
\end{itemize}
In particular, if $m \gtrsim 1$,  $d \gtrsim (m \log d)^{\Omega(1)}$, $\chi^2(A,\mathsf{N}(0,1)) \lesssim m$, then
any successful SQ algorithm either uses $2^{d^{\Omega(1)}}$ many queries or tolerance at most $d^{-\Omega(m)}$.
\end{lemma}
Computational lower bounds for the NGCA problem under the same moment-matching condition on $A$ have also been established for sum-of-squares hierarchies~\cite{DiaKPP24-sos}, low-degree polynomial tests~\cite{BreBHLS21}, and polynomial threshold functions~\cite{DiaKLP25-ptf}.
Since our computational lower bounds for mean and covariance estimation under realizable contamination are proved using the NGCA framework, we immediately obtain lower bounds for these families of algorithms as well.

\subsection{Minimax lower bounds} 
Consider a parameter space $\Theta$ and an observation space $\mathcal{X}$.  Let $L: \Theta \times \Theta \rightarrow \mathbb{R}_{+}$ denote a metric on the parameter space.  For each $\theta \in \Theta$, let $P_\theta \in \mathcal{P}(\mathcal{X})$ denote a probability distribution on the space of observations $\mathcal{X}$.  We define the $1-\delta$ quantile as
\[
\mathrm{Quantile}\bigl(1-\delta, P_{\theta}, L(\widehat{\theta}, \theta)\bigr) := \inf\big\{r \geq 0\ \big|\ P_{\theta}\bigl(L(\widehat{\theta}, \theta) \leq r\bigr) \geq 1-\delta \big\}.
\]
Our information-theoretic lower bounds will be stated in terms of the minimax $(1-\delta)$ quantile, defined as
\begin{align}\label{def:minimax-quantile}
\mathcal{M}(\delta, \mathcal{P}_{\Theta}, L) := \inf_{\widehat{\theta} \in \widehat{\Theta}} \sup_{\theta \in \Theta} \sup_{P_{\theta} \in \mathcal{P}_{\theta}}\; \mathrm{Quantile}\bigl(1-\delta, P_{\theta}, L(\widehat{\theta}, \theta)\bigr).
\end{align}
We will commonly consider the setting in which $n$ samples are drawn i.i.d. from an underlying distribution and use the shorthand $\cM_n(\delta, \cP_{\Theta}, L)$  accordingly.

Our lower bounds will use the following variant of Fano's inequality, the main statement and proof of which can be found in several standard statistics textbooks~\cite{wainwright2019high,duchi2025info,tsybakov2009nonparametric}.  
\begin{lemma}[Fano's inequality]\label{lem:fano}
Suppose that $H \in \mathcal{P}(\mathcal{X})$ and $\theta_1, \ldots, \theta_M \subseteq \Theta$ satisfy
\begin{enumerate}
    \item (Loss separation) For all $\ell \neq k \in [M]$ and all $\theta \in \Theta$, $\max\{L(\theta, \theta_k), L(\theta, \theta_\ell)\} \geq \gamma$;
    \item (Bounded KL) $\frac{1}{M} \sum_{\ell \in [M]}\, \mathrm{KL}(P_{\theta_{\ell}}, H) + \log(2 - M^{-1}) \leq \frac{1}{2} \log(M)$.
\end{enumerate}
Then, for all $\delta \leq 1/2$, it holds that $\mathcal{M}(\delta, \mathcal{P}_{\Theta}, L) \geq \gamma$.
\end{lemma}

\section{Information-theoretic limits of mean and covariance estimation} \label{sec:it}

\subsection{Algorithms via variational principles} \label{sec:alg-variational}
Our information-theoretically optimal algorithms use a familiar technique from the robust statistics literature in which optimal univariate estimators are combined using a variational principle to obtain an optimal multivariate estimator.  This idea dates back at least to the Tukey median and the notion of Tukey depth~\cite{tukey1975mathematics} and has been deployed in problems beyond mean estimation~\cite{gao2020robust} and using a variety of variational principles, such as the PAC-Bayes method~\cite{minasyan2025statistically}.  We next briefly describe the variational principle we use for mean estimation and for covariance estimation in turn.

\paragraph{Mean estimation.} Let $\theta, \theta_{\star} \in \R^d$ and note that the $\ell_2$ norm error admits the variational characterization
\begin{align} \label{eq:variational-mean}
\| \theta - \theta_{\star} \|_2 = \sup_{v \in \mathbf{S}^{d-1}}\; \langle v, \theta - \theta_{\star} \rangle. 
\end{align}
Suppose that for each direction $v \in \mathbf{S}^{d-1}$,
we had an estimator $\widehat{\theta}_v$ of the mean $\langle v, \theta_{\star}\rangle$.
In this case, we could define an estimator $\widehat{\theta}$ to be any element
\[
\widehat{\theta} \in \argmin_{\theta \in \mathbf{R}^d}
  \sup_{v \in \mathbf{S}^{d-1}}\;
    \bigl \lvert \langle v, \theta \rangle - \widehat{\theta}_v \bigr \rvert.
\]
We then have the oracle inequality
\[
\sup_{v \in \mathbf{S}^{d-1}}\; \bigl \lvert \langle v, \widehat{\theta} \rangle - \widehat{\theta}_v \bigr \rvert \leq \sup_{v \in \mathbf{S}^{d-1}}\; \bigl \lvert \langle v, \theta_{\star} \rangle - \widehat{\theta}_v \bigr \rvert,
\]
from which we combine the characterization~\eqref{eq:variational-mean} with the triangle inequality to obtain
\begin{align*}
\bigl\| \widehat{\theta} - \theta_{\star} \bigr\|_2
= \sup_{v \in \mathbf{S}^{d-1}}\; \bigl\langle v, \widehat{\theta} - \theta_{\star} \bigr\rangle
&\leq \sup_{v \in \mathbf{S}^{d-1}}\; \bigl\lvert \bigl\langle v, \widehat{\theta}\bigr\rangle  - \widehat{\theta}_v \bigr\rvert + \sup_{v \in \mathbf{S}^{d-1}}\; \bigl \lvert \langle v, \theta_{\star}\rangle  - \widehat{\theta}_v \bigr \rvert \\
&\leq 2\sup_{v \in \mathbf{S}^{d-1}}\; \bigl \lvert \langle v, \theta_{\star}\rangle  - \widehat{\theta}_v \bigr \rvert.
\end{align*}
It thus follows that if each of the univariate estimators $\widehat{\theta}_v$ are information-theoretically optimal, we may be able to combine these estimators into a multivariate estimator with similar statistical guarantees.  In the setting of mean estimation, we note that the optimal univariate estimator is due to forthcoming work~\cite{ma2026adaptive}; for the reader's convenience, we provide an alternative method which achieves the same rate. \hfill $\clubsuit$

\paragraph{Covariance estimation.}
For simplicity, we consider operator norm in this paragraph, although our main guarantees follow a slightly different path as we are interested in relative error.  

Following similar steps, let $\Sigma, \Sigma_{\star} \in \mathcal{C}^d_{++}$ and note that the operator norm admits the variational representation
\begin{align*}%
\| \Sigma - \Sigma_{\star} \|_{\mathrm{op}} = \sup_{v \in \mathbf{S}^{d-1}} \bigl \langle v, \bigl(\Sigma - \Sigma_{\star} \bigr) v \bigr\rangle.
\end{align*}
As before, suppose we had an estimator $\widehat{\sigma}_v^2$ of the variance $v^{\top} \Sigma_{\star} v$ in each direction $v \in \mathbf{S}^{d-1}$. We then define an estimator $\widehat{\Sigma}$ as any element
\[
\widehat{\Sigma} \in \argmin_{\Sigma \in \mathcal{C}_{++}^{d}} \sup_{v \in \mathbf{S}^{d-1}} \bigl \lvert \langle v, \Sigma v \rangle - \widehat{\sigma}_v^2 \bigr \rvert.
\]
Combining the preceding two displays with the triangle inequality yields
\begin{align*}
    \bigl \| \widehat{\Sigma} - \Sigma_{\star} \bigr \|_{\mathrm{op}} = \sup_{v \in \mathbf{S}^{d-1}} \bigl \langle v, \bigl(\widehat{\Sigma} - \Sigma_{\star} \bigr) v \bigr\rangle &\leq \sup_{v \in \mathbf{S}^{d-1}} \bigl \lvert \bigl \langle v, \widehat{\Sigma}v\bigr \rangle - \widehat{\sigma}_v^2 \bigr \rvert+ \sup_{v \in \mathbf{S}^{d-1}} \bigl \lvert \bigl \langle v, \Sigma_{\star} v \bigr\rangle - \widehat{\sigma}_v^2 \bigr\rvert\\
    &\leq 2 \cdot \sup_{v \in \mathbf{S}^{d-1}} \bigl \lvert \bigl \langle v, \Sigma_{\star} v \bigr\rangle - \widehat{\sigma}_v^2 \bigr\rvert.
\end{align*}
Hence, as in the mean estimation setting, this suggests that it suffices to find optimal univariate variance estimators in each direction and that these can be converted into covariance estimators.  In the sequel, we show that a variant of the minimum Kolmogorov distance estimator introduced by~\cite{ma2024estimation} yields a suitable univariate estimator in each direction.  \hfill $\clubsuit$

We emphasize that as written, these estimators cannot be computed in finite time.  A straightforward solution is to approximate the unit sphere by a suitable $\epsilon$-net of the unit sphere and compute univariate estimators on each element of the $\epsilon$-net.  While this strategy yields information-theoretically optimal estimators, we note that the size of these $\epsilon$-nets scales exponentially in the dimension and hence, in general, yields computationally-inefficient estimators.

\subsection{Mean estimation}

We first consider information-theoretic rates for Gaussian mean estimation.
We take the parameter space and loss to be
\begin{align*}
  \cP_{\rm mean}(\sigma^2, \epsilon) = \bigl\{\cR(\mathsf{N}(\theta, \sigma^2 I_d), \epsilon) \mid \theta \in \bR^d\bigr\}
  \qquad\textrm{and}\qquad
  L(\theta, \theta_\star) = \| \theta - \theta_\star \|_2.
\end{align*}
In the theorem below, we give tight minimax rates on this problem.
We provide the proof of the upper bound in \cref{pf:mean-est-it-ub} and of the lower bound in \cref{pf:mean-est-it-lb}.  As mentioned in the previous subsection, the upper bound is a corollary of the optimal univariate estimator of~\cite{ma2026adaptive}; our main contribution is a matching lower bound in the high-dimensional and high-confidence setting.

\begin{theorem}\label{thm:mean-est-it-ub-lb}
Let $\sigma^2 > 0$, $\epsilon < 1$, $\delta \leq 0.5$, $n\in\mathbf{N}$ such that
\[
  n \gtrsim \frac{d + \log(\frac{1}{\delta})}{1-\epsilon}.
\]
The minimax $1-\delta$ quantile $\cM_n(\delta, \cP_{\rm mean}(\sigma^2, \epsilon),L)$~\eqref{def:minimax-quantile} satisfies
\[
  \cM_n(\delta, \cP_{\rm mean}(\sigma^2, \epsilon),L)
  \asymp \frac{\sigma \log(\frac{1}{1-\epsilon})}{\sqrt{\log \Bigl(1 + \frac{\epsilon^2}{1-\epsilon} \cdot \frac{n}{d+\log(\frac{1}{\delta})}\Bigr)}}.
\]
\end{theorem}

A few remarks are in order.
First, let us simplify the expression of the minimax rate in the regime of interest where $\delta$ is constant and $\epsilon \leq 1-c$ for some small constant $c$.
In this setting,
 the rate simplifies to  
\begin{align*}
  \cM_n(\delta, \cP_{\rm mean}(\sigma^2, \epsilon),L)
  \asymp \frac{\sigma \epsilon}{\sqrt{\log\left(1 + \epsilon^2 n/d\right)}}.
\end{align*}
Note that if $\epsilon \lesssim \sqrt{d/n}$, this recovers the parametric rate of estimation $\sigma \sqrt{d/n}$.  
On the other hand, if the contamination fraction $\epsilon$ is at least a constant, the bound further simplifies to 
    \begin{align*}
    \cM_n(\delta, \cP_{\rm mean}(\sigma^2, \epsilon),L) \asymp  \frac{\sigma}{\sqrt{\log(n/d)}}\,.
\end{align*}
In turn, this implies that in terms of sample complexity, in order to reach error $\sigma\rho > 0$, it is both necessary and sufficient to take $n \asymp d e^{\Theta(1/\rho^2)}$ many samples.  

Next, let us compare this rate to similar results in the literature, starting with~\cite[Theorem 8]{ma2024estimation}, which considers the same setting.  In particular, re-arranging our sample size condition, we require that 
\[
\delta \geq \exp\bigl\{-c \cdot \bigl(n(1 -\epsilon) - d\bigr)\bigr\}.
\]
By contrast,~\cite{ma2024estimation} impose the condition that
\[
\delta \geq \exp\Bigl\{- c \Bigl(\frac{\{n(1-\epsilon)\}^{31/36}}{\log(n(1-\epsilon))} - d\Bigr)\Bigr\},
\]
which additionally imposes the sub-optimal sample complexity requirement $n \gtrsim d^{36/31}/(1-\epsilon)$.  As we show, the correct sample complexity scales linearly in the dimension $d$.  Let us next briefly isolate the dependence on the failure probability $\delta$.  In particular, taking $d = 1$ and $\sigma, \epsilon$ to be constants, we see that  
    \begin{align*}
    \cM_n(\delta, \cP_{\rm mean}(\sigma^2, \epsilon),L) \asymp  \frac{1}{\sqrt{\log(n) - \log\log(1/\delta)}}\,,
\end{align*}
so that the dependence on the failure probability $\delta$ is doubly logarithmic in the denominator.  

    Finally, we comment briefly on the proof technique.  As previously mentioned, our main contribution in \Cref{thm:mean-est-it-ub-lb} is the lower bound and we restrict our commentary to this result.  We note that it is typical in robust estimation problems to obtain an error which decomposes as the sum of two terms: One consisting of the ``clean" term corresponding to $\epsilon = 0$ and one capturing the dependence on $\epsilon$.  Typically, the clean term interacts with the dimension $d$, the sample size $n$, and the confidence $\delta$, whereas the second term is a function solely of the contamination fraction $\epsilon$~\cite{DiaKKLMS16-focs}.  From the perspective of lower bounds, these two terms can typically by captured by (i.) an application of Fano's inequality or Assouad's method for the ``clean" term and (ii.) an application of Le Cam's two point method for the contamination term.  By contrast, the realizable contamination model considered here has a contamination term which depends nontrivially on the dimension $d$.  In order to capture this dependence,
    we build on the proof technique of the forthcoming work of \cite{DiaGKPX26-adaptive} (who applied it in the context of linear regression):
    we construct a family of exponentially many hard distributions (using the high-dimensional hidden direction distributions of Definition~\ref{def:hidden-direction-dist}) and conclude via Fano's inequality (Lemma~\ref{lem:fano}).

\subsection{Covariance estimation}
We now consider information-theoretic rates for Gaussian covariance estimation in relative operator norm error.
Specifically, we take the parameter space and loss to be
\begin{align*}
  \cP_{\rm cov}(\epsilon) = \bigl\{\cR(\mathsf{N}(0, \Sigma), \epsilon) \mid \Sigma \in \mathcal{C}_{++}^d\bigr\}
  \qquad\textrm{and}\qquad
  L(\Sigma, \Sigma_\star)
  = \bigl \| \Sigma_\star^{-\frac{1}{2}}\Sigma\Sigma_\star^{-\frac{1}{2}} - I_d \bigr \|_{\rm op}.
\end{align*}
In the theorem below, we give upper and lower bounds on the minimax rates.
We provide the proof of the upper bound in \cref{sec:proof-thm-cov-est-it-ub} and the lower bound in \cref{pf:cov-est-it-lb}, which follow similar techniques as the proof of \cref{thm:mean-est-it-ub-lb}.

\begin{theorem}\label{thm:cov-est-it-ub-lb}
Let $\epsilon < 1$, $\delta \leq 0.5$, and $n\in\mathbf{N}$ such that
\[
  n \gtrsim \frac{d (1 + \log(\frac{1}{1-\epsilon})) + \log(\frac{1}{\delta})}{1-\epsilon}.
\]
The minimax $1-\delta$ quantile~\eqref{def:minimax-quantile} $\cM_n\bigl(\delta, \cP_{\rm cov}(\eps), L\bigr)$ satisfies
\[
  \sqrt{\frac{d + \log(\frac{1}{\delta})}{n(1-\eps)}} + \frac{\log(\frac{1}{1-\epsilon})}{\log\left(1 + \frac{\epsilon n}{d+\log(\frac{1}{\delta})}\right)}
  \lesssim \cM_n\bigl(\delta, \cP_{\rm cov}(\eps), L\bigr)
  \lesssim \frac{\log(\frac{1}{1-\epsilon})}{\log\Big(1 + \sqrt{\frac{\epsilon^2}{1-\epsilon} \cdot\frac{n}{d (1+\log(\frac{1}{1-\epsilon})) + \log(\frac{1}{\delta})}}\Big)}.
\]
\end{theorem}

We analyze the upper and lower bounds of Theorem~\ref{thm:cov-est-it-ub-lb} in a few different regimes.
In order to ease the notation, we let $\alpha = \sqrt{\frac{d+\log(\frac{1}{\delta})}{n(1-\epsilon)}}$
and $\beta = \log(\frac{1}{1-\epsilon})$.
In the low-to-constant contamination regime of $\eps \leq \frac{1}{2}$,
we have $\beta \asymp \epsilon$ and so
the bounds of Theorem~\ref{thm:cov-est-it-ub-lb} simplify to
\begin{align*}
  \alpha + \frac{\epsilon}{\log\left(1 + \epsilon/\alpha^2\right)}
  \lesssim \cM_n\bigl(\delta, \cP_{\rm cov}(\eps), L\bigr)
  \lesssim \frac{\epsilon}{\log(1 + \epsilon/\alpha)}.
\end{align*}
When $\epsilon \lesssim \alpha$, both the upper and lower bounds yield $\cM_n\bigl(\delta, \cP_{\rm cov}(\eps), L\bigr) \asymp \alpha$, which we recognize as the parametric rate.
When $\epsilon \gtrsim \alpha^{1-c}$ for some small constant $c$ (which contains the regime when $\epsilon$ is at least a small constant), both bounds agree and yield the rate
$\cM_n\bigl(\delta, \cP_{\rm cov}(\eps), L\bigr) \asymp \frac{\epsilon}{\log(1/\alpha)}$.
In the intermediate range between these two settings, where $\epsilon = \alpha^{1-\gamma}$ for $\gamma = o(1)$, our bounds read as
\begin{align*}
  \frac{\epsilon}{\log(1/\alpha)}
  \lesssim \cM_n\bigl(\delta, \cP_{\rm cov}(\eps), L\bigr)
  \lesssim \frac{\epsilon}{\gamma \log(1/\alpha)},
\end{align*}
representing a gap on the order of $\frac{1}{\gamma}$.
For example, when $\epsilon = \alpha\log(1/\alpha)$ and $\delta \gtrsim e^{-d}$, we observe a gap in our bounds on the order of $\frac{\log(1/\alpha)}{\log\log(1/\alpha)} \asymp \frac{\log(n/d)}{\log\log(n/d)}$.

Now we consider the large contamination regime $\eps \geq \frac{1}{2}$.
The upper and lower bounds of Theorem~\ref{thm:cov-est-it-ub-lb} simplify to
\begin{align*}
  \alpha + \frac{\beta}{\beta + \log(1/\alpha)}
  \lesssim \cM_n\bigl(\delta, \cP_{\rm cov}(\eps), L\bigr)
  \lesssim \frac{\beta}{\beta -\log(1+\beta) + \log(1/\alpha)}.
\end{align*}
Because $\beta$ is bounded away from zero in this regime, the upper and lower bounds simplify to
$\cM_n\bigl(\delta, \cP_{\rm cov}(\eps), L\bigr) \asymp \frac{\beta}{\beta + \log(1/\alpha)}$.

Finally, note that when the mechanism is nearly missing completely at random setting (i.e. when $\epsilon \lesssim \alpha$), the rates of estimation for mean estimation in $\ell_2$ norm and covariance estimation in operator norm coincide.  On the other hand, when $\epsilon > \alpha$, the rate of estimation for covariance estimation is quadratically faster.  In particular, if we fix the parameters $\epsilon$ and $\delta$ as constants, we see that
\[
\cM_n\bigl(\delta, \cP_{\rm cov}(\eps), L\bigr) \asymp \frac{1}{\log(n/d)}, \quad \text{ whereas } \quad \cM_n(\delta, \cP_{\rm mean}\bigl(1, \epsilon),L\bigr) \asymp \frac{1}{\sqrt{\log(n/d)}}.
\]
In terms of sample complexity, this implies that in order to achieve covariance norm error $\Theta(\rho)$, we require $n \gtrsim de^{1/\rho}$ many samples as opposed to $de^{1/\rho^2}$ samples for mean estimation.

\section{Computationally-efficient mean and covariance estimation} \label{sec:comp-mean-cov}
In the previous section, we established (nearly) tight statistical rates for mean and covariance estimation, but the algorithms we employed required exponential time computationally.  In this section, we study polynomial time algorithms and rates of estimation for polynomial time algorithms.  In Section~\ref{sec:efficient testing}, we preview the key idea underlying our approach in the simpler setting of hypothesis testing.  Then, in Section~\ref{sec:sos-mean-est} we provide an efficient algorithm for mean estimation and demonstrate that its sample complexity is optimal among various well-studied algorithmic classes.  Finally, in Section~\ref{sec:sos-cov-est}, we provide a parallel story in covariance estimation.

\subsection{Warm-up: Efficient testing via polynomials} \label{sec:efficient testing}
In order to build our intuition, let us begin with the univariate testing problem.  In particular, we will try to test 
\[
H_0:\; Q \in \mathcal{R}\bigl(\mathsf{N}(0, 1), \epsilon\bigr) \quad \text{ vs. } \quad H_1: \; Q \in \mathcal{R}\bigl(\mathsf{N}(\rho, 1), \epsilon\bigr).
\]
for $\rho \in \R$.
Note that the hardest distributions to test put the same amount of mass on $\{\star\}$, so it suffices to test between the conditional distributions $Q_{\R} := \mathsf{Law}(X \mid X \in \R)$.  The minimum distance estimators studied in Section~\ref{sec:it} use fine-grained information from the distribution functions under each hypothesis.  Here, we note that, at the population level, thresholding a simple polynomial yields a consistent test.  We will use the notation
\[
\tau := \frac{\epsilon}{1 - \epsilon},
\]
throughout the derivation.  Indeed, from the discussion following Lemma~\ref{lem:all-or-nothing-characterization}, we see that 
\begin{align} \label{ineq:sandwich-Q_R}
\frac{1}{1 + \tau} = 1 - \epsilon \leq \frac{\mathrm{d}Q_{\R}}{\mathrm{d}P}(z) \leq 1 + \frac{\epsilon}{1-\epsilon} = 1 + \tau \quad \text{ for all } z \in \R.
\end{align}
At the population level, consider the family of testing functions, indexed by $k \in \mathbf{N}$, $\psi^{(k)}: \R \rightarrow \R$ defined as 
\[
\psi^{(k)}(X) = X^{2k}.  
\]
It is a straightforward consequence of the sandwich relation~\eqref{ineq:sandwich-Q_R} that
\[
\Bigl(\frac{1}{1 + \tau}\Bigr) \cdot \mathbb{E}_P\bigl[\psi^{(k)}(X)\bigr] \leq \mathbb{E}_{Q_{\R}}\bigl[\psi^{(k)}(X)\bigr] \leq (1 + \tau) \cdot  \mathbb{E}_P\bigl[\psi^{(k)}(X)\bigr].
\]
Consequently, under $H_0$, we see that 
\[
\mathbb{E}_{Q_{\R}}\bigl[\psi^{(k)}(X)\bigr] \leq (1 + \tau) \mathbb{E}_{X \sim \mathsf{N}(0, 1)}\bigl[\psi^{(k)}(X)\bigr] = (1 + \tau) \mathbb{E}\bigl[G^{2k}\bigr] = (1 + \tau) (2k-1)!!,
\]
where we have used the notation $G \sim \mathsf{N}(0, 1)$.  On the other hand, under $H_1$,
\begin{align*}
\mathbb{E}_{Q_{\R}}\bigl[\psi^{(k)}(X)\bigr] &\geq \Bigl(\frac{1}{1 + \tau}\Bigr) \mathbb{E}_{X \sim \mathsf{N}(\gamma, 1)}\bigl[\psi^{(k)}(X)\bigr] = \Bigl(\frac{1}{1 + \tau}\Bigr) \mathbb{E}\bigl[(G + \rho)^{2k}\bigr] \\
&= \Bigl(\frac{1}{1 + \tau}\Bigr) \sum_{\ell=0}^{2k}\binom{2k}{\ell} \rho^\ell \mathbb{E}\bigl[G^{2k-\ell}\bigr] \\
&= \Bigl(\frac{1}{1 + \tau}\Bigr) \sum_{\ell=0}^{k}\binom{2k}{2\ell} \rho^{2\ell} \mathbb{E}\bigl[G^{2k-2\ell}\bigr] \\
&\geq \Bigl(\frac{1}{1 + \tau}\Bigr) \cdot \Bigl\{\mathbb{E}[G^{2k}] + \binom{2k}{2} \rho^2 \mathbb{E}\bigl[G^{2k-2}\bigr]\Bigr\} \\
&= \Bigl(\frac{1}{1 + \tau}\Bigr) \cdot \bigl\{\mathbb{E}\bigl[G^{2k}\bigr]+ k \rho^2 (2k - 1)!!\bigr\},
\end{align*}
where in the final display we used the fact that $\mathbb{E}[G^{2k-2}] = (2k-3)!!$.  Re-arranging, we deduce that any separation
\begin{align*}
\rho \geq \sqrt{\frac{\{(1 + \tau)^2 - 1\} \cdot \mathbb{E}[G^{2k}]}{k (2k-1)!!}} = \sqrt{\frac{\tau^2 + 2\tau}{k}},
\end{align*}
suffices to test between $H_0$ and $H_1$.  In particular, the larger we take $k$ (the degree of the polynomial), the better our test performs.  

Our main goal in this section is to develop computationally-efficient algorithms in higher dimension and towards this goal, we can use a similar technique as in Section~\ref{sec:alg-variational} to obtain a natural analog in higher dimension.  In particular, we consider the testing function
\[
\psi^{(k)}_d := \sup_{v \in \mathbf{S}^{d-1}}\;\Bigl\{ \mathbb{E}_{Q_{\R}}\bigl[\langle v, X \rangle^{2k}] \Bigr\},
\]
and its empirical counterpart
\[
\psi^{(k)}_{d,n}(Z_1, \ldots, Z_n) := \sup_{v \in \mathbf{S}^{d-1}}\;\Bigl\{ \frac{1}{n} \sum_{i=1}^{n} \langle v, Z_i \rangle^{2k} \Bigr\}.
\]
We recognize the above quantity as the injective tensor norm of the tensor $\frac{1}{n} \sum_{i=1}^{n} Z_i^{\otimes 2k}$.  The key insight which we build on is that while the computation of the polynomial optimization problem $\psi^{(k)}_{d, n}$ is believed to be intractable in the worst case~\cite{BarBHKSZ12}, under certain average-case assumptions on the data, 
a good (additive) approximation of it can be computed in polynomial time using the sum-of-squares technique~\cite{KotSS18,HopLi18}.  
Indeed, our algorithms for mean and covariance estimation to follow build directly on this intuition.

Before describing our computationally-efficient algorithms, we begin with some preliminaries.  

\subsection{Preliminaries}
Equipped with the intuition from the previous section, we turn now to our main algorithmic development.  We focus in this section on the all-or-nothing setting.
The mean and covariance estimators we propose in this section will therefore discard all missing data and only consider the remaining observations.

We will mainly consider the following data generation process (Definition~\ref{def:data-generating-process}), which is strong enough to simulate $n$ i.i.d.\ samples for any $Q \in \cR(P,\epsilon)$.  
\begin{definition}\label{def:data-generating-process}
    (Data generating process)
    Let $n \in \bN$, $\epsilon \in [0,1]$. Let $P \in \mathcal{P}(\bR^d)$. 
    We generate $n$ samples in $\R^d \cup \{\star^d\}$ as follows:
    \begin{enumerate}
      \item Let $\Tstar \subseteq \R^d$ be a multiset of $m$ i.i.d.\ data from $P$ for $m \sim \Bin(n, 1- \epsilon)$ (the MCAR observations). 
      \item Let $\Tprime \subseteq \R^d$ be a multiset of $n-m$ i.i.d.\ data from $P$ and $\Tdprime \subset \Tprime$ be an arbitrary subset (the non-missing MNAR observations).
      \item Let $T:= \Tstar + \Tdprime \subseteq \R^d$. The estimator observes $T$ and $n - |T|$ copies of $\star^d$. 
    \end{enumerate}
\end{definition}

We will take $P = \mathsf{N}(\mu,\Sigma)$ for some $\mu \in \R^d$ and $\Sigma \succ 0$. In particular, for the mean estimation setting we will assume $\Sigma = I_d$ and $\mu$ is unknown, while for the covariance estimation setting we will assume $\Sigma$ is unknown and $\mu = 0$.
\begin{definition}[Good event]
\label{def:good-event}
Consider the sets $\Tstar$ and $\Tprime$ from \Cref{def:data-generating-process}.
Let $\mc{E}$ be the event that:
\begin{enumerate}
    \item $\frac{n}{m} \leq \frac{1+\epsilon}{1-\epsilon}$;
    \item For all $\ell \in [2k]$ and $S \in \{\Tstar, \Tstar + \Tprime\}$:
    \begin{align*}
     \bigl\|\E_{x \sim S}\bigl[\bigl(\Sigma^{-1/2}(x-\mu)\bigr)^{\otimes \ell}\bigr] - \E_{z \sim \mathsf{N}(0,I_d)}\bigl[z^{\otimes \ell}\bigr]\bigr\|_\infty \leq \frac{\epsilon}{(4d)^{k} 2k}   \,.
    \end{align*}
\end{enumerate}    
\end{definition}

We will prove our mean and covariance estimators are accurate assuming $\mc{E}$ holds. The following lemma thus bounds the sample complexity of both estimators.

\begin{lemma}[Sample complexity of good event $\mc{E}$]
\label{lem:lower-bound-E-prob}
    Consider \Cref{def:data-generating-process,def:good-event} and suppose that $P = \mathsf{N}(\mu,\Sigma)$.
    There exists a tuple of universal, positive constants $c_1, c_2$ such that 
        \begin{align*} %
        n \geq \frac{\left(c_1 k d \log(1/\delta)\right)^{c_2 k}}{(1-\eps)\eps^2}\,,
        \end{align*}
    implies the event $\mc{E}$ occurs with probability at least $1-3\delta$.
\end{lemma}
We defer the proof of this lemma to Section~\ref{sec:proof-lem-sample-complexity}.  Equipped with these preliminaries, we turn next to studying efficient algorithms and computational lower bounds for mean estimation.

\subsection{Mean estimation}\label{sec:sos-mean-est}
We first describe our sum-of-squares based estimator and provide an upper bound on its sample complexity in Theorem~\ref{thm:mean-est-sos-upper-bound}.
We defer the proof of this theorem to \Cref{app:pf-mean-est-sos-ub}.

\begin{definition}[SoS program for mean estimation]
    \label{def:sos-mean-axioms}
    Consider the following\footnote{Refer to preliminaries for notation.} constraint on the $d$-dimensional (indeterminate) variable $\theta$.
        \begin{align}
        \left\{\norm{v}_2^2 = 1\right\} \sos{v}{4k} \E_{x \sim T} \langle v, x - \theta \rangle^{2k} \leq \left(\frac{(1+\eps)^2}{1-\epsilon} \right)\E_{G \sim \mathsf{N}(0,1)} [G^{2k}]\,.
        \label{eqn:sos-constraint-mean-est}
    \end{align}
\end{definition}
\begin{algorithm}[htb]{}
	\begin{algorithmic}[1]
		\Function{Mean-estimator}{$T,k,\epsilon$}
		\State Find a pseudo-expectation $\pE$ of degree $4k$ which satisfies the system of \Cref{def:sos-mean-axioms}. If no such $\pE$ exists, \Return the all-zeros vector.
		\State \Return  $\hat{\theta} := \pE[\theta]$.
		\EndFunction
	\end{algorithmic}
	\caption{Mean Estimator}
	\label{alg:mean-estimator}
\end{algorithm}

\begin{theorem}[SoS algorithm for mean estimation]\label{thm:mean-est-sos-upper-bound}
    Let $P = \mathsf{N}(\theta_\star,I_d)$
    and generate $T \subset \R^d$ by Definition~\ref{def:data-generating-process}. Let $k \in \bN$ satisfy $k \leq \sqrt{d}$ and let $\widehat{\theta}$ be the output of \Cref{alg:mean-estimator}.
    Suppose $n$ satisfies the conditions of Lemma~\ref{lem:lower-bound-E-prob}; i.e., $n \gtrsim \left(\frac{\left(\Theta(k d \log(1/\delta))\right)^{\Theta(k)}}{(1-\eps)\eps^2}\right)$.
    Then, with probability at least $1-3\delta$, the following two conditions both hold:
    \begin{enumerate}
        \item (Satisfiability) The constraint~\eqref{eqn:sos-constraint-mean-est} is satisfiable.
        \item (Accuracy) Let $\gamma = \left(\frac{(1+\epsilon)^2 - (1-\epsilon)^3}{(1-\epsilon)^2}\right)$.
        Any degree-$4k$ pseudoexpectation $\pE$ satisfying this constraint also satisfies 
        \begin{align*}
            \bigl \|\widehat{\theta} - \theta_\star \bigr\|_2 \leq 8 \left(\sqrt{\frac{\gamma}{k}} + \frac{\gamma}{k}\right).
        \end{align*}
    \end{enumerate}
\end{theorem}

Let us compare this result with the information-theoretic results established in \Cref{thm:mean-est-it-ub-lb}. 
First, the algorithm above runs in time $\poly(n^k, d^k,1/\eps)$ as opposed to the algorithm in \Cref{thm:mean-est-it-ub-lb}, which needs  $\poly(n,e^d)$ time.
To facilitate the comparison on sample complexity, consider the case when $\epsilon \in (0,1)$ is a constant.
Then $\gamma = O(1)$ and \Cref{thm:mean-est-sos-upper-bound} implies that to obtain error $\rho \in (0,1)$,
one can take $k \asymp 1/\rho^2$ and the resulting sample complexity would be $ d^{O(1/\rho^2)} \cdot f(\rho)$ for some function $f$.
Contrast this with the information-theoretic rate (\Cref{thm:mean-est-it-ub-lb}), where one needs only $d\cdot\widetilde{f}(\rho)$ samples for some function $\widetilde{f}$.  Given this gap, it is natural to wonder whether there are computationally-efficient algorithms which improve upon the sample complexity required by Algorithm~\ref{alg:mean-estimator}.  The next result shows that such an improvement is impossible for any computationally-efficient statistical query algorithm.

\begin{definition}[Testing Problem for Mean Estimation]
\label{def:testing-mean-lb}
    Let $P$ be the underlying (unknown) distribution equal to $\mathsf{N}(\mu,I_d)$ for $\mu \in \R^d$.
    Let the contamination rate be $\eps$
    and the signal strength be $\rho$.  Let $P'$ be a contaminated conditional distribution $P' \in \cR_\R(P,\eps)$.
    Suppose the tester has access (either through i.i.d.\ samples $(X_1,\dots,X_n)\sim {P'}^{\otimes n}$ or an SQ oracle $\mathrm{STAT}_{P'}(\kappa)$ with tolerance $\kappa$) to the distribution $P'$. 
    The testing problem is the following:
    \[
    H_0:\; \mu = 0 \quad \text{ vs. } \quad H_1: \; \|\mu\|_2 \geq \rho.
    \]
\end{definition}
\begin{theorem}[SQ lower bound for mean estimation]
\label{thm:sq-lower-bound-mean}
There exists a large constant $C > 0$ and small constants $c,c'>0$ such that the following holds.
Consider \Cref{def:testing-mean-lb} with (i) $\rho \leq \epsilon/C$ and (ii) $d \geq  \left( (\eps/\rho)\log (2\dimension) \right)^{C}$.
Suppose that $\eps \leq c$ and $1 - q \leq c$.  Then any SQ algorithm that solves \Cref{def:testing-mean-lb} with fewer than $2^{\dimension^{c'}}$ many queries must use tolerance $\SQtolerance \leq d^{-m}$ for $m\geq c' \frac{\eps^2/\rho^2}{\log(\eps^2/\rho^2)}$.
\end{theorem}
The proof of this theorem uses the NGCA framework (\Cref{def:ngca}) and is deferred to \Cref{app:sq-lower-bound}.
Taken together, \Cref{thm:mean-est-sos-upper-bound} and \Cref{thm:sq-lower-bound-mean}, essentially settle the computational sample complexity for constant $\eps$.
For vanishing $\eps \to 0$, there is a gap: the lower bound on the effective sample complexity scales as $d^{\eps^2/\rho^2}$, 
whereas the upper bound requires roughly $d^{\eps/\rho^2}$ samples.

In the next section, we turn our attention to developing a parallel story for covariance estimation.  

\subsection{Covariance estimation}\label{sec:sos-cov-est}
Our sum-of-squares based covariance estimator solves for a pseudoexpectation over variables $M,B \in \R^{d \times d}$ satisfying the following constraints. 

\begin{definition}[SoS program for covariance estimation]
\label{defn:cov-sos-constraints}
    Let $\alpha > 0$. Let $\mc{A}_{\mathrm{covariance}}$ to be the union of the following sets of constraints\footnote{Refer to preliminaries for notation.} over $M,B \in \R^{d \times d}$:
\begin{itemize}
    \item[1.] Let $\mc{A}_{\mathrm{moments}}$ be the set
    \begin{align}
     \left\{\left\{\norm{v}_2^2 = 1\right\} \sos{v}{8k} \E_{x \sim T} \langle v, M (\E_{x \sim T} [x x^\top]^{-\frac{1}{2}}) x \rangle^{2\ell} \in \left(1\pm 10\epsilon \right)\E_{G \sim \mathsf{N}(0,1)} [G^{2\ell}] \right\}_{\ell \in [k]}.
    \label{eqn:sos-constraint-cov-est}
\end{align}
    \item[2.] Let $\mc{A}_{\mathrm{deviation}} := \{M \succeq 0\} \cup \{B = M - I_d\} \cup \{\alpha I \succeq B \succeq -\alpha I_d\} \cup \{ B^\top B \preceq \alpha^2 I_d\}$.
\end{itemize}

\end{definition}

Analogous to the mean estimation setting, it is relatively straightforward to show that any pseudoexpectation $\pE$ that satisfies the constraint~\eqref{eqn:sos-constraint-cov-est} must also satisfy the relation $\pE\bigl[\|\Sigma^{1/2} M v\|_2^k\bigr] \approx 1$, from which we may conclude that $\| \pE[M \Sigma M] - I\|_{\rm op}$ is small. However, this in and of itself does not guarantee that $\|\pE[M] - \Sigma\|_{\rm op}$ is small (i.e., $\pE[M]$ may not necessarily be a good estimator). Thus, we take the additional step of solving for an explicit matrix $A$ that satisfies  $\|\pE[M A M] - I\|_{\rm op}$. This step motivates the additional constraints in $\mc{A}_{\mathrm{deviation}}$: if $\pE$ satisfies all of these constraints, then we can guarantee $A$ is a good estimator (modulo the whitening step implicit in constraint~\eqref{eqn:sos-constraint-cov-est}). We make this key fact precise in Lemma~\ref{lem:cov-est-sos-upper-bound-helper}, whose proof we provide in Section~\ref{sec:proof-lem-cov-est-sos-upper-bound-helper}. We then use it to prove an upper bound on the sample complexity of our estimator in Theorem~\ref{thm:cov-est-sos-upper-bound}.

\begin{algorithm}[htb]{}
	\begin{algorithmic}[1]
		\Function{Covariance-estimator}{$T,k,\epsilon,\alpha$}
		\State Find a pseudo-expectation $\pE$ of degree $8k$ which satisfies the system of \Cref{defn:cov-sos-constraints}. If no such $\pE$ exists, \Return $ \E_{x \sim T} [x x^\top]$.
        \State Let $A_*$ be a matrix that minimizes 
        $$\min_{A: A \succeq 0}\|\pE[M A M] - I_d\|_{\rm op}\,.$$ 
		\State \Return $\widehat{\Sigma} = (\E_{x \sim T} [x x^\top]^{1/2})A_*(\E_{x \sim T} [x x^\top]^{1/2})$.
		\EndFunction
	\end{algorithmic}
	\caption{Covariance Estimator}
	\label{alg:cov-estimator}
\end{algorithm}
\begin{lemma}\label{lem:cov-est-sos-upper-bound-helper}
    Let $\pE$ be a degree-$4$ pseudoexpectation satisfying $\mc{A}_{\mathrm{deviation}}(\alpha)$ , and for $\beta > 0$ let $H \in \R^{d \times d}$ be an arbitrary symmetric matrix satisfying $\|\pE[M H M]\|_{\rm op} \leq \beta$. Then if $1-2\alpha-\alpha^2 > 0$ we have
    \begin{align*}
        \norm{H}_{\rm op} \leq \frac{\beta}{1-2\alpha-\alpha^2}.
    \end{align*}
\end{lemma}

We describe our estimator and upper bound its sample complexity in Theorem~\ref{thm:cov-est-sos-upper-bound}.

\begin{theorem}\label{thm:cov-est-sos-upper-bound}
    Fix $\eps < 1/100$. Let $P = \mathsf{N}(0,\Sigma)$ for $\Sigma \succ 0$. Generate $T \subset \R^d$ by Definition~\ref{def:data-generating-process}. 
    Let $k = 2^p$ for some $p \in \bN$ and let $\alpha=10\eps$ (Definition~\ref{defn:cov-sos-constraints}).
    Suppose $n$ satisfies the conditions of Lemma~\ref{lem:lower-bound-E-prob}, i.e., $n \gtrsim \frac{\left(\Theta(k d \log(1/\delta))\right)^{\Theta(k)}}{\eps^2}$.
    Then with probability at least $1-3\delta$, the following conditions both hold:
    \begin{itemize}
        \item[1.] (Satisfiability) The constraints $\mc{A}_{\mathrm{covariance}}$ are satisfiable.
        \item[2.] (Accuracy) Let $\widehat{\Sigma}$ be the output of \Cref{alg:cov-estimator}. 
        Then 
        \[
        \bigl \| \Sigma^{-1/2} \widehat{\Sigma} \Sigma^{-1/2} - I_d \bigr\|_{\mathrm{op}} \lesssim \frac{\eps}{k}.
        \]
    \end{itemize}
\end{theorem}

Let us compare this result with the information-theoretic rate established in \Cref{thm:cov-est-it-ub-lb}.
First, the range of permissible $\eps$ in \Cref{thm:cov-est-sos-upper-bound} is restricted to be bounded away from $1$; by contrast, \Cref{thm:cov-est-it-ub-lb} demonstrates that Algorithm~\ref{alg:mean-estimator} works for all $\eps < 1$. 
Next, let us discuss the sample complexity to obtain error $\rho$ when $\eps$ is a constant.
In \Cref{thm:cov-est-sos-upper-bound}, we can set $k \asymp 1/\rho$,
and the sample complexity of the algorithm above is $d^{\Theta(1/\rho)} \cdot f(\rho)$, as opposed to the information-theoretic sample complexity of $d\cdot \widetilde{f}(\rho)$ in \Cref{thm:cov-est-it-ub-lb}; here $f$ and $\widetilde{f}$ are dimension-independent functions.
We next demonstrate such a gap is inherent for all computationally-efficient SQ algorithms.

\begin{definition}[Testing Problem for Covariance Estimation]
\label{def:testing-cov-lb}
    Let $P$ be the underlying (unknown) distribution equal to $\mathsf{N}(0,I_d + ww^\top)$ for an unknown vector $w \in \R^d$.
    Let the contamination rate be $\eps \in (0,1/2)$
    and the signal strength be $\rho \in (0,1/2)$.
    Let $P'$ be a contaminated conditional distribution $P' \in \cR_\R(P,\eps)$.
    Suppose the tester has access (either through i.i.d.\ samples $(X_1,\dots,X_n)\sim P^{\otimes n}$ or an SQ oracle $\mathrm{STAT}_{P'}$) to the distribution $P'$. 
    The testing problem is the following:
    \[
    H_0:\; \|ww^\top\|_{\rm op}= 0 \quad \text{ vs. } \quad
    H_1: \; \|ww^\top\|_{\rm op} \geq \rho.
    \]
\end{definition}

The following theorem provides an SQ lower bound for the above testing problem, again relying on the NGCA framework (\Cref{def:ngca}).

\begin{theorem}[Computational Lower Bound for Covariance Estimation]
\label{thm:sq-lower-bound-cov}
There exist a large constant $C > 0$ and a small constant $c>0$ such that the following holds.
Consider \Cref{def:testing-cov-lb} with (i) $\eps/\rho \geq C$ and (ii) $d \geq  \left( (\eps/\rho)\log (2\dimension) \right)^{C}$.
Then any SQ algorithm that solves \Cref{def:testing-cov-lb} with fewer than $2^{\dimension^{c}}$ many queries must use tolerance $\SQtolerance \leq d^{-m}$ for $m\geq c \frac{\eps/\rho}{\log(\eps/\rho)}$.
\end{theorem}

It is worthwhile to contrast this guarantee with that of Theorems~\ref{thm:mean-est-sos-upper-bound} and~\ref{thm:sq-lower-bound-mean}: Even for vanishing $\eps \to 0$, we see that both the upper bound in \Cref{thm:cov-est-sos-upper-bound} and the SQ lower bound in \Cref{thm:sq-lower-bound-cov} scale roughly as $d^{\Theta(\eps/\rho)}$.

\section{Nearly optimal, computationally-efficient linear regression}

\label{sec:lin-reg}
In contrast with Gaussian mean and covariance estimation, we show that in Gaussian linear regression with missing observations, there is (nearly) no statistical-computational gap in many parameter regimes of interest.
Before presenting the efficient estimator in Section~\ref{sec:main-result-lr}, we first specify the model in Section~\ref{sec:lr-model} and then provide a glimpse at the technique in Section~\ref{sec:warmup-lr}.

\subsection{Model preliminaries} \label{sec:lr-model}
In this section, we consider linear regression in the setting with isotropic Gaussian covariates and Gaussian noise.
That is, for $\theta_\star \in \R^d$, we take the base distribution $P_{\theta_\star}$ to be the law of $(X, Y)$ generated as
\begin{subequations} \label{eq:model-lr}
\begin{align}
  X \sim \mathsf{N}(0, I_d), \quad \text{ and } \quad  Y \mid X \sim \mathsf{N}\bigl(X^{\top} \theta_\star, \sigma^2\bigr).
\end{align}
We take the parameter space and loss to be
\begin{align}
  \cP_{\mathrm{LR}}(\sigma^2, \epsilon) = \{\cR(P_\theta, \epsilon) \mid \theta \in \bR^d\}
  \qquad\textrm{and}\qquad
  L(\theta, \theta') = \| \theta - \theta' \|_2.
\end{align}
\end{subequations}
Let us emphasize that in this setting, when an observation is masked, the entire observation is masked, rather than just the response.
We also emphasize that in the MNAR component, the missingness can depend arbitrarily on both the response $Y$ and the covariate $X$.

In order to further clarify the setting, it will be helpful to introduce auxiliary random variables to represent contaminated distributions.
To this end, for a distribution $R \in \cR(P_\theta, \epsilon)$, we know that there exist random variables 
$\Omega \in \{0, 1\}, B \sim \Bern(\epsilon), (X, Y) \sim P$
such that
\begin{align*}
  (X, Y) \ostar \bigl\{(1-B) + B \Omega\bigr\} \sim R,
\end{align*}
where $B$ is independent of the triple $(X,Y, \Omega)$.
Moreover, the response $Y$ can be generated as $X \cdot \theta_\star + \sigma G$ for $G \sim \mathsf{N}(0, 1)$ independent of $X$.
Observe that $(X,Y) \neq \star$ if and only if $(1-B) + B\Omega = 1$.  

Equipped with these preliminaries, we next provide the key intuition behind our estimator.

\newcommand{\thetapop}{\theta_{\mathrm{pop}}}
\newcommand{\ofk}{^{(k)}}

\subsection{Warm-up: The population setting} \label{sec:warmup-lr}
The natural estimator to consider is that of ordinary least squares (OLS), which minimizes the sum of squared residuals.  Unfortunately, in the realizable contamination setting, the OLS estimator can suffer from  bias.  In order to mitigate this, we consider an analogue of the polynomial estimators developed in Section~\ref{sec:comp-mean-cov}.  In particular, rather than minimizing the sum of squared residuals, we attempt to minimize a loss which raises the residuals to a larger power $2k>2$.
This magnifies the error resulting from deviation away from $\theta_\star$, regardless of the contamination.
By taking $k \uparrow \infty$, the effects of the contamination vanish and the population estimator converges to $\theta_\star$.  This intuition can be made precise via a short calculation in the population setting, which we provide presently.  

For $k \in \bN$, we define the loss $F^{(k)}: \R^d \rightarrow \R$ and its minimizer $\thetapop^{(k)}$ as 
\[
F^{(k)}(\theta) := \mathbb{E}\bigl[(Y - X^{\top} \theta)^{2k} \cdot \1\{(X, Y) \neq \star\}\bigr] \quad \text{ and } \quad \thetapop^{(k)} := \argmin_{\theta \in \R^d}\, F^{(k)}(\theta).
\]
Our strategy will be to control the error $\| \thetapop^{(k)} - \theta_{\star} \|_2$ by applying the inequality
\[
\bigl \| \thetapop^{(k)} - \theta_{\star} \bigr \|_2 \leq \frac{\| \nabla F^{(k)}(\theta_{\star})\|_2}{\mu},
\]
where $\mu$ denotes the modulus of strong convexity of the population loss.  Let us begin by computing $\mu$.
We compute the strong convexity parameter by lower bounding the Hessian at every $\theta$ by
\begin{align*}
\nabla^2 F^{(k)}(\theta) &= 2k (2k-1) \mathbb{E}\bigl[(Y - X^{\top} \theta_{\star})^{2k-2} XX^{\top} \cdot \1\{(X, Y) \neq \star\}\bigr] \\
&\succeq 2k (2k-1) \mathbb{E}\bigl[(1-B)\cdot \bigl\{X^{\top}(\theta_{\star} - \theta) + \sigma G\bigr\}^{2k-2} XX^{\top}\bigr]\\
&\stackrel{(a)}{=} (1-\epsilon) 2k (2k-1) \sum_{\ell=0}^{2k-2} \sigma^{2k-2-\ell} \E[G^{2k-2-\ell}] \cdot \mathbb{E}\bigl[\bigl\{X^{\top} (\theta_{\star} - \theta)\bigr\}^{\ell} XX^{\top}\bigr]\\
&\stackrel{(b)}{\succeq} (1-\epsilon) 2k (2k-1) \sigma^{2k-2} \mathbb{E}\bigl[G^{2k-2}\bigr] \cdot I_d\\
& = (1-\epsilon) 2k (2k-1) \sigma^{2k-2} (2k - 3)!! \cdot I_d.
\end{align*}
Above, step $(a)$ follows by mutual independence of $X$, $G$, and $B$; and step $(b)$ follows from the fact that odd moments of the standard Gaussian distribution are zero.

On the other hand, we upper bound the norm of its gradient at the true solution $\theta_{\star}$ as
\begin{align*}
    \| \nabla F^{(k)}(\theta_{\star}) \|_2 &= 2k \cdot \bigl \| \mathbb{E}\bigl[(Y - X^{\top} \theta_{\star})^{2k-1} X \cdot  \1\{(X, Y) \neq \star\}\bigr]\bigr\|_2 \\
    &= 2k \sigma^{2k-1}\cdot \bigl \| \mathbb{E}\bigl[(1-B) \cdot G^{2k-1} X\bigr] + \mathbb{E}\bigl[B\Omega \cdot G^{2k-1} X\bigr]\bigr\|_2 \\
    &= 2k \sigma^{2k-1} \epsilon \bigl \| \mathbb{E}\bigl[G^{2k-1} X \cdot  \Omega\bigr]\bigr\|_2 \\
    &= 2k \sigma^{2k-1} \epsilon \cdot \sup_{v \in \mathbf{S}^{d-1}}\; \bigl\{ \mathbb{E}\bigl[G^{2k-1} (v^{\top}X) \Omega\bigr]\bigr\} \\
    &\leq 2k \sigma^{2k-1} \epsilon \cdot \sup_{v \in \mathbf{S}^{d-1}}\; \bigl\{ \mathbb{E}\bigl[|G|^{2k-1} |v^{\top}X|\bigr]\bigr\} \\
    &= 2\sqrt{2/\pi} \cdot k \sigma^{2k-1} \epsilon \cdot (2k-2)!!,
\end{align*}
where again we make use of the mutual independence of $X$, $G$, and $B$.
Putting the pieces together, we deduce that
\begin{align*}
\bigl \| \thetapop^{(k)} - \theta_{\star} \bigr \|_2 &\leq \frac{\| \nabla F^{(k)}(\theta_{\star}) \|_2}{(1-\epsilon) 2k (2k-1) \sigma^{2k-2} (2k - 3)!!}
\lesssim \sigma \cdot \frac{\epsilon}{1- \epsilon} \cdot \frac{(2k-2)!!}{(2k-1)!!} \lesssim \frac{\sigma \epsilon}{(1-\epsilon)\sqrt{k}},
\end{align*}
where the final inequality follows from Stirling's inequality (see \cref{fact:stirling}).
Therefore, while the quadratic loss (corresponding to $k=1$) yields an error guarantee of $O(\frac{\sigma \epsilon}{1-\epsilon})$, taking $k$ to be large gives a population quantity with arbitrarily small error guarantees.
In the next section, we derive finite sample guarantees for the empirical variant of this estimator by carefully choosing the value of $k$.

\subsection{Main result} \label{sec:main-result-lr}
In the population setting, we established that raising residuals to arbitrarily large polynomial powers has a monotone effect on the estimation error.  In the finite sample setting, this is no longer the case as raising the residuals to higher powers render the empirical risk heavier tailed.  Our main result demonstrates that selecting a value $k$ which balances the population error and the heavy-tailed nature of the risk yields an estimator that is nearly optimal. %
Formally, given samples $\{(x_1, y_1), (x_2, y_2), \dots, (x_n, y_n)\} \subseteq \bR_\star^{d+1}$ we define the empirical loss 
\begin{subequations}
\begin{align} \label{eq:empirical-loss-k-power}
F\ofk_n(\theta) := \frac{1}{n}\sum_{i=1}^n \mathbf{1}\{(x_i, y_i) \neq \star\} \cdot (y_i - x_i \cdot \theta)^{2k}
\end{align}
as well as its minimizer
\begin{align} \label{eq:empirical-minimizer-k-pwer}
    \widehat{\theta}\ofk_n = \arg\min_\theta F\ofk_n(\theta).
\end{align}
\end{subequations}

The following theorem provides information-theoretic upper bounds for the linear regression problem as well as the upper bound obtained by a computationally-efficient estimator that approximates $\hat{\theta}\ofk_n$.
We provide the proof of the upper bound in Section~\ref{pf:lin-reg-ub} and the lower bound, which builds on the construction in \cite{DiaGKPX26-adaptive}, in Section~\ref{sec:proof-linreg-est-it-lb}.

\begin{theorem}\label{thm:lin-reg-ub-lb}
There exists constants $c, C > 0$ such that the following is true.
Let $\sigma^2 > 0$, $\epsilon < 1$, $n\in\bN$, $\delta \in [\{(1-\epsilon)n\}^{-c}, 0.5]$ and let $\alpha = \sqrt{\frac{d+\log(\frac{1}{\delta})}{(1-\epsilon)n}}$.
If
\[
  n \geq C \cdot \frac{d+\log(\frac{1}{\delta})}{1-\epsilon},
\]
then the minimax $1-\delta$ quantile~\eqref{def:minimax-quantile} satisfies
\[
  \frac{\sigma \cdot \log(\frac{1}{1-\epsilon})}{\sqrt{\log\Big(1 + \big\{\frac{\epsilon}{(1-\epsilon)\alpha}\big\}^2\Big)}}
  \lesssim
  \cM(\delta, \cP_{\rm LR}(\sigma^2, \epsilon),L)
  \lesssim \frac{\sigma \cdot \big(\frac{\epsilon}{1-\epsilon}\big) \cdot \big(1 \vee \log\log\big(\frac{\epsilon}{\alpha(1-\epsilon)}\big)\big)}{\sqrt{\log\left(1 + \big\{\frac{\epsilon}{(1-\epsilon)\alpha}\big\}^2\right)}}.
\]
Moreover, the upper bound is achieved by the estimator $\widehat{\theta}\ofk_n$ for some $k \lesssim \log(\frac{\epsilon}{\alpha(1-\epsilon)})$ which can be approximated to accuracy of the same order in ${\rm poly}(k, d, n)$ time.
\end{theorem}
We next provide several remarks about this theorem. 
First, let us emphasize that the estimator $\widehat{\theta}_n^{(k)}$ can be approximated to sufficient accuracy in polynomial time.  In particular, in Section~\ref{sec:lr-efficient-alg} we show that (for instance) the ellipsoid method can compute an approximate minimizer in polynomial time.  We contrast this with the estimator provided in~\cite[Section 5]{ma2024estimation}, which is based on projections in the Kolmogorov distance and requires exponential time to compute.  

Second, given the results of the previous two sections and the fact that the estimator of~\cite{ma2024estimation} is based on Kolmogorov projections, it is tempting to guess that there may a statistical--computational gap in the regression setting as well.  Indeed, in our well-specified setting, the linear regression problem can be reduced to that of covariance estimation.  Indeed, we note that the covariate-response pair $(X, Y)$ is jointly Gaussian:
\[
(X, Y) \sim \mathsf{N}\biggl(0, \begin{bmatrix}
I_d & \theta_{\star} \\
\theta_{\star}^{\top} & \| \theta_{\star} \|_2^2 + \sigma^2
\end{bmatrix}
\biggr).
\]
We can thus apply the efficient estimator developed in Section~\ref{sec:comp-mean-cov} to obtain an estimate of the coefficients $\theta_{\star}$.  In order to obtain an estimate of $\theta_{\star}$ with error at most $\rho$, this would require at least $n \gtrsim d^{\Omega(1/\rho^2)}$ many samples.  A similar dimension dependence was obtained by~\cite[Corollary 3.20]{LeeMeZa24} in the related setting of efficient linear regression with truncated statistics, where the sample complexity is also $n \gtrsim d^{\mathrm{poly}(1/\rho)}$.
On the other hand, our empirical risk minimization procedure is computable in polynomial time and requires a significantly smaller sample complexity, which we discuss next.  

Next, in order to simplify the upper and lower bounds, let us consider the parameters $\epsilon, \delta, \sigma$ to be constants.  The upper and lower bounds then simplify to
\[
\frac{1}{\sqrt{\log(1/\alpha^2)}} \lesssim \mathcal{M}_n(\delta, \mathcal{P}_{\mathrm{LR}}\bigl(\sigma^2, \epsilon), L\bigr) \lesssim \frac{\log\log(1/\alpha)}{\sqrt{\log(1/\alpha^2)}}.
\]
This reveals that in this setting, our method is optimal up to a multiplicative factor which is a doubly iterated logarithm in $\alpha$.  Moreover, for constant $\epsilon$ and $\delta$, we find that $\alpha \asymp \sqrt{d/n}$ so that the upper and lower bounds further simplify to
\[
\frac{1}{\sqrt{\log(n/d)}} \lesssim \mathcal{M}_n(\delta, \mathcal{P}_{\mathrm{LR}}\bigl(\sigma^2, \epsilon), L\bigr) \lesssim \frac{\log\log(n/d)}{\sqrt{\log(n/d)}}.
\]
In particular, this implies that our estimator requires a sample complexity of $n \gtrsim d e^{\Omega(\log^2(1/\rho)/\rho^2)}$ to obtain error $\rho$, which is linear in dimension.  We note that this sample complexity is polynomially smaller in dimension than that required by the inefficient algorithm of~\cite{ma2024estimation}, which requires $n \gtrsim d^{36/31}$.  

Next, we note that, while we proved the upper bound for covariates distributed as $X \sim \mathsf{N}(0, I_d)$, only mild assumptions are needed for this estimator to work.  Indeed, even if $X$ is heavy-tailed (with at least a constant number of bounded moments), our algorithm---with a modified analysis---continues to succeed with similar error rates in the constant failure probability ($\delta$) regime.

Let us mention two limitations of our result.  First, we note that, in contrast with Theorems~\ref{thm:mean-est-it-ub-lb} and~\ref{thm:cov-est-it-ub-lb}, we require the lower bound $\delta \geq \{(1-\epsilon)n\}^{-c}$ for a small enough constant $c > 0$.
This is not a huge limitation, as one could take the guarantees of \Cref{thm:lin-reg-ub-lb} for constant failure probability $1/4$ with sample complexity $n_0$ and use a simple boosting procedure to turn it into a (computationally-efficient) estimator that works for arbitrary $\delta \in (0,1/4)$ with a sample complexity of $\log(1/\delta) \cdot n_0$.\footnote{The boosting procedure is as follows: run the geometric median-of-means~\cite{Min15} procedure on $k\asymp \log(1/\delta)$ independent estimates (obtained by running \Cref{thm:lin-reg-ub-lb} on disjoint subsets), each of which is correct with probability $3/4$ (guaranteed by \Cref{thm:lin-reg-ub-lb} if each subset is of size $n_0$).}
Still, the resulting sample complexity has a multiplicative dependence on $\log(1/\delta)$ as opposed to an additive dependence in the lower bound.
We leave further investigation in the extremely high-confidence setting as an interesting open direction.  Second, our algorithm yields an error which scales linearly in $\tau = \epsilon/(1-\epsilon)$.  By contrast, the lower bound (as well as the rates for mean and covariance estimation) scale logarithmically in $\tau$.  While these match for $\epsilon < c$ for $c > 0$ a small enough universal constant, this exponentially large gap becomes more striking as $\epsilon \rightarrow 1$.   

\looseness=-1Finally, we remark that the recent work~\cite{KouridakisMeKaCa26} also considers linear regression with truncated statistics.  The authors focus on the regime where the truncation  depends only on the response variable and demonstrate a polynomial time estimator that achieves error $\rho$ with a sample
complexity of ${\rm poly}(d, 1/\rho, k)$, where $k$ is the number of intervals that form the truncation set.  This sample complexity admits a dimension-dependence which is polynomially larger than our estimator, but a dependence on $1/\rho$ which is exponentially smaller than ours.  We re-emphasize that, in general, these models are incomparable (see Section~\ref{sec:related}) and that our information-theoretic lower bound precludes such improvements to our dependence on $1/\rho$.

\section{Discussion} \label{sec:discussion}
In this paper, we studied several estimation problems with missing data in the realizable contamination problem.  We focused on the so-called all-or-nothing setting in which each observation is either completely observed, or not at all.  Under this setting we showed:
\begin{itemize}
    \item In \emph{mean estimation}, to achieve error in $\ell_2$ norm at most $\rho$, there appears to be a sample complexity gap: Ignoring all other problem parameters, $n \gtrsim d e^{1/\rho^2}$ samples are necessary and sufficient information-theoretically, but computationally-efficient statistical query algorithms (along with other restricted families of algorithms mentioned in \Cref{sec:prelim-sq-lower-bound}) require $n \gtrsim d^{\Omega(1/\rho^2)}$ many observations.  This latter lower bound is achieved by a sum-of-squares estimator.
    \item In \emph{covariance estimation}, to achieve error $\rho$ in relative operator norm, analogous claims hold: $n \gtrsim d e^{1/\rho}$ samples are necessary and sufficient information-theoretically, but statistical query algorithms  require $n \gtrsim d^{1/\rho}$ samples.
    \item In \emph{linear regression}, we show that for most parameters, these statistical--computational gaps do not exist. In particular, we show that the sample complexity $n \gtrsim d e^{1/\rho^2}$ is necessary information-theoretically and can be matched by a computationally-efficient algorithm.
\end{itemize}
There remain several open questions, several of which we discuss here.  First, our computationally-efficient algorithm for linear regression does not match the information-theoretic rate for extremely low failure probabilities $\delta$ or when $\epsilon \to 1$.  
Going further, it would be interesting to understand how the rates of convergence change---both computationally and statistically---for more complicated regression models than linear regression.

Next, we note that the all-or-nothing setting of missingness is less common than the multiple pattern setting (see Figure~\ref{fig:both}).  In Appendix~\ref{sec:multiple-patterns}, we show how to extend our algorithms and lower bounds to the setting with multiple patterns and show that for both mean and covariance estimation, our all-or-nothing algorithms can be adapted to achieve error which is a multiplicative factor of $C_{\lvert \mathbb{S} \rvert}$ larger than the optimal error, where $C_{\lvert \mathbb{S} \rvert}$ is a positive constant depending only on the number of patterns.  While this error inflation may be reasonable for a small number of patterns, it may be suboptimal for a large number of missingness patterns.  An important area for future work is to determine the fundamental limits of estimation, even for simple tasks like mean and covariance estimation, in this multiple pattern setting.

\printbibliography
\newpage
\appendix
\section{Proofs of information-theoretic upper bounds for mean and covariance estimation}
\subsection{Mean estimation: Proof of Theorem~\ref{thm:mean-est-it-ub-lb} upper bound}\label{pf:mean-est-it-ub}
We prove the claim for general $q \in (0, 1]$.
To establish the upper bound, we rely on the following minimax optimal univariate estimator, due to~\cite[Theorem 1]{ma2026adaptive}.  For completeness, we provide a self-contained proof (using a different estimator) in Section~\ref{sec:univariate-mean-it-proof}.
\begin{corollary}[\cite{ma2026adaptive} Theorem 1]\label{cor:univariate-mean}
    Let $\epsilon \in [0, 1]$, $q \in (0, 1]$, $n \in \bN$,
    $\delta \in (0, 1)$, and define $\alpha = \sqrt{\frac{\log(\frac{1}{\delta})}{nq(1-\eps)}}$.
    Further, let $\theta_{\star} \in \bR$ and suppose that the observations
    $Z_1, \ldots, Z_n \overset{\mathrm{iid}}{\sim} R \in \mathcal{R}(P_{\theta_{\star}}, \epsilon, q)$.
    There exists universal, positive constants $C_1, C_2$ such that if
    \[n \geq \frac{C_1\log(\frac{1}{\delta})}{q(1-\epsilon)},\]
    then
    there exists an estimator $\widehat{\theta}: \bR_{\star}^{n} \rightarrow \bR$ which satisfies, with probability at least $1 - \delta$,  
    \begin{align*}
      \bigl \lvert \widehat{\theta} - \theta_{\star} \bigr \rvert
      \leq \frac{C_2 \log(1+\tau)}{\sqrt{\log(1 + \tau^2/\alpha^2)}}.
    \end{align*}
\end{corollary}

Equipped with this univariate estimator, we turn to the proof of our upper bound, which follows a standard approximation argument on an $\epsilon$-net.

To this end, let $\mathcal{V}$ denote a $\frac{1}{2}$--packing of the unit sphere $\bS^{d-1}$, which by, e.g.,~\cite[Corollary 4.2.11]{vershynin2025high} exists and satisfies the cardinality bound $\lvert \mathcal{V} \rvert \leq 5^{d}$.  Now, for each unit vector $v \in \mathcal{V}$, let $\{Z_i^v\}_{i \in [n]}$ denote the projection of the $i$th observation $Z_i$ onto $v$, noting that if $Z_i = \star^{d}$, then $Z_i^v = \star$.  Importantly, if $\law(Z_i) \in \mathcal{R}(P_{\theta_{\star}}, \epsilon, q)$, it follows that $\law(Z_i^v) \in \mathcal{R}(P_{\theta_{\star}^{\top}v}, \epsilon, q)$.  

Next, for each $v \in \cN$, let $\widehat{\theta}_v$ denote the univariate estimator implicit in Corollary~\ref{cor:univariate-mean}.  Armed with each of these univariate estimators, we define our multivariate mean estimator $\widehat{\theta} \in \bR^d$ to be any element
\begin{align} \label{def:multivariate-mean-eps-net-est}
\widehat{\theta} \in \argmin_{\theta \in \bR^d} \max_{v \in \cN} \; \bigl \lvert \widehat{\theta}^{\top} v - \widehat{\theta}_v \bigr \rvert,  
\end{align}
where for each $v \in \cN$, $\widehat{\theta}_v$ denotes the estimator in Corollary~\ref{cor:univariate-mean}.  For parameters $n, q, \epsilon, \delta$, let $r(n, q, \epsilon, \delta)$ denote the corresponding error rate in the corollary.  Next, take $\delta' = \delta/(5^d)$ and let $\Omega_{v}$ denote the event that
\begin{align}
    \label{ineq:univariate-bound-mean-eps-net}
\bigl \lvert \widehat{\theta}_v - \theta_{\star}^{\top} v \bigr \rvert \leq r\bigl(n, q, \epsilon, \delta'\bigr).
\end{align}
Note that by Corollary~\ref{cor:univariate-mean}, $\mathbb{P}(\Omega_v^c) \leq \delta/(5^d)$, whence with $\Omega := \bigcup_{v \in \cN} \Omega_v$, an application of the union bound implies that $\mathbb{P}(\Omega) \geq 1 - \delta$.

Now, working on the event $\Omega$, we have
\begin{align*}
    \bigl \| \widehat{\theta} - \theta_{\star} \bigr \|_2 &\overset{(a)}{\leq} 2 \cdot \max_{v \in \cN}\; \bigl(\widehat{\theta} - \theta_{\star} \bigr)^{\top}v \leq 2 \cdot \max_{v \in \cN}\;\bigl \lvert \widehat{\theta}^{\top}v - \widehat{\mu}_v \bigr \rvert + 2 \cdot \max_{v \in \cN}\; \bigl \lvert \widehat{\mu}_v - \theta_{\star}^{\top}v \bigr \rvert \\
    &\overset{(b)}{\leq} 4 \cdot \max_{v \in \cN}\; \bigl \lvert \widehat{\mu}_v - \theta_{\star}^{\top}v \bigr \rvert \leq 4r(n,q,\epsilon, \delta'),
\end{align*}
where step $(a)$ follows from an identical argument to the approximation in~\cite[Lemma 4.4.1]{vershynin2025high}, step $(b)$ follows from the definition of $\widehat{\theta}$ in Eq.~\eqref{def:multivariate-mean-eps-net-est} and the final inequality follows from Eq.~\eqref{ineq:univariate-bound-mean-eps-net}.

Finally, we see that
\[
r(n, q, \epsilon, \delta') \leq \frac{C_2 \log(1+\tau)}{\sqrt{\log\left(1 + \frac{\tau^2 nq(1-\epsilon)}{d+\log(\frac{1}{\delta})}\right)}},
\]
which completes the proof. \hfill $\qed$

\subsection{Covariance estimation: Proof of Theorem~\ref{thm:cov-est-it-ub-lb} upper bound}\label{sec:proof-thm-cov-est-it-ub}
We prove the claim for general $q \in (0, 1]$.
We devise a two-step estimator:
with one half of the data, we compute the unnormalized empirical covariance $\wt{\Sigma}$, scale the remaining data by $M$,
then finally combine 1-dimensional variance estimates of the scaled data over many directions, as in the mean estimation algorithm.

We define
\begin{align*}
  \wt{\Sigma} = \frac{2}{n} \sum_{i=1}^{n/2} Z_i Z_i^{\top} \1\{Z_i \neq \star\}.
\end{align*}
Observe that $\Sigma_\star^{-\frac{1}{2}}Z \1\{Z \neq \star\}$
is a $O(1)$-subgaussian random vector, so by the concentration of i.i.d.\ subgaussian random vectors
(see e.g.~\cite[Theorem 4.7.1]{vershynin2025high}) and our assumption that $n \gtrsim d+\log(1/\delta)$,
we have that with probability at least $1 - \delta/2$,
\begin{align*}
  \norm{\Sigma_\star^{-\frac{1}{2}}\wt{\Sigma}\Sigma_\star^{-\frac{1}{2}} - \Sigma_\star^{-\frac{1}{2}}\E[ZZ^T\1\{Z \neq \star\}] \Sigma_\star^{-\frac{1}{2}}}_{\rm op}
  \leq \frac{1}{2} \norm{\Sigma_\star^{-\frac{1}{2}} \E[ZZ^T\1\{Z \neq \star\}] \Sigma_\star^{-\frac{1}{2}}}_{\rm op}
  \leq \frac{1}{2}.
\end{align*}
Furthermore, our contamination model implies that
\begin{align*}
  q(1-\epsilon)\Sigma_\star \preceq \E[ZZ^T\1\{Z \neq \star\}] \preceq (q(1-\epsilon) + \epsilon)\Sigma_\star,
\end{align*}
and so combining the above two displays, we have that with probability at least $1 - \delta/2$,
\begin{align*}
  \frac{q(1-\epsilon)}{2} \Sigma_{\star} \preceq \wt{\Sigma} \preceq 2(q(1-\epsilon) + \epsilon)\Sigma_{\star}.
\end{align*}
We condition on this event for the remainder of the proof.
Letting $M = \tilde{\Sigma}^{-\frac{1}{2}}$,
we have by rearranging the above display that
\begin{align}
  \frac{1}{2(q(1-\epsilon) + \epsilon)} I_d
  \preceq M \Sigma_\star M
  \preceq \frac{2}{q(1-\epsilon)} I_d.
  \label{eq:tildeSigma-guarantee}
\end{align}

Now define the scaled data $\tilde{Z}_i = M Z_{i + n/2}$
for $i \in \{1, \ldots, n/2\}$, over which we will apply a univariate
covariate estimator.
For this, we require the following lemma, whose proof we defer to Section~\ref{sec:proof-lem-univariate-cov-ub}.

\begin{lemma} \label{lem:univariate-cov-ub}
  Let $\epsilon \in [0, 1]$, $q \in (0, 1]$, $n \in \bN$, and $\delta \in (0, 1)$,
  and define $\alpha = \sqrt{\frac{\log(\frac{1}{\delta})}{q(1-\epsilon)n}}$.
  Further, let $\sigma_{\star} \in \bR_{++}$, let $P_{\sigma} = \sigma^2 \chi^2$
  denote the scaled $\chi$-squared distribution, and suppose that the observations
  $Z_1, \ldots, Z_n \overset{\mathrm{iid}}{\sim} R \in \mathcal{R}(P_{\sigma_{\star}}, \epsilon, q)$.
  There exists a pair of universal, positive constants $C_1, C_2$ such that if
  \[n \geq \frac{C_1 \log(\frac{1}{\delta})}{q(1-\epsilon)},\]
  then there exists an estimator $\widehat{\sigma}: \bR_{\star}^{n} \rightarrow \bR$ which satisfies, with probability at least $1 - \delta$,  
  \begin{align*}
    \bigl \lvert \widehat{\sigma}^2 - \sigma_{\star}^2 \bigr \rvert \leq \frac{C_2\log(1+\tau)}{\log(1+\tau/\alpha)}\cdot \sigma_{\star}^2.
  \end{align*}
\end{lemma}

Continuing, let $\cN$ denote a $\beta$--net of the unit sphere $\bS^{d-1}$, for $\beta = \frac{1}{16\sqrt{1+\tau}}$,
which has size $e^{O(d\log(1/\beta))}$.
For each $v \in \cN$, let $\widehat{\sigma}_v^2$ be the estimator from
Lemma~\ref{lem:univariate-cov-ub} applied to observations
$\{\tilde{Z}_i \cdot v\}_{i \in [n/2]}$ and failure probability $\delta/(2|\cN|)$.
Applying Lemma~\ref{lem:univariate-cov-ub} for each $v \in \cN$ and applying a union bound,
we have with probability at least $1 - \delta/2$, for all $v \in \cN$,
\begin{align*}
  \bigl \lvert \widehat{\sigma}_v^2 - v^{\top} M \Sigma_\star M v \bigr \rvert
  \leq \frac{C_2\log(1+\tau)}{\log\left(1+\tau \sqrt{\frac{nq(1-\epsilon)}{d\log(1/\beta) + \log(8/\delta)}}\right)} \cdot v^{\top} M \Sigma_\star M v.
\end{align*}
Next, we define $\widehat{\Sigma}$ as any $\Sigma$ such that for all $v \in \cN$,
\begin{align*}
  |\hat{\sigma}_v^2 - v^{\top} M \Sigma M v|
  \leq \underbrace{\left[\frac{C_2\log(1+\tau)}
    {\log\left(1+\tau \sqrt{\frac{nq(1-\epsilon)}{d\log(1/\beta) + \log(8/\delta)}}\right)}\right]}_{=:\gamma} \cdot v^{\top} M \Sigma M v,
\end{align*}
which is feasible as $\Sigma_\star$ is one such matrix.

Our choice of $\beta$ satisfies $\log(1/\beta) \asymp 1 + \log(1+\tau)$.
so it suffices to establish that $L(\widehat{\Sigma}, \Sigma) = O(\gamma)$.
We can assume that $\gamma \leq \frac{1}{4}$, as otherwise, we can instead take $\widehat{\Sigma} = 0$
which satisfies $L(\widehat{\Sigma}, \Sigma_\star) = 1 \leq 4\gamma$.

By the guarantees of our estimator, we have for all $v \in \cN$ that
\begin{align*}
\bigl \lvert v^{\top} M \widehat{\Sigma}M v - v^{\top} M \Sigma_\star M v \bigr \rvert
\leq \gamma v^\top M \widehat{\Sigma} M v + \gamma v^{\top} M \Sigma_\star M v.
\end{align*}
Re-arranging this we obtain that for all $v \in \cN$,
\begin{align*}
  \frac{1-\gamma}{1+\gamma} v^\top M \Sigma_\star M v
  \leq v^\top M\widehat{\Sigma}M v
  \leq \frac{1+\gamma}{1-\gamma} v^\top M \Sigma_\star M v.
\end{align*}
For $\gamma \leq \frac{1}{4}$, we have $\frac{1-\gamma}{1+\gamma} \geq 1 - 3\gamma$ and $\frac{1+\gamma}{1-\gamma} \leq 1 + 3\gamma$,
and so the above guarantee implies that for all $v \in \cN$,
\begin{align*}
  |v^\top M(\widehat{\Sigma} - \Sigma_\star)M v| \leq 3\gamma v^\top M \Sigma_\star M v.
\end{align*}

To turn this into a bound on relative spectral norm, first observe that
\begin{align*}
  L(\widehat{\Sigma}, \Sigma_\star)
  = \sup_{v \neq 0} \frac{|v^\top (\Sigma_\star^{-\frac{1}{2}}\widehat{\Sigma}\Sigma_\star^{-\frac{1}{2}} - I_d)v|}{v^\top v}
  = \sup_{v \neq 0} \frac{|v^\top (M\widehat{\Sigma}M - M \Sigma_\star M)v|}{v^\top M \Sigma_\star M v}.
\end{align*}
To bound this,
we first define the matrix
$E = \Sigma_\star^{-\frac{1}{2}}\widehat{\Sigma}\Sigma_\star^{-\frac{1}{2}} - I_d$
and note that both $L(\widehat{\Sigma}, \Sigma_\star) = \norm{E}_{\rm op}$
and that for any $v \in \bR^d$,
\begin{align*}
  \frac{v^\top (M\widehat{\Sigma}M - M \Sigma_\star M)v}{v^\top M \Sigma_\star M v}
  = \tilde{v}^\top E \tilde{v},
\end{align*}
where $\tilde{v} = \big\|\Sigma_\star^{\frac{1}{2}} M v\big\|^{-1} \Sigma_\star^{\frac{1}{2}} M v$.
For any $v \in \bS^{d-1}$, we know there exists $u \in \cN$ such that $\norm{v - u} \leq \beta$.
Let $\tilde{v}$ and $\tilde{u}$ be defined as above for $v$ and $u$, respectively.
Then we have
\begin{align}
  \left|
    \frac{v^\top (M\widehat{\Sigma}M - M \Sigma_\star M)v}{v^\top M \Sigma_\star M v}
    - \frac{u^\top (M\widehat{\Sigma}M - M \Sigma_\star M)u}{u^\top M \Sigma_\star M u}
  \right|
  &= |\tilde{v}^\top E \tilde{v} - \tilde{u}^\top E \tilde{u}| \notag \\
  &\leq |\tilde{v}^\top E \tilde{v} - \tilde{u}^\top E \tilde{v}| + |\tilde{u}^\top E \tilde{v} - \tilde{u}^\top E \tilde{u}| \notag \\
  &\leq 2\norm{E}_{\rm op} \cdot \norm{\tilde{v} - \tilde{u}}_2. \label{eq:cov-it-ub-net-approximation}
\end{align}
Furthermore, we can bound the distance between $\tilde{v}$ and $\tilde{u}$ as
\begin{align*}
  \big\|\tilde{v} - \tilde{u}\big\|
  &= \Bigg\|\frac{\Sigma_\star^{\frac{1}{2}}Mv}{\big\|\Sigma_\star^{\frac{1}{2}}Mv\big\|} - \frac{\Sigma_\star^{\frac{1}{2}}Mu}{\big\|\Sigma_\star^{\frac{1}{2}}Mu\big\|}\Bigg\| \\
  &\leq \Bigg\|\frac{\Sigma_\star^{\frac{1}{2}}Mv}{\big\|\Sigma_\star^{\frac{1}{2}}Mv\big\|}
      - \frac{\Sigma_\star^{\frac{1}{2}}Mv}{\big\|\Sigma_\star^{\frac{1}{2}}Mu\big\|}\Bigg\|
    + \Bigg\|\frac{\Sigma_\star^{\frac{1}{2}}Mv}{\big\|\Sigma_\star^{\frac{1}{2}}Mu\big\|}
      - \frac{\Sigma_\star^{\frac{1}{2}}Mu}{\big\|\Sigma_\star^{\frac{1}{2}}Mu\big\|}\Bigg\| \\
  &\leq \frac{2\big\|\Sigma_\star^{\frac{1}{2}}M(v-u)\big\|}{\big\|\Sigma_\star^{\frac{1}{2}}M u\big\|}
  \leq 4\beta \sqrt{\frac{q(1-\epsilon) + \epsilon}{q(1-\epsilon)}} = 4\beta\sqrt{1+\tau},
\end{align*}
where the final inequality follows from \eqref{eq:tildeSigma-guarantee}.
Plugging this bound into \eqref{eq:cov-it-ub-net-approximation}, and taking a supremum over $v \in \bS^{d-1}$, we obtain that
\begin{align*}
  L(\widehat{\Sigma}, \Sigma_\star)
  \leq 3\gamma + 8\beta\sqrt{1+\tau} \cdot L(\widehat{\Sigma}, \Sigma_\star) \cdot \beta
  \Longrightarrow
  L(\widehat{\Sigma}, \Sigma_\star) \leq 6\gamma,
\end{align*}
where the final implication follows from our choice of $\beta$.\qed

\subsection{Univariate upper bounds} \label{sec:it-ub-univariate}
Our information-theoretically optimal mean estimation algorithm follows the
well-worn path of applying optimal univariate estimators over a covering.
For our univariate estimators, we will use the minimum Kolmogorov distance
estimator introduced by~\cite{ma2024estimation}.  
is a minimum distance estimator using the Kolmogorov distance
$d_K: \mathcal{P}(\bR_{\star}) \times \mathcal{P}(\bR_{\star})$ defined as
\[
  d_K(R_1, R_2) := \sup_{A \in \mathcal{A}}\; \bigl \lvert R_1(A) - R_2(A)\bigr \rvert,
\]
where we take the collection $\mathcal{A}$ to denote all upper half intervals 
\[
  \mathcal{A} := \{(-\infty, r] \mid r \in \bR\}. 
\]
We note in passing that this definition of the Kolmogorov distance is a symmetrized variant of the usual Kolmogorov distance.

Given a collection of probability measures $\mathcal{R}' \subseteq \mathcal{P}(\bR_{\star})$ and a single probability measure $R \in \mathcal{P}(\bR_{\star})$, we define the projection in Kolmogorov distance onto the set $\mathcal{R}'$ as
\[
  d_K(R, \mathcal{R}') := \inf_{R' \in \mathcal{R}'}\; d_K(R, R').
\]
Next, given observations $Z_1, \ldots, Z_n \in \bR_{\star}$, we let
\[
  \widehat{R}_n := \frac{1}{n} \sum_{i=1}^{n} \delta_{Z_i},
\]
denote the empirical distribution of the observations.  

Importantly, the empirical distribution converges in Kolmogorov distance to its
population version at the parametric rate, as the following lemma establishes.
\begin{lemma}
\label{lem:DKW-Kolmogorov}
Let $\epsilon \in [0, 1)$, $q \in (0, 1]$, and $P \in \mathcal{P}(\bR)$.
Suppose that
$Z_1, \ldots, Z_n \overset{\mathrm{iid}}{\sim} R \in \mathcal{R}(P, \epsilon, q)$
and let $\widehat{R}_n$ denote the empirical distribution of the observations.
For any $\delta \in (0, 1)$, it holds that
\[
  \mathbb{P}\biggl\{d_K\Bigl(\widehat{R}_n, \mathcal{R}(P, \epsilon, q)\Bigr)
  > 3 \sqrt{\frac{\{q (1-\epsilon) + \epsilon\} \log(\frac{1}{\delta})}{n}} \biggr\}
  \leq \delta.
\]
\end{lemma}
We omit the proof of this lemma as its proof follows an identical sequence of steps leading to~\cite[Ineq. (56)]{ma2024estimation} (in particular applying the DKW inequality~\cite{massart1990tight} in conjunction with Bernstein's inequality~\cite{vershynin2025high}).  Note that here we depart slightly from the setting of~\cite{ma2024estimation} in that we consider the symmetrized Kolmogorov distance, but a careful inspection of their proof shows that the result remains unchanged.

\subsubsection{Univariate mean estimation: Proof of Corollary~\ref{cor:univariate-mean}} \label{sec:univariate-mean-it-proof}
Let $P_{\theta} = \mathsf{N}(\theta, 1)$.  We consider the minimum distance estimator $\widehat{\theta}$, defined to be any element
    \begin{align} \label{def:univariate-kolmogorov-mean}
    \widehat{\theta}(Z_1, \ldots, Z_n) \in \argmin_{\theta \in \bR}\, d_K\bigl(\widehat{R}_n, \mathcal{R}(P_{\theta}, \epsilon, q)\bigr).
    \end{align}
    In words, the estimator $\widehat{\theta} \in \bR$ is taken as the parameter which contains the closest realizable contamination to the empirical distribution $\widehat{R}_n$.
    
Applying Lemma~\ref{lem:DKW-Kolmogorov} in conjunction with the triangle inequality, we deduce that
\[
d_K(\mathcal{R}(P_{\widehat{\theta}}, \epsilon, q), \mathcal{R}(P_{\theta_{\star}}, \epsilon, q)) \leq 6\sqrt{\frac{\{q(1 - \epsilon) + \epsilon\}\log(\frac{1}{\delta})}{n}}
= 6q(1-\epsilon)\alpha\sqrt{1+\tau}.
\]
Note further that, for any $r \in \bR$, 
\begin{align*}
  d_K\bigl(\mathcal{R}(P_{\theta}, \epsilon, q), \mathcal{R}(P_{\theta_{\star}}, \epsilon, q)\bigr) &\geq q(1-\epsilon) \cdot \mathbb{P}(G + \theta \leq \theta_{\star} - r) - \bigl\{q(1 - \epsilon) + \epsilon\bigr\} \cdot \mathbb{P}(G + \theta_{\star} \leq \theta_{\star} - r)\\
  &= q(1-\epsilon) \cdot \big[\underbrace{\bar{\Phi}(r - t) - (1 + \tau)\,  \bar{\Phi}(r)}_{=: \; g(t;r)} \big],
\end{align*}
where above we have $G \sim \mathsf{N}(0,1)$ and recall that $\bar{\Phi}$ is the survival function of the standard Gaussian.
The final relation follows by taking $\theta = \theta_{\star} - t$.  The function $g: \bR_{+} \times \bR \rightarrow \bR$ defined above is strictly increasing in its first argument. Hence, 
\begin{align*}
\bigl \lvert \widehat{\theta} - \theta_{\star} \bigr \rvert &\leq \sup\, \biggl\{t \geq 0: \lvert \theta - \theta_{\star} \rvert \leq t \, \text{ and } \, d_K\bigl(\mathcal{R}(P_{\theta}, \epsilon, q), \mathcal{R}(P_{\theta_{\star}}, \epsilon, q)\bigr) \leq 6q(1-\epsilon)\alpha\sqrt{1+\tau}\biggr\} \\
&\leq \sup\, \biggl\{t \geq 0: \; \sup_{r \in \bR}\, g(t; r) \leq 6\alpha\sqrt{1+\tau}\biggr\}\\
&= \inf\,\biggl\{t \geq 0:  \; \sup_{r \in \bR}\, g(t; r) > 6\alpha\sqrt{1+\tau}\biggr\}.
\end{align*}

Re-arranging and invoking the decreasing nature of $\bar{\Phi}^{-1}$, we see that for any choice of $r > 0$, the condition is satisfied if 
\begin{align} \label{ineq:exact-condition-t-mean}
t \geq r - \bar{\Phi}^{-1}\biggl((1 + \tau) \bar{\Phi}(r) + 6\alpha \sqrt{1+\tau}\biggr).
\end{align}

We next consider several cases for the value of $\tau$, select values of $r$ and upper bound the right-hand side of~\eqref{ineq:exact-condition-t-mean}.  To facilitate the analysis, we introduce the shorthand
\[
    \Delta := \tau \bar{\Phi}(r) + 6\alpha\sqrt{1+\tau}.
\]
Applying the mean value theorem in conjunction with Lemma~\ref{lem:Gaussian-Q-prop}, we deduce that
\begin{align} \label{ineq:expansion-Q-inv}
\bar{\Phi}^{-1}\bigl(\bar{\Phi}(r) + \Delta\bigr) &\geq r - \Delta \cdot \sup_{x' \in [\bar{\Phi}(r), \bar{\Phi}(r) + \Delta]}\, \frac{1}{\phi(\bar{\Phi}^{-1}(x'))} \nonumber\\
&\geq r - \Delta \cdot \sup_{x' \in [\bar{\Phi}(r), \bar{\Phi}(r) + \Delta]}\, \frac{1}{x' \bar{\Phi}^{-1}(x')} \nonumber \\
&\geq r - \frac{\Delta}{\bar{\Phi}(r) \cdot \bar{\Phi}^{-1}\bigl(\bar{\Phi}(r) + \Delta\bigr)}.
\end{align}

We proceed with by consider several cases on $\tau$.
Before doing so, first observe that our assumption $n \geq \frac{C_1 \log(\frac{1}{\delta})}{q(1-\epsilon)}$ implies that
$\alpha^2 \leq 1/C_1$.

\paragraph{Case 1: $\tau \leq 125\alpha$.} 
In this case, we set $r=2$.  Note that by assumption, $\alpha^2 \leq 1/C_1$.
Expanding the definition of $\Delta$ and taking $C_1$ sufficiently large yields the conclusion 
$
\Delta \leq C\alpha
$
for a sufficiently large constant $C$.
Applying the first inequality in~\eqref{ineq:expansion-Q-inv}, we find that
\[
\bar{\Phi}^{-1}\bigl(\bar{\Phi}(r) + \Delta\bigr) \geq r - \Delta \cdot \sup_{x' \in [\bar{\Phi}(r), \bar{\Phi}(r) + \Delta]}\, \frac{1}{\phi(\bar{\Phi}^{-1}(x'))} \geq r - \frac{\Delta}{\phi(r)} = r - \Delta \sqrt{2\pi} e^{r^2/2} = r - \Delta e\sqrt{2\pi},
\]
where the second inequality follows because the Gaussian PDF $\phi$ is decreasing on non-negative reals.  It thus follows that any separation
\[
t \geq r - \biggl(r - C e\sqrt{2\pi} \alpha \biggr) = C e\sqrt{2\pi}\alpha
\asymp \frac{\log(1+\tau)}{\sqrt{\log(1+\tau^2/\alpha^2)}}
\]
suffices, where the final equality (up to constant factors) holds because $\tau$ and $\tau/\alpha$ are both bounded by a constant in this regime.

\paragraph{Case 2: $125 \alpha < \tau \leq 8$.}  In this case, we set
\[
r = \sqrt{\log\bigl(\tau/\{6\sqrt{2\pi(1+\tau)} \alpha\}\bigr)} \quad\Longrightarrow\quad
6\sqrt{2\pi(1+\tau)} \cdot \alpha e^{r^2} = \tau,
\]
and note $r \geq 1$ for this range of $\tau$.
Moreover, applying the standard Mills' ratio lower bound (see \cref{lem:mills}) yields
$\bar{\Phi}(r) \geq \frac{e^{-r^2/2}}{r\sqrt{2\pi}} \geq \frac{e^{-r^2}}{\sqrt{2\pi}}$.
It thus follows that
\begin{align*} %
\frac{\Delta}{\bar{\Phi}(r)} \leq \tau + 6\sqrt{2\pi(1 + \tau)} \cdot  \alpha e^{r^2} = 2\tau.
\end{align*}
On the other hand, applying the first inequality in~\eqref{ineq:expansion-Q-inv} in conjunction with the decreasing nature of the Gaussian PDF $\phi$, we deduce that
\[
\bar{\Phi}^{-1}\bigl(\bar{\Phi}(r) + \Delta\bigr) \geq r - \frac{\Delta}{\phi(r)} \geq r - \frac{\tau}{r} + 6\sqrt{2\pi(1 + \tau)} \cdot  \alpha e^{r^2} \geq r, 
\]
where the penultimate inequality follows from the Mills' ratio upper bound $\bar{\Phi}(x) \leq \phi(x)/x$ (see \cref{lem:mills}) and the final inequality follows because $r \geq 1$. 
Combining the inequalities in the previous two displays with the final inequality in~\eqref{ineq:expansion-Q-inv}, we find that any separation
\[
t \geq r - \Bigl(r - \frac{2\tau}{r}\Bigr) \geq \frac{2\tau}{r} = \frac{2\tau}{\sqrt{\log(\tau/\alpha) - \log(6\sqrt{2\pi(1+\tau)})}},
\]
suffices.
To conclude this case, we have that $\log(1+\tau) \asymp \tau$ because $\tau < 8$ and
\[\log(\tau/\alpha) - \log(6\sqrt{2\pi(1+\tau)}) \alpha) \asymp \log(\tau/\alpha) \asymp \log(1 + \tau^2/\alpha^2)\] because
$\frac{\tau}{\alpha} > 125$ and $6\sqrt{2\pi(1+\tau)} < \frac{125}{2}$.

\paragraph{Case 3: $8 < \tau \leq \frac{1}{20}\alpha^{-1/4}$.}
In this case, we take
\[
r = \sqrt{\log(1/6\sqrt{2\pi}\alpha)}.
\]
Note that under this setting, taking $C_1$ large enough implies that $r \geq 1$, which combined with the Mills' ratio lower bound (see \cref{lem:mills}) gives
\begin{align*}
  \bar{\Phi}(r) \geq \frac{r}{r^2+1} \frac{e^{-r^2/2}}{\sqrt{2\pi}}
  \geq \frac{1}{2\sqrt{2\pi}}e^{-r^2}
  = 3\alpha.
\end{align*}
Then because $\tau > 8$, we find that
\[
6 \alpha \sqrt{1+\tau} \leq \tau \bar{\Phi}(r).
\]
In turn, this implies that
\[
\bar{\Phi}^{-1}\bigl(\bar{\Phi}(r) + \Delta\bigr) \geq \bar{\Phi}^{-1}\bigl(2\tau \bar{\Phi}(r)\bigr).
\]
We claim that with the shorthand $\zeta := \bar{\Phi}^{-1}\bigl(2\tau \bar{\Phi}(r)\bigr)$, the following inequality holds, deferring its proof to the end of the case,
\begin{align}
    \label{ineq:Q-inv-lb-tighter}
    \zeta^2 \geq r^2 - 2\log(2\tau) + 2\log\biggl(\frac{r}{\zeta + 1/\zeta}\biggr).
\end{align}
We further note the weaker inequality
\[
\zeta \geq \frac{1}{2}\sqrt{\log\biggl(\frac{1}{2 \tau \bar{\Phi}(r)}\biggr)} \overset{(a)}{\geq} \frac{1}{2}\sqrt{\log\biggl(\frac{e^{r^2/2}}{2 \tau}\biggr)} = \frac{1}{2}\sqrt{\frac{r^2}{2} + \log\bigl(1/(2\tau)\bigr)} \overset{(b)}{\geq} \frac{r}{4}
\geq \frac{1}{4}
\]
where above, step $(a)$ follows from the bound $\bar{\Phi}(r) \leq e^{-r^2/2}$ and step (b) follows since for the considered range of $\tau$ and the setting of $r$, $2 \tau \leq e^{r^2/4}$.  

On the other hand, we have that $\zeta \leq r$ because $2\tau > 1$ and $\bar{\Phi}^{-1}$ is decreasing.
Combining these inequalities, we deduce that 
\[
\frac{r}{\zeta + 1/\zeta} \geq \frac{r}{r + 4} \geq \frac{1}{5}.
\]
Substituting this back into the inequality~\eqref{ineq:Q-inv-lb-tighter} then yields
\[
\zeta^2 \geq r^2 - 2\log(10\tau) \geq r^2/2,
\]
where the final inequality follows because $10\tau \leq e^{r^2/4}$.  Hence, we find that
\[
r - \sqrt{r^2 - 2\log(\tau/4)} \leq \frac{4\log(\tau/4)}{r},
\]
and consequently deduce that any separation
\[
t \geq \frac{4\log(\tau/4)}{r} = \frac{4\log(\tau/4)}{\sqrt{\log(1/6\sqrt{2\pi}\alpha)}},
\]
suffices.
Then we have that $\log(1+\tau) \asymp \log(\tau/4)$ because $\tau > 8$ and $\log(1/6\sqrt{2\pi} \alpha) \asymp \log(\tau/\alpha)$ because,
taking $C_1$ large enough, $\log(1/\alpha)$ is bounded below by a sufficiently large constant and $\log 8 \leq \log(\tau) \leq \frac{1}{4}\log(1/\alpha) - \log(20)$.

It remains to establish the inequality~\eqref{ineq:Q-inv-lb-tighter}.

\medskip
\noindent \underline{Proof of the inequality~\eqref{ineq:Q-inv-lb-tighter}:}
Note that $2 \tau \bar{\Phi}(r) \leq \frac{1}{2}$, so that $\zeta \geq 0$.  We thus apply the standard Mills' ratio bounds (see \cref{lem:mills}) to obtain the sandwich relation
\[
\frac{\zeta}{\zeta^2 + 1} \phi(\zeta) \leq \bar{\Phi}(\zeta) = 2 \tau \bar{\Phi}(r) \leq \frac{2 \tau}{r} \phi(r).
\]
It thus follows that
\[
\frac{1}{\zeta + 1/\zeta} e^{-\zeta^2/2} \leq \frac{2\tau}{r} e^{-r^2/2}.
\]
Hence, taking logarithms and re-arranging yields the conclusion
\[
\zeta^2 \geq r^2 - 2\log(2 \tau) + 2\log\biggl(\frac{r}{\zeta + 1/\zeta}\biggr),
\]
as desired.

\paragraph{Case 4: $\tau > \frac{1}{20}\alpha^{-\frac{1}{4}}$}
In this case, we take a slightly modified version of the previous estimator.
Rather than the Kolmogorov distance, we take conditional Kolmogorov distance
\begin{align*}
  d_K(P, Q \mid \bR) := d_K(P(\cdot \mid \bR), Q(\cdot \mid \bR)),
\end{align*}
then we take $\widehat{\theta}$ to be
\begin{align} \label{def:univariate-cond-kolmogorov-mean}
  \widehat{\theta}(Z_1, \ldots, Z_n) \in\argmin_{\theta \in \bR}\,
    d_K\bigl(\widehat{R}_n, \mathcal{R}(P_{\theta}, \epsilon, q) \mid \bR\bigr).
\end{align}
Similar to the previous estimator, we can easily bound the conditional Kolmogorov
distance between $P_{\widehat{\theta}}$ and $P_{\theta_\star}$.
Let $\bar{q} = R(\bR) \geq q(1-\epsilon)$ and let $\bar{R} = R(\cdot \mid \bR)$
and let $S \subseteq [n]$ be the random set of indices
$S = \{i \in [n] \mid Z_i \neq \star\}$.
Observe that $|S| \sim \Bin(n, \bar{q})$.
By Bernstein's inequality and our assumption that
$n \gtrsim \frac{\log(1/\delta)}{q(1-\epsilon)}$, we have $|S| \geq \bar{q}n/2$
with probability at least $1-\delta/2$.
Now conditioned on $S$, observe that $\{Z_i\}_{i \in S} \simiid \bar{R}^{\otimes |S|}$,
so we can apply (e.g.) the DKW inequality~\cite{massart1990tight} to obtain that with probability at least $1-\delta/2$,
\begin{align*}
  d_K(\hat{R}_n, \cR(P_{\theta_\star}, \epsilon, q) \mid \bR)
  \leq d_K(\hat{R}_n, R \mid \bR)
  = d_K\left(\frac{1}{|S|} \sum_{i \in S} \delta_{Z_i}, \bar{R}\right)
  \leq \sqrt{\frac{\log(\frac{1}{\delta})}{2|S|}}.
\end{align*}
Thus by a union bound and triangle inequality, we have with probability at least $1-\delta$ that
\begin{align}
d_K(\mathcal{R}(P_{\widehat{\theta}}, \epsilon, q), \mathcal{R}(P_{\theta_{\star}}, \epsilon, q) \mid \bR)
\leq 2\sqrt{\frac{\log(\frac{1}{\delta})}{q(1-\epsilon)n}}
\leq 2 C_1^{-\frac{1}{2}}.
\label{eq:cond-kolmogorov-guarantee-mean}
\end{align}

We will be done once we establish a lower bound on
$d_K(\mathcal{R}(P_{\theta_{\star}}, \epsilon, q), \mathcal{R}(P_{\theta}, \epsilon, q) \mid \bR)$
for any $\theta$ such that $|\theta - \theta_{\star}|$ is sufficiently large.
Let $R \in \mathcal{R}(P_{\theta_{\star}}, \epsilon, q)$ and
$R' \in \mathcal{R}(P_{\theta}, \epsilon, q)$ be arbitrary contaminations.
By translating and negating the observations if necessary,
we may assume without loss of generality that $\theta_{\star} = 0$ and $\theta > 0$.
Then by the definition of our contamination set, we have for all $S \subseteq \bR$
\[
  q(1-\eps) \Phi(S) \leq R(S) \leq (q(1-\eps) + \epsilon) \Phi(S)
  = (1+\tau) q(1-\epsilon) P_\theta(S),
\]
and similar for $R'$ with $\Phi(S-\theta)$ in place of $\Phi(S)$.
Noting that $(a, b) \mapsto a/(a+b)$ is increasing in $a$ and decreasing in $b$,
we have that for any $r \in \bR_+$,
\begin{align*}
  d_K(R, R' \mid \bR)
  &\geq R(\cdot \leq r \mid \bR) - R'(\cdot \leq r \mid \bR) \\
  &= \frac{R(\cdot \leq r)}{R(\cdot \leq r) + R(\cdot > r)}
     - \frac{R'(\cdot \leq r)}{R'(\cdot \leq r) + R'(\cdot > r)} \\
  &= \frac{\Phi(r)}{\Phi(r) + (1+\tau)(1-\Phi(r))}
    - \frac{(1+\tau)\Phi(r-\theta)}{(1+\tau)\Phi(r-\theta) + (1-\Phi(r-\theta))} \\
  &= \frac{\Phi(r)(1-\Phi(r-\theta)) - (1+\tau)^2\Phi(r-\theta)(1-\Phi(r))}
    {[\Phi(r) + (1+\tau)(1-\Phi(r))][(1+\tau)\Phi(r-\theta) + (1-\Phi(r-\theta))]} \\
  &= \frac{\Phi(r)\Phi(\theta-r) - (1+\tau)^2\Phi(r-\theta)\Phi(-r)}
    {[\Phi(r) + (1+\tau)\Phi(-r)][(1+\tau)\Phi(r-\theta) + \Phi(\theta-r)]},
\end{align*}
where the final equality follows by applying the symmetry of the Gaussian distribution.
Now for $\theta$ such that $\Phi(-\theta/2) \leq \frac{1}{4}(1+\tau)^{-1}$,
by taking $r = \theta/2$, we have
\begin{align*}
  d_K(R, R' \mid \bR)
  &\geq \frac{\Phi(\theta/2)^2 - (1+\tau)^2\Phi(-\theta/2)^2}
    {[\Phi(\theta/2) + (1+\tau)\Phi(-\theta/2)]^2}
  \geq \frac{\frac{1}{4} - \frac{1}{16}}{(\frac{1}{2} + \frac{1}{4})^2}
  \geq 3C_1^{-\frac{1}{2}},
\end{align*}
where we used that $\Phi(\theta/2) \geq \frac{1}{2}$ and $\Phi(-\theta/2) \leq \frac{1}{4}(1+\tau)^{-1}$
in the second inequality and took $C_1$ sufficiently large in the final inequality.
By a Gaussian tail bound,
$\Phi(-\theta/2) \leq e^{-\theta^2/8}$, so
$\Phi(-\theta/2) \leq \frac{1}{4}(1+\tau)^{-1}$ 
holds for all $\theta \geq \sqrt{8\log(4(1+\tau))}$.

Therefore, by \cref{eq:cond-kolmogorov-guarantee-mean}, we have that 
with probability at least $1-\delta$,
\[
  |\widehat{\theta} - \theta_{\star}| \leq \sqrt{8\log(4(1+\tau))}
  \asymp \frac{\log(1+\tau)}{\sqrt{\log(1 + \tau^2/\alpha^2)}},
\]
where the final asymptotic equivalence holds because $\tau > \frac{1}{20}\alpha^{-1/4}$.

This proves the desired result in this case and concludes the proof of the corollary.
\qed

\subsubsection{Univariate variance estimation: Proof of Lemma~\ref{lem:univariate-cov-ub}} \label{sec:proof-lem-univariate-cov-ub}
We note that this proof is nearly identical to that of Corollary~\ref{cor:univariate-mean}, hence we omit several details already provided in that proof. 

We consider the minimum distance estimator $\widehat{\sigma}$, defined to be any element
\begin{align} \label{def:univariate-kolmogorov-var}
\widehat{\sigma}(Z_1, \ldots, Z_n) \in \sup\left\{\sigma \in \mathbf{R}_{++} \mid d_K\bigl(\widehat{R}_n, \mathcal{R}(P_{\sigma}, \epsilon, q)\bigr) \leq 3q(1-\epsilon)\alpha\sqrt{1+\tau}\right\}.
\end{align}
Applying Lemma~\ref{lem:DKW-Kolmogorov} in conjunction with the triangle inequality, we deduce that
\[
\hat{\sigma} \geq \sigma_\star
\qquad\textrm{and}\qquad
d_K(\mathcal{R}(P_{\widehat{\sigma}}, \epsilon, q), \mathcal{R}(P_{\sigma_{\star}}, \epsilon, q)) \leq 6q(1-\epsilon)\alpha\sqrt{1+\tau},
\]
so it remains to upper bound $\hat{\sigma}$.
Note further that, for any $r \in \mathbf{R}$, 
\begin{align*}
2 d_K\bigl(\mathcal{R}(P_{\sigma}, \epsilon, q), \mathcal{R}(P_{\sigma_{\star}}, \epsilon, q)\bigr) &\geq q(1-\epsilon) \cdot \mathbb{P}(\sigma^2 G^2 \geq r) - \bigl\{q(1 - \epsilon) + \epsilon\bigr\} \cdot \mathbb{P}(\sigma_{\star}^2 G^2 \geq r)\\
 &= q(1 - \epsilon) \, \bar{\Phi}\biggl(\sqrt{\frac{r}{\sigma^2}}\biggr) - \bigl\{q(1-\epsilon) + \epsilon\bigr\}\,  \bar{\Phi}\biggl(\sqrt{\frac{r}{\sigma_{\star}^2}}\biggr),
\end{align*}
where above we have used the notation $G \sim \mathsf{N}(0,1)$ and let $\bar{\Phi}$ denote the survival function of the standard Gaussian (that is, $\bar{\Phi}(x) = \mathbb{P}(G \geq x)$).   Letting $\sigma^2 = \sigma_{\star}^2 (1 + t)^2$ and re-scaling, we see that
        \begin{align*}
 d_K\bigl(\mathcal{R}(P_{\sigma}, \epsilon, q), \mathcal{R}(P_{\sigma_{\star}}, \epsilon, q)\bigr) &\geq q(1 - \epsilon) \bigg[\underbrace{\bar{\Phi}\biggl(\frac{r}{1 + t}\biggr) - (1+\tau)  \bar{\Phi}(r)}_{=:\; g(t,r)}\bigg].
\end{align*}
Note that the function $g$ defined above is strictly increasing in its first argument.  It thus follows that with
\begin{align*}
\psi(\epsilon, q, n, \delta) &:= \sup\, \biggl\{t \geq 0: \; \sup_{r \in \mathbf{R}}\, g(t; r) \leq 6\alpha\sqrt{1+\tau}\biggr\} \\
&= \inf\,\biggl\{t \geq 0:  \; \sup_{r \in \mathbf{R}}\, g(t; r) > 6\alpha\sqrt{1+\tau}\biggr\},
\end{align*}
we have
\[
|\widehat{\sigma}^2 - \sigma_\star^2| \leq \sigma_{\star}^2 \cdot \bigl[2 \psi(\epsilon, q, n, \delta) + \{\psi(\epsilon, q, n, \delta)\}^2\bigr].
\]
Re-arranging and invoking the decreasing nature of $\bar{\Phi}^{-1}$, we see that for any choice of $r > 0$, the condition is satisfied if 
\begin{align} \label{ineq:exact-condition-t-var}
t \geq r \cdot \biggl\{\bar{\Phi}^{-1}\biggl((1 + \tau) \bar{\Phi}(r) + 6\alpha\sqrt{1+\tau}\biggr)\biggr\}^{-1} - 1.
\end{align}
Let us now bound this term in three cases, depending on the value of $\tau$.  

\paragraph{Case 1: $\tau \leq 125\alpha$.} 
In this case, we set $r=2$.  The same sequence of steps as in Case 1 of the proof of Corollary~\ref{cor:univariate-mean} yields
\[
\bar{\Phi}^{-1}\biggl((1 + \tau) \bar{\Phi}(r) + 6\alpha\sqrt{1+\tau}\biggr) \geq r - \Delta e\sqrt{2\pi}, 
\]
where $\Delta = \tau \bar{\Phi}(r) + 6\alpha\sqrt{1+\tau}$.
Hence, since $\Delta \lesssim \alpha$ is smaller than a sufficiently small constant,
\[
\frac{r}{r - \Delta e \sqrt{2\pi}} - 1 \leq \frac{\Delta e \sqrt{2\pi}}{r - \Delta e \sqrt{2\pi}} \leq C\alpha,
\]
and so any separation
\[
\psi(\epsilon, q, n, \delta) = C\alpha
\]
suffices.  In turn, for large enough $C_1$, this implies that $\psi(\epsilon, q, n, \delta) < 1$, so that in this case
\[
    \bigl \lvert \widehat{\sigma}^2 - \sigma_{\star}^2 \bigr \rvert \leq C \alpha \asymp \frac{\log(1+\tau)}{\log(1+\tau/\alpha)},
\]
where the last equality is because $\tau \lesssim \alpha \lesssim 1$.

\paragraph{Case 2: $125\alpha < \tau \leq 8$.}  As in Corollary~\ref{cor:univariate-mean}, in this case, we set
\[
r = \sqrt{\log(\tau/12\alpha)},
\]
and deduce the lower bound
\[
\bar{\Phi}^{-1}\biggl((1 + \tau) \bar{\Phi}(r) + 6\alpha\sqrt{1+\tau}\biggr) \geq r - \frac{2\tau}{r/2}. 
\]
Hence, since
\[
\frac{r}{r - \frac{2\tau}{r/2}} - 1 \leq \frac{C\tau}{r^2} < 1,
\]
we deduce that in this case 
\[
    \bigl \lvert \widehat{\sigma}^2 - \sigma_{\star}^2 \bigr \rvert \leq \frac{C\tau}{\log(\tau/12\alpha)}
    \asymp \frac{\log(1+\tau)}{\log(1+\tau/\alpha)},
\]
where the last inequality is because $\alpha \lesssim \tau \lesssim 1$.

\paragraph{Case 3: $8 < \tau \leq \frac{1}{20}\alpha^{-1/4}$.}
As in Corollary~\ref{cor:univariate-mean}, in this case, we take
\[
r = \sqrt{\log(1/12\sqrt{10}\alpha)}.
\]
Under this setting, following the same steps as in Corollary~\ref{cor:univariate-mean}, we find that
\[
\bar{\Phi}^{-1}\biggl((1 + \tau) \bar{\Phi}(r) + 6\alpha\sqrt{1+\tau}\biggr) \geq \sqrt{r^2 - 2\log(\tau/4)}.
\]
Thus, since
\[
\frac{r}{\sqrt{r^2 - 2\log(\tau/4)}} -1 \leq \frac{C\log(\tau/4)}{r^2},
\]
we deduce that in this case, 
\[
    \bigl \lvert \widehat{\sigma}^2 - \sigma_{\star}^2 \bigr \rvert \leq \frac{C\log \tau}{\log(1/12\sqrt{10}\alpha)}
    \asymp \frac{\log(1+\tau)}{\log(1+\tau/\alpha)},
\]
where the last equality is because $\alpha \lesssim 1$ and $1 \lesssim \tau \lesssim \alpha^{-1/4}$.

\paragraph{Case 4: $\tau > \frac{1}{20}\alpha^{-1/4}$.}  In this case, instead of our previous estimator, we simply take $\widehat{\sigma} = 0$.
Because $\alpha^2 \leq \frac{1}{C_1}$ for sufficiently large $C_1$ and $\tau > \frac{1}{20}\alpha^{-1/4}$, we have that $\lvert \sigma_{\star}^2 - \widehat{\sigma}^2 \rvert = \sigma_{\star}^2 \leq \frac{C_2 \log(1+\tau)}{\log(1+\tau/\alpha)} \sigma_\star^2$ by picking $C_2$ sufficiently large.

This proves the desired result in all cases.
\qed

\section{Proof of upper bound for linear regression}
\subsection{Linear regression: Proof of Theorem~\ref{thm:lin-reg-ub-lb} upper bound}\label{pf:lin-reg-ub}
We prove the claim for general $q \in (0, 1]$.
We show that our choice of estimator $\widehat{\theta}\ofk_n$ achieves the desired rate.
To do so, we introduce some key lemmas to control relevant quantities with high probability.
We defer the proofs of \cref{lem:grad_f_n-ub,lem:mu_n-lb} to the subsequent subsections.

\begin{lemma}\label{lem:grad_f_n-ub}
  Let $\alpha$ be as in Theorem~\ref{thm:lin-reg-ub-lb}, $m = q(1-\epsilon)n$,
  and suppose that $\delta \geq m^{-c}$ for some constant $c > 0$.
  With probability at least $1-O(\delta)$, it holds that
  \begin{align}
  \bigl \| \nabla F\ofk_n(\theta_\star) \bigr \|
    \lesssim q(1-\epsilon) k \sigma^{2k-1} \cdot &\Bigl\{\tau (2k-2)!! + \alpha (1+\sqrt{\tau})\sqrt{(4k-3)!!} \notag\\
    &\qquad + \gamma_{1} m^{-\frac{1}{4}} + \gamma_{2} m^{-\frac{1}{2}} + \gamma_3 m^{-1}\Bigr\},
    \label{eq:grad_f_n-ub}
  \end{align}
  where
  \begin{align*}
     \gamma_1 &= \alpha \log^{\frac{1}{2}}(m)\sqrt{(4k-3)!!} + \alpha (1 + \tau^{\frac{1}{4}})O(\log\{(1+\tau)m\})^{\frac{k}{2}}, \\
     \gamma_2 &= \alpha [\log^{\frac{1}{2}}(m)\sqrt{(4k-3)!!} + O(\log\{(1+\tau)m\})^{\frac{k}{2}+\frac{1}{4}}]
       + \sqrt{\tau} O(\log\{(1+\tau)m\})^{\frac{k}{2}}, \\
     \gamma_3 &= \log(m) (2k-2)!! + O(\log\{(1+\tau)m\})^{\frac{k+1}{2}}.
  \end{align*}
\end{lemma}

\begin{lemma}\label{lem:mu_n-lb}
  Let $\alpha$ be as in Theorem~\ref{thm:lin-reg-ub-lb}.  Suppose $\alpha^{-2} \gtrsim e^{10 k \log^2k}$.
  With probability at least $1-O(\delta)$, the empirical risk \(F\ofk_n\) is uniformly strongly convex:
  \begin{align*}
    \inf_{\theta \in \R^d} \inf_{v \in \bS^{d-1}} 
    v^T \nabla^2 F\ofk_n(\theta) v 
    \gtrsim q(1-\epsilon)\sigma^{2k-2}(2k+1)!!.
  \end{align*}
\end{lemma}

\begin{proof}[Proof of \cref{thm:lin-reg-ub-lb}]
Suppose first that we choose $k$ satisfying $e^{10 k \log^2k} \lesssim \alpha^{-2}$.
By Lemmas~\ref{lem:grad_f_n-ub}, we deduce that with probability at least $1-\delta/2$, $F_n\ofk$ is $\mu_n$--strongly convex.
From strong convexity of $F_{n}^{(k)}$, we obtain the inequality
\[
\bigl\{ \nabla F_n\ofk \bigl(\widehat{\theta}_n\ofk\bigr) - \nabla F\ofk_n(\theta_{\star}) \bigr\}^{\top} \bigl\{\widehat{\theta}_n\ofk - \theta_{\star}\bigr\} \geq \mu_n \cdot \| \widehat{\theta}_n\ofk - \theta_{\star} \|_2^2.
\]
Since by definition, $\nabla F_n\ofk \bigl(\widehat{\theta}_n\ofk\bigr) = 0$, applying the Cauchy--Schwarz inequality to the left-hand side and re-arranging yields
\begin{align*}
  \bigl \| \theta_\star - \widehat{\theta}\ofk_n \bigr \|_2
  \leq \frac{\bigl \| \nabla F\ofk_n(\theta_\star)\bigr\|_2}{\mu_n}.
\end{align*}
Applying Lemmas~\ref{lem:grad_f_n-ub} and~\ref{lem:mu_n-lb}, which hold with probability $1 - \delta$, then yields
\begin{align}
 \frac{1}{\sigma} \cdot \bigl\| \widehat{\theta}\ofk_n - \theta_\star\bigr \| &\lesssim  \frac{k \cdot \bigl\{\tau (2k-2)!! + \alpha (1+\sqrt{\tau})\sqrt{(4k-3)!!} + \gamma_{1} m^{-\frac{1}{4}} + \gamma_{2} m^{-\frac{1}{2}} + \gamma_3 m^{-1}\bigr\}}{(2k+1)!!}\nonumber\\
 &\lesssim \frac{\tau}{\sqrt{k}} + \frac{2^k \alpha (1+\sqrt{\tau})}{\sqrt{k}} + \frac{\gamma_{1} m^{-\frac{1}{4}} + \gamma_{2} m^{-\frac{1}{2}} + \gamma_3 m^{-1}}{(2k-1)!!},
  \label{eq:est-err-bd}
\end{align}
where the first two terms in the final step follow from Stirling's inequality (see \cref{fact:stirling}).
We analyze this bound in different cases for $\tau$.
To do so, let $C' \geq 1$ be a sufficiently large constant and $c < 1$ a sufficiently small constant that we will specify later.

\paragraph{Small $\tau$: $\tau \leq C'\alpha$.}
In this case we take $k = 1$ and have that $\alpha^{-2} \gtrsim e^{10 k\log^2 k}$ because $\alpha \leq 1/\sqrt{C}$ by our assumption $m \geq C (d+\log(\frac{1}{\delta}))$.
Because $\tau \leq C'\alpha \lesssim 1$, the error bound \eqref{eq:est-err-bd} simplifies to
\begin{align*}
  \frac{1}{\sigma} \cdot \bigl\| \widehat{\theta}\ofk_n - \theta_\star\bigr \|
  \lesssim \alpha \cdot \big(1 + m^{-\frac{1}{4}}\sqrt{\log m} + m^{-\frac{1}{2}}\log^{\frac{3}{4}}(m)\big) + m^{-1} \log m 
  \lesssim \alpha,
\end{align*}
where the final step follows because $\alpha \geq \frac{1}{\sqrt{m}}$.
In this regime, $\frac{\tau (1 \vee \log\log(\tau/\alpha))}{\sqrt{\log(1 + \tau^2/\alpha^2)}} \asymp \alpha$,
so this matches our desired rate.

\paragraph{Large $\tau$: $\tau \geq C'\alpha$.}
In this case, we take $k = \big\lfloor\frac{c\log(\tau/\alpha)}{(\log\log(\tau/\alpha))^2}\big\rfloor$.
By choosing $C'$ sufficiently large (as a function of $c$), we have that $k \geq 1$ because $\tau/\alpha \geq C'$.
This implies $|\log k| = \log k \leq \log\log(1/\alpha)$
and so $k \log^2 k \leq c\log(1/\alpha)$.
By picking $c = 0.1$, we have that $k$ is small enough to apply \cref{lem:mu_n-lb}.
Substituting this into \eqref{eq:est-err-bd}, we have
\begin{align}
  &\frac{1}{\sigma} \cdot \bigl\| \widehat{\theta}\ofk_n - \theta_\star\bigr \| \notag \\
  &\qquad\lesssim \frac{\tau \log\log(\tau/\alpha)}{\sqrt{\log(\tau/\alpha)}}
  + \frac{\alpha (\tau/\alpha)^{\frac{c\log 2}{(\log\log(\tau/\alpha))^2}} \log\log(\tau/\alpha)}{\sqrt{\log(\tau/\alpha)}}
  + \frac{\gamma_{1} m^{-\frac{1}{4}} + \gamma_{2} m^{-\frac{1}{2}} + \gamma_3 m^{-1}}{(2k-1)!!} \label{eq:large-tau-err}
\end{align}
Enlarging $C'$ if necessary, we see that $\frac{c\log 2}{(\log\log(\tau/\alpha))^2} \leq 1$, which implies that the second term is dominated by the first.
We now bound $\frac{\gamma_1}{m^{\frac{1}{4}}(2k-1)!!}$.
We have
\begin{align*}
  \frac{\gamma_1}{m^{\frac{1}{4}}(2k-1)!!}
  &\asymp \frac{\alpha 2^k \log^{\frac{1}{2}}(m)}{m^{\frac{1}{4}}} + \frac{\alpha O(\log(m))^{\frac{k}{2}}}{m^{\frac{1}{4}}(2k-1)!!}.
\end{align*}
Because $(2k-1)!! = O(k)^k$, we see that $\frac{O(\log m)^{\frac{k}{2}}}{(2k-1)!!} \leq e^{O(\sqrt{\log m})} \leq m^{\frac{1}{8}}$ for $m \geq C''$, where $C''$ is a sufficiently large constant.
Now because $\alpha^{-2} \leq m$, we have that $m^{\frac{1}{8}} \gtrsim \sqrt{\log(1/\alpha)}$.
Applying these bounds to the above display, we see that as long as $m \geq C'' \vee e^{(8\log 2)k}$,
\[\frac{\gamma_1}{m^{\frac{1}{4}}(2k-1)!!} \lesssim \frac{\alpha}{\sqrt{\log(1/\alpha)}}.\]
Note that by taking $C_1$ to be a sufficiently large constant, the assumption that $\alpha \leq 1/\sqrt{C_1}$ implies that $\frac{8c\log 2}{(\log\log(1/\alpha))^2} \leq 2$.
This in turn implies $m \geq e^{(8\log 2)k}$ because \[e^{(8\log 2)k} = \alpha^{-\frac{8c\log 2}{(\log\log(1/\alpha))^2}} \leq \alpha^{-2} \leq m.\]
The other two terms in \eqref{eq:large-tau-err} follow similarly, so we have the error bound
\begin{align*}
  \frac{1}{\sigma} \cdot \bigl\| \widehat{\theta}\ofk_n - \theta_\star\bigr \|
  \lesssim \frac{\tau \log\log(\tau/\alpha)}{\sqrt{\log(\tau/\alpha)}} + \frac{\alpha}{\sqrt{\log(1/\alpha)}} \asymp \frac{\tau (1 \vee \log\log(\tau/\alpha))}{\sqrt{\log(1 + \tau^2/\alpha^2)}},
\end{align*}
as desired.
\end{proof}

It remains to prove the above lemmas.  We introduce some notation that will be useful in both lemmas.  In particular, by the usual decomposition, we know that there exists $\Omega^\mcar \sim \Bern(q), \Omega^\mnar \in \{0, 1\}, B \sim \Bern(\epsilon), (X, Y) \sim P$ such that $R$ is the law of
\begin{align*}
  R = \mathsf{Law}(Z), \quad \text{ and } \quad Z = (X, Y) \ostar \bigl\{(1-B) \Omega^\mcar + B \Omega^\mnar\bigr\},
\end{align*}
where $\Omega^\mcar$, $B$, and $(Z, \Omega^\mnar)$ are mutually independent.  Notice that this slightly generalizes the development in Section~\ref{sec:lin-reg} where we specify $q=1$.
Moreover, $Y$ can be generated as $X \cdot \theta_\star + \sigma g$ for $g \sim \mathsf{N}(0, 1)$ independent of $X$.
Observe that $Z \in \mathbb{R}^{d+1}$ if and only if $(1-B)\Omega^\mcar + B\Omega^\mnar = 1$.  Let $\omega^\mcar_i, \omega^\mnar_i, b_i$ be the realizations of these random variables in the sample.

\subsubsection{Proof of Lemma~\ref{lem:grad_f_n-ub}}
Expanding out $\nabla F^{(k)}_n(\theta_\star)$ and using the fact that $y_i = X_i \cdot \theta_\star + g_i$, we have
\begin{align*}
  \nabla F^{(k)}_n(\theta_\star)
  &= \frac{2k \sigma^{2k-1}}{n}\sum_{i=1}^n ((1-b_i)\omega^\mcar_i + b_i \omega^\mnar_i) g_i^{2k-1} X_i
\end{align*}
so that
\begin{align}
  \norm{\nabla F^{(k)}_n(\theta_\star)}
  &\leq 2k \sigma^{2k-1}\norm{\frac{1}{n}\sum_{i : b_i = 0} \omega^\mcar_i g_i^{2k-1} X_i} + 2k \sigma^{2k-1}\norm{\frac{1}{n} \sum_{i : b_i = 1} \omega^\mnar_i g_i^{2k-1} X_i}
  \label{eq:grad-ub}
\end{align}
We will be done once we bound each vector norm in \eqref{eq:grad-ub} with probability at least $1-O(\delta)$ by
\begin{align}
  B := q(1-\epsilon) \cdot \left\{\tau (2k-2)!! + \alpha (1+\sqrt{\tau})\sqrt{(4k-3)!!} + \gamma_{1} m^{-\frac{1}{4}} + \gamma_{2} m^{-\frac{1}{2}} + \gamma_3 m^{-1}\right\},
  \label{eq:norm-bd-goal}
\end{align}
where we recall that
\begin{align*}
 \gamma_1 &= \alpha \log^{\frac{1}{2}}(m)\sqrt{(4k-3)!!} + \alpha (1 + \tau^{\frac{1}{4}})O(\log\{(1+\tau)m\})^{\frac{k}{2}},  \\
 \gamma_2 &= \alpha [\log^{\frac{1}{2}}(m)\sqrt{(4k-3)!!} + O(\log \{(1+\tau)m\})^{\frac{k}{2}+\frac{1}{4}}]
   + \sqrt{\tau} O(\log\{(1+\tau)m\})^{\frac{k}{2}}, \\
 \gamma_3 &= \log(m) (2k-2)!! + O(\log\{(1+\tau)m\})^{\frac{k+1}{2}}.
\end{align*}

\paragraph{Bounding the first term of \eqref{eq:grad-ub}.}
Observe that $(1-b_i)\omega_i^{\mcar} \sim \Bern(q(1-\epsilon))$ are i.i.d.\ and independent of $\{(g_i, X_i)\}$.
Thus, conditioning on $n_0 = \sum_{i=1}^n (1-b_i)\omega_i$, we can assume without loss of generality that the first $n_0$ samples are the samples for which $(1-b_i)\omega_i = 1$ and then bound
\begin{align*}
  \norm{\frac{1}{n}\sum_{i=1}^{n_0} g_i^{2k-1} X_i}.
\end{align*}
Now we condition on $\bg = \{g_i\}$ and observe that $\frac{1}{n}\sum_{i=1}^{n_0} g_i^{2k-1} X_i \sim N(0, \sigma^2(\bg) I_d )$, where
\begin{align*}
  \sigma^2(\bg) = \frac{1}{n^2}\sum_{i=1}^{n_0} g_i^{4k-2}.
\end{align*}
Thus a standard Gaussian tail bound yields that with probability at least $1-\delta$,
\begin{align}
  \norm{\frac{1}{n}\sum_{i=1}^{n_0} g_i^{2k-1} X_i}
  \leq \frac{1}{n} \cdot \sqrt{\sum_{i=1}^{n_0} g_i^{4k-2}} \cdot \left(\sqrt{d} + \sqrt{2\log(\tfrac{1}{\delta})}\right) \label{eq:norm-bd-lr}.
\end{align}
We have by \cref{lem:gaussian-power-concentration-truncation} and the assumption that $\delta \geq m^{-c}$ that with probability at least $1-\delta$,
\begin{align}
  \sum_{i=1}^{n_0} g_i^{4k-2}
  &\leq n_0 (4k-3)!! + \sqrt{n_0 2^{4k-3} \log(2/\delta) \log(2n_0/\delta)^{2k-1}} \notag \\
  &\lesssim n_0 (4k-3)!! + [O(\log n_0 + \log m)]^k \sqrt{n_0}. \label{eq:g-bd-lr}
\end{align}
Now because $n_0 \sim \Bin(n, q(1-\epsilon))$ we apply Bernstein's inequality to obtain that with probability at least $1-e^{-\Omega(m)} \geq 1-\delta$, $n_0 \leq 2m$.

Putting it all together, we deduce that with probability at least $1-O(\delta)$,
\begin{align*}
  \norm{\frac{1}{n}\sum_{i=1}^{n_0} g_i^{2k-1} X_i}
  &\lesssim \frac{\sqrt{d + \log(\frac{1}{\delta})}}{n} \cdot \left(\sqrt{m q(1-\epsilon) (4k-3)!!}
    + O(\log m)^{\frac{k}{2}}m^{\frac{1}{4}} \right) \\
  &= q(1-\epsilon)\alpha
    \left(\sqrt{(4k-3)!!} + m^{-\frac{1}{4}}O(\log m)^{\frac{k}{2}}\right).
\end{align*}
This is bounded by $O(B)$ because $\gamma_1 \geq \alpha \cdot O(\log m)^{\frac{k}{2}}$.

\paragraph{Bounding the second term of \eqref{eq:grad-ub}.}
Again observe that $\{b_i\}$ is independent of everything else, so we can $n_1 = \sum_{i=1}^n b_i$ and assume without loss of generality that the first $n_1$ values of $b_i$ are equal to $1$.
Then we have
\begin{align*}
  \Bigl \| \frac{1}{n} \sum_{i : b_i = 1} \omega^\mnar_i g_i^{2k-1} X_i\Bigr\|
  &= \frac{1}{n} \sup_{v \in \bS^{d-1}} \left| \sum_{i=1}^{n_1} \omega^\mnar_i g_i^{2k-1} \cdot X_i ^{\top} v \right| \\
  &\leq \frac{1}{n} \sup_{v \in \bS^{d-1}} \sum_{i=1}^{n_1} |g_i^{2k-1}| \cdot |X_i^{\top} v|.
\end{align*}
Now we condition on $\bg = \{g_i\}_{i=1}^{n}$ and let $a_i = |g_i^{2k-1}|$.

We define $T_v := \sum_{i=1}^{n_1} a_i |X_i^{\top} v|$
and let $\cN$ be a $1/2$-packing of $\bS^{d-1}$ of size at most $5^d$.
Then for any $v \in \bS^{d-1}$, let $v' \in \cN$ satisfy $\norm{v - v'} \leq 1/2$ and observe that
\begin{align*}
  |T_v - T_{v'}|
  \leq \sum_{i=1}^{n_1} a_i \cdot \big| |X_i^{\top} v| - |X_i^{\top} v'| \big|
  \leq \sum_{i=1}^{n_1} a_i \cdot |X_i^{\top} (v - v')|
  \leq \norm{v - v'} \cdot \sum_{i=1}^{n_1} a_i |X_i^{\top} u|
\end{align*}
where $u \in \bS^{d-1}$ is such that $v = v' + \norm{v - v'} u$.
We thus have that
\begin{align*}
  T_v \leq T_{v'} + \norm{v - v'} \sum_{i=1}^{n_1} a_i |X_i^{\top} u|
  \leq \sup_{u \in \cN} T_u + \frac{1}{2} \sup_{u \in \bS^{d-1}} T_u.
\end{align*}
Taking a supremum over $v \in \bS^{d-1}$ and rearranging yields
$\sup_{v \in \bS^{d-1}} T_v \leq 2 \sup_{v \in \cN} T_v$.

Since $X_i \simiid \mathsf{N}(0, I_d)$, we have that for each $v \in \cN$ that $\E[|X_i^{\top} v|] = \sqrt{2/\pi}$ and that $T_v - \sqrt{2/\pi}\sum_{i=1}^{n_1} a_i$ is $\left(\sum_{i=1}^{n_1} a_i^2\right)$-subgaussian.
Thus,
\begin{align*}
  \P\left(\sup_{v \in \cN} T_v - \sqrt{\frac{2}{\pi}}\sum_{i=1}^{n_1} a_i \geq t\right)
  \leq |\cN| \sum_{v \in \cN} \P\left(T_v - \sqrt{\frac{2}{\pi}} \sum_{i=1}^{n_1} a_i \geq t\right)
  \leq 5^d \exp\left(-\frac{t^2}{2\sum_{i=1}^{n_1} a_i^2}\right).
\end{align*}
Picking $t = \sqrt{2(d\log 5 + \log(\frac{1}{\delta}))\sum_{i=1}^{n_1} a_i^2}$, we have with probability at least $1-\delta$ that
\begin{align*}
  \sup_{v \in \cN} T_v
  \leq \sqrt{\frac{2}{\pi}}\sum_{i=1}^{n_1} a_i + \sqrt{\sum_{i=1}^{n_1} a_i^2(d\log 5 + \log(\tfrac{1}{\delta}))}.
\end{align*}
Assembling the pieces, we have conditional on $n_1$ and $\bg$ that with probability at least $1-\delta$,
\begin{align}
  \norm{\frac{1}{n} \sum_{i : b_i = 1} \omega^\mnar_i g_i^{2k-1} X_i}
  \lesssim \frac{1}{n} \left(\sum_{i=1}^{n_1} |g_i|^{2k-1} + \sqrt{\sum_{i=1}^{n_1} g_i^{4k-2}(d + \log(\tfrac{1}{\delta}))}\right).
  \label{eq:bi1-ub}
\end{align}
Again by Lemma~\ref{lem:gaussian-power-concentration-truncation}, we have with probability at least $1-\delta$ that
\begin{align*}
  \sum_{i=1}^{n_1} |g_i|^{2k-1}
  &\leq n_1 (2k-2)!! + O(\log (n_1/\delta))^{\frac{k}{2}}\sqrt{n_1} \quad \text{ and }\\
  \sum_{i=1}^{n_1} |g_i|^{4k-2} 
  &\leq n_1 (4k-3)!! + O(\log (n_1/\delta))^{k}\sqrt{n_1}.
\end{align*}
Recalling that $n_1 \sim \Bin(n, \epsilon)$, we have that with probability at least $1-\delta$ that $n_1 \lesssim \epsilon n + \log (\frac{1}{\delta}) = \tau m + O(\log m)$.
Now we condition on all these events, which occur simultaneously with probability at least $1-O(\delta)$.
Note that $\log(n_1/\delta) = \log(\tau m + O(\log m)) + O(\log m)) \asymp \log((1+\tau) m)$.
Then the first term in \eqref{eq:bi1-ub} is upper bounded by 
\begin{align*}
  \frac{1}{n} \sum_{i=1}^{n_1} |g_i|^{2k-1}
  &\lesssim \frac{1}{n} \cdot \big[(\tau m + \log m) (2k-2)!! + O(\log \{(1+\tau)m\})^{\frac{k}{2}}(\sqrt{\tau m} + \sqrt{\log m})\big] \\
  &\asymp q(1-\epsilon) \Big[\tau (2k-2)!! + m^{-1}\log(m) (2k-2)!! \\
  &\hspace{72pt}+ m^{-\frac{1}{2}}O(\log \{(1+\tau)m\})^{\frac{k}{2}}\sqrt{\tau} + m^{-1}O(\log \{(1+\tau)m\})^{\frac{k+1}{2}}\Big].
\end{align*}
Collecting terms with the same exponent on $m$ and recalling the definitions of $\gamma_2$ and $\gamma_3$, we have that the above display is $O(B)$ because
\begin{align*}
  \gamma_2 m^{-\frac{1}{2}} &\geq m^{-\frac{1}{2}}O(\log \{(1+\tau)m\})^{\frac{k}{2}}\sqrt{\tau} \quad\textrm{and} \\
  \gamma_3 m^{-1} &\geq m^{-1}(\log(m) (2k-2)!! + O(\log \{(1+\tau)m\})^{\frac{k+1}{2}}).
\end{align*}
The second term in \eqref{eq:bi1-ub} is bounded above by
\begin{align*}
  \frac{1}{n}\sqrt{\sum_{i=1}^{n_1} g_i^{4k-2}(d + \log(\tfrac{1}{\delta}))}
  &\lesssim
  \frac{\sqrt{d + \log(\frac{1}{\delta})}}{n}
  \cdot \Big[\sqrt{\tau m (4k-3)!!} + \log^{\frac{1}{2}}(m)\sqrt{(4k-3)!!} \\
  &\hspace{80pt}+ O(\log \{(1+\tau)m\})^{\frac{k}{2}}[(\tau m)^{\frac{1}{4}} + \log^{\frac{1}{4}}(m)]\Big] \\
  &\asymp\ q(1-\epsilon)\alpha \Big[\sqrt{\tau (4k-3)!!} + m^{-\frac{1}{2}}\log^{\frac{1}{2}}(m) \sqrt{(4k-3)!!} \\
  &\hspace{80pt}+ O(\log\{(1+\tau)m\})^{\frac{k}{2}}[m^{-\frac{1}{4}}\tau^{\frac{1}{4}} + m^{-\frac{1}{2}}\log^{\frac{1}{4}}(m)]\Big].
\end{align*}

Similar to the previous term, this display is also $O(B)$ because
\begin{align*}
  \gamma_1 m^{-\frac{1}{4}} &\geq \alpha \tau^\frac{1}{4} O(\log((1+\tau)m))^{\frac{k}{2}} m^{-\frac{1}{4}} \quad \textrm{and}\\
  \gamma_2 m^{-\frac{1}{2}} &\geq \alpha(\log^{\frac{1}{2}}(m)\sqrt{(4k-3)!!} + \log^{\frac{1}{4}}(m)O(\log((1+\tau)m))^{\frac{k}{2}}) m^{-\frac{1}{2}}.
\end{align*}
We have thus bounded all terms by $O(B)$, proving the claim.
\qed

\subsubsection{Proof of Lemma~\ref{lem:mu_n-lb}}

Expanding out $\nabla^2 F^{(k)}_n$, we have for any unit vector $v$ that
\begin{align*}
v^T \nabla^2F\ofk_n(\theta) v
&= \frac{(2k)(2k-1)}{n}\sum_{i=1}^n ((1-b_i)\omega^\mcar_i + b_i \omega^\mnar_i)
  \bigl(\bx_i^\top(\theta_\star-\theta)+\sigma g_i\bigr)^{2k-2}
  (\bx_i^\top v)^2 \nonumber\\
&\geq \frac{(2k)(2k-1)}{n}\sum_{i=1}^n (1-b_i)\omega^\mcar_i
  \bigl(\bx_i^\top(\theta_\star-\theta)+\sigma g_i\bigr)^{2k-2}
  (\bx_i^\top v)^2,
\end{align*}
where the inequality is because the summands are nonnegative.
Defining the random set $\cI = \{i \in [n] \mid (1-b_i)\omega^\mcar_i = 1\}$, we equivalently have
\begin{align}
  v^T \nabla^2 F\ofk_n(\theta) v
  \geq \frac{(2k)(2k-1)}{n}\sum_{i \in \cI} \bigl(\bx_i^\top(\theta_\star-\theta)+\sigma g_i\bigr)^{2k-2}
  (\bx_i^\top v)^2 \label{eq:hess-Fk}
\end{align}
We condition on $\cI$ and let $n_0 = |\cI|$.
Because $((1-b_i)\omega_i^\mcar)_{i \in [n]}$ is independent of $((\bx_i,g_i))_{i \in [n]}$, we have that condition on $\cI$,
$((\bx_i, g_i))_{i \in \cI}$ are $n_0$ i.i.d.\ samples from $\mathsf{N}(0, I_{d}) \otimes \mathsf{N}(0, \sigma^2)$.

We now introduce a finite sequence of noise bins indexed by $\ell = 0, \ldots, L$, where $t_\ell = (1 + \tfrac{1}{2k})^\ell$ and $L$ is the unique integer such that $2\sqrt{k}\log(3k) \leq t_L < (1 + \tfrac{1}{2k})[2\sqrt{k}\log(3k)]$.
Then for each $\ell \in [L]$, $\theta \in \R^d$, and unit vector $v$, we define the subsets
\begin{align*}
S_{\ell,\theta,v} := \bigl\{ i \in \cI :\, \bx_i \cdot (\theta_\star-\theta) \geq 0, |\bx_i \cdot v| \geq \sqrt{\pi/8}, \text{ and } g_i \in [t_{\ell-1}, t_\ell) \bigr\}.
\end{align*}
For $i \in S_{\ell, \theta, v}$, we have the lower bounds
\begin{align*}
(y_i - \bx_i^\top \theta) = (\bx_i^\top \theta_\star + \sigma g_i) - \bx_i^\top \theta \geq \sigma t_{\ell-1} \quad \text{and} \quad (\bx_i^\top v)^2 \geq \frac{\pi}{8} > \frac{1}{4}.
\end{align*}
Now observe that for each $\theta$ and $v$, each index $i \in \cI$ is in $S_{\ell,\theta,v}$ for at most one $\ell \in [L]$.
Applying these lower bounds to \eqref{eq:hess-Fk}, we have that
\begin{align*}
v^\top \nabla^2 F\ofk_n(\theta) v
&\geq \frac{(2k)(2k-1)}{n} \sum_{\ell=1}^{L} \sum_{i \in S_{\ell,\theta,v}} (\sigma t_\ell)^{2k-2} (\bx_i^\top v)^2
\geq \frac{k^2\sigma^{2k-2}}{2n} \sum_{\ell=1}^{L} |S_{\ell,\theta,v}| t_{\ell-1}^{2k-2}.
\end{align*}

To prove the claim, we find uniform lower bounds on $|S_{\ell,\theta,v}|$ that hold with probability at least $1-O(\delta)$. 
For $\ell \in [L]$, let $p_{\ell} := \mathbb{P}(G \in [t_{\ell-1}, t_{\ell}))$, where $G \sim \mathsf{N}(0, 1)$.
Observe that for $\ell \in [L]$,
\begin{align*}
  p_\ell
  &= \frac{1}{\sqrt{2\pi}} \int_{t_{\ell-1}}^{t_\ell} e^{-t^2/2} dt
  \geq \frac{t_\ell - t_{\ell-1}}{\sqrt{2\pi}} e^{-t_\ell^2/2}
  = \frac{1}{(2k+1)\sqrt{2\pi}}t_\ell e^{-t_\ell^2/2}.
\end{align*}
Now because $t \mapsto te^{-t^2/2}$ is decreasing for $t \geq 1$ and because  for all $\ell \in [L]$, $t_\ell \leq 3\sqrt{k}\log(3k)$ we have that
\begin{align}
  p_\ell
  &\geq \frac{3\sqrt{k}\log(3k)}{(2k+1)\sqrt{2\pi}}e^{-\frac{9}{2}k\log^2(3k)}
  \gtrsim e^{-5k\log^2 k}.
  \label{eq:pl-lb}
\end{align}

We now have the following claim which we prove later.
\begin{claim}\label{claim:Sltv-lb}
  Let $C > 0$ be a sufficiently large constant.
  With probability at least $1-\delta$,
  it holds simultaneously for all $\ell \in [L]$, $\theta \in \R^d$, and $v \in \bS^{d-1}$ that
  \begin{align}
    |S_{\ell,\theta,v}| \geq \frac{p_\ell n_0}{4} - C\sqrt{(d + \log(\frac{1}{\delta}))n_0}. \label{eq:Sltv-lb}
  \end{align}
\end{claim}

We condition on the event defined in the claim.
Now recall that $n_0 \sim \Bin(n, q(1-\epsilon))$.
Then, $\alpha^2 \leq 1/C_1$ implies $q(1-\epsilon)n \gtrsim \log(\frac{1}{\delta})$ and so with probability at least $1-\delta$, we have that $q(1-\epsilon)n/2 \leq n_0 \leq 2q(1-\epsilon)n$.
Further conditioning on this event, we have by combining this with \eqref{eq:Sltv-lb} that for all $\ell \in [L]$,
\[
|S_{\ell,\theta,v}| \geq \frac{p_{\ell}q(1-\epsilon)n}{8} - C\sqrt{2(d + \log(\frac{1}{\delta}))q(1-\epsilon)n}.
\]
Now, applying the assumption $\alpha^{-2} \geq C' e^{10 k\log^2 k}$ combined with \eqref{eq:pl-lb} implies
\begin{align*}
  C\sqrt{2(d + \log(\frac{1}{\delta}))q(1-\epsilon)n} \leq \frac{C\sqrt{2}}{C'} \cdot e^{5 k\log^2 k}n \leq \frac{p_{\ell}q(1-\epsilon)n}{16},
\end{align*}
where the final inequality follows by taking $C'$ to be a sufficiently large constant.
Thus, it holds for all $\ell \in [L]$ that
\[
|S_{\ell,\theta,v}| \geq \frac{q(1-\eps)p_\ell n}{16}.
\]
Consequently, for all $\theta \in \R^d$ and all unit vectors $v$, we have
\begin{align*}
  v^T \nabla^2 F\ofk_n(\theta) v \geq \frac{(1-\eps)k^2\sigma^{2k-2}}{16} \sum_{\ell=1}^{L} p_\ell t_{\ell-1}^{2k-2}.
\end{align*}

Recalling that $t_{\ell-1} = (1+\frac{1}{2k})^{-1}$ and $p_\ell = \P(G \in [t_{\ell-1}, t_\ell))$, we have
\begin{align*}
  \sum_{\ell=1}^{L} p_\ell t_{\ell-1}^{2k-2}
  &= \left(1+\frac{1}{2k}\right)^{-1} \sum_{\ell=1}^{L} p_\ell t_{\ell}^{2k-2}
  = \left(1+\frac{1}{2k}\right)^{-1} \sum_{\ell=1}^{L} \int_{t_{\ell-1}}^{t_\ell} t_{\ell}^{2k-2} \phi(t) dt
  \\&\geq \frac{2}{3}\sum_{\ell=1}^{L} \int_{t_{\ell-1}}^{t_\ell} t^{2k-2} \phi(t) dt
  = \frac{2}{3} \int_{1}^{t_L} t^{2k-2} \phi(t) dt.
\end{align*}

The last expression is a truncated Gaussian moment and Lemma~\ref{lemma:trunc-kmom-lb} gives that $\int_{1}^{t_L} t^{2k-2} \phi(t) dt \geq \frac{(2k-3)!!}{3}$ because $t_L \geq 2\sqrt{k}\log(3k)$.
Combining the pieces we have for all $\theta \in \R^d$ and all unit vectors $v$ that
\begin{align*}
  v^T \nabla^2 F\ofk_n(\theta) v \geq \frac{q(1-\eps) k^(2k-3)!!\sigma^{2k-2}}{36}
  \asymp q(1-\eps)\sigma^{2k-2}(2k+1)!!,
\end{align*}
proving the claim. \qed

\paragraph{Proof of the Claim~\ref{claim:Sltv-lb}.}
The proof of this bound follows from a standard VC dimension argument.
For each $\ell \in [L]$, we define the function classes
\begin{align*}
\mathcal{F}_\ell := \Bigl\{(\bx, g) \mapsto \mathbf{1}\Bigl\{ \bx \cdot (\theta_{\star} - \theta)&\geq 0,\, |\bx \cdot v| \geq \sqrt{\frac{\pi}{8}}, \text{ and } t_{\ell} \leq g < t_{\ell + 1}\Bigr\}: v \in \bS^{d-1},\, \theta \in \theta_\star + \bS^{d-1}\Bigr\}
\end{align*}
and also define $\mathcal{F} = \cup_{\ell \in [L]} \mathcal{F}_\ell$.
Observe that $\mathcal{F}$ is a subclass of the 3-fold intersection of the following function classes:
\begin{align*}
  \mathcal{F}_1 &= \Bigl\{(x, g) \mapsto \mathbf{1}\{ \bx \cdot (\theta_{\star} - \theta) \geq 0\}: \theta \in \theta_\star + \bS^{d-1} \Bigr\} \\
  \mathcal{F}_2 &= \Bigl\{(x, g) \mapsto \mathbf{1}\{ |\bx \cdot v| \geq \sqrt{\pi/8}\}: v \in \bS^{d-1}\Bigr\} \\
  \mathcal{F}_3 &= \Bigl\{(x, g) \mapsto \mathbf{1}\{ a \leq g < b \}: a,b \in \mathbb{R}\Bigr\}.
\end{align*}

Common VC-dimension calculations give that $\mathrm{VC}(\mathcal{F}_1) \leq d+1$, $\mathrm{VC}(\mathcal{F}_2) \leq 2(d+1)$ and $\mathrm{VC}(\mathcal{F}_3)\leq 2$.
We can then apply Lemma~\ref{lem:vc-dim-intersection} to deduce that $\mathrm{VC}(\mathcal{F}) \leq 12\log(18)(d+1) \asymp d$.
Using this VC-dimension bound, we can bound the expected worst-case deviation using \cite[Theorem 8.3.5]{vershynin2025high} to obtain
\begin{align*}
    \mathbb{E}\biggl[\sup_{f \in \mathcal{F}} \Bigl \lvert \frac{1}{n_0} \sum_{i=1}^{n_0} f(\bx_i, g_i) - \mathbb{E}\bigl[f(\bx, g)\bigr]\Bigr \rvert \biggr] \lesssim \sqrt{\frac{d}{n}}.
\end{align*}
To convert this to a high probability bound on deviation, we note that the random variable $\sup_{f \in \mathcal{F}} \bigl \lvert \frac{1}{n_0} \sum_{i=1}^{n_0} f(\bx_i, g_i) - \mathbb{E}\bigl[f(\bx, g)\bigr]\bigr \rvert$ satisfies the bounded differences property with constant $1/n_0$, so by combining the preceding display with, e.g.,~\cite[Theorem 6.2]{BouLM13}, we obtain
\begin{align}
  \sup_{f \in \mathcal{F}} \Bigl \lvert \frac{1}{n} \sum_{i=1}^{n_0} f(\bx_i, g_i) - \mathbb{E}\bigl[f(\bx, g)\bigr]\Bigr \rvert
  \lesssim \sqrt{\frac{d}{n}} + \sqrt{\frac{\log(2/\delta)}{n_0}} \quad \text{ with probability} \geq 1 - \delta. \label{eq:hp-deviation-bd}
\end{align}
Now, since $\bx \sim \mathsf{N}(0, I_d)$, we have that for all $\theta \in \theta_\star + \bS^{d-1}$ and all $v \in \bS^{d-1}$ that
\begin{align*}
  &\P(\bx \cdot (\theta_\star - \theta) < 0) = \frac{1}{2}
  ~~~\textrm{and}~~~
  \P\left(|\bx \cdot v| \geq \sqrt{\frac{\pi}{8}}\right)
  \leq \frac{1}{\sqrt{2\pi}}\cdot \sqrt{\frac{\pi}{8}}
  = \frac{1}{4} \\
  \Longrightarrow&
  \P\left(\bx \cdot (\theta_\star - \theta) \geq 0 \text{ and } |\bx \cdot v| \geq \sqrt{\frac{\pi}{8}}\right) \geq \frac{1}{4}.
\end{align*}
Further, because $\bx$ and $g$ are independent, we have that
$
\mathbb{E}\bigl[f(\bx, g)\bigr] \geq \frac{1}{4} p_{\ell},
$
for every $f \in \mathcal{F}_{\ell}$.
Combining this with \eqref{eq:hp-deviation-bd} proves the bound.\qed

\subsection{Efficient algorithm} \label{sec:lr-efficient-alg}
So far we have shown the rates achievable by some $\widehat{\theta}\ofk$.
Now we show that an efficient estimator has a similar rate.
We consider a two-step procedure.

First we observe that because the loss which $\widehat{\theta}^{(1)}$ minimizes is a strongly convex, quadratic function, we can efficiently compute it to any desired accuracy.
By the analysis of the proof of the upper bound of \cref{thm:lin-reg-ub-lb}, we know that both 
\begin{align*}
  \bigl\|\widehat{\theta}^{(1)} - \theta_\star\bigr\|_2
  \leq \underbrace{C\sigma(1 + {\rm poly}(\tau))}_{=: R}
  \qquad\textrm{and}\qquad
  \bigl \|\widehat{\theta}^{(k)} - \theta_\star\bigr\|_2
  \leq \underbrace{\frac{C\sigma\tau\log\log(1/\alpha)}{\sqrt{\log(1/\alpha)}}}_{=: \beta},
\end{align*}
for a sufficiently large constant $C$.
Then for the first step of the procedure, we compute $\tilde{\theta}_1$ satisfying $\|\tilde{\theta}_1 - \widehat{\theta}^{(1)}\|_2 \leq R$ which implies that $\| \tilde{\theta}_1 - \widehat{\theta}\ofk|_2 \leq 2R+\beta$.
Recall by \cref{lem:mu_n-lb}, we have with probability at least $1-\delta$ that $F\ofk_n$ is $\mu_n$-strongly convex.
We then run the ellipsoid algorithm (see, e.g.,~\cite[Theorem 2.4]{Bub15}) on the domain $\mc{B} = \mathbf{B}_2(\tilde{\theta}_1, 2(R+\beta))$, until we obtain $\tilde{\theta}_2$ satisfying an excess risk bound on $F\ofk_n$ of at most $\mu_n \beta^2/2$ using a gradient oracle.
Strong convexity implies that $\|\tilde{\theta}_2 - \widehat{\theta}\ofk\|_2 \leq \beta$ and so $\tilde{\theta}_2$ has the same error, up to constant factors, as $\widehat{\theta}\ofk$.
To establish runtime, all that remains is to establish that the set 
\[
\mc{C} := \Bigl\{\theta:\; F\ofk_n(\theta) - F\ofk_n(\widehat{\theta}\ofk_n) \leq \mu_n \frac{\beta^2}{2}\Bigr\},
\]
contains a ball of large enough size.  To this end, on the domain $\mc{B}$, we can crudely upper bound the gradient norm $\| \nabla^2 F\ofk_n(\theta)\|_2$ as
\begin{align*}
  \sup_{\theta \in \mc{B}}\; \bigl\|\nabla^2 F\ofk_n(\theta)\bigr\|_2
  &\lesssim k^2 \sup_{i \in [n]} \; \bigl\{
  \bigl(2(R+\beta)\norm{\bx_i}_2+\sigma |g_i|\bigr)^{2k-2}
  \norm{x_i}_2^2\bigr\} \\
  &\lesssim O\bigl((R+\beta+ 1)\sqrt{d+\log(n/\delta)}\bigr)^{2k}.
\end{align*}
Thus, $\mathbf{B}_2\bigl(\widehat{\theta}\ofk, r\bigr) \subseteq \mc{C}$, where
\[
  r \gtrsim \frac{q(1-\epsilon)\sigma^{2k-2}(2k+1)!!\beta^2}{O((R+\beta+ 1)\sqrt{d+\log(n/\delta)})^{2k}}.
\]
Therefore, the ellipsoid portion of the algorithm has an oracle complexity of
\begin{align}
    O\left(d^2 \log \left(\frac{(R+\beta){O((R+\beta+ 1)\sqrt{d+\log(n/\delta)})^{2k}}}{q(1-\epsilon)\sigma^{2k-2}(2k+1)!!\beta^2}\right)\right),
\end{align}
which is polynomial in $k,d,n$.

\section{Proofs of information-theoretic lower bounds}

\subsection{Lower bound constructions}
\label{sec:it-bound-univariate}

For all lower bounds, we rely on hidden-direction distributions of the form $P_{A, v}$ (see \Cref{def:hidden-direction-dist}) in both the information-theoretic and SQ lower bounds.
Crucial to proving these is the following lemma, which characterizes the bounds on the density of distributions in the contamination set of a given distribution.

\begin{corollary}[Corollary of Lemma~\ref{lem:all-or-nothing-characterization}]
\label{cor:conditional-distribution-realizable}
Let $P$ be a distribution over $\R^d$.
Let $L = q(1-\epsilon)$ and $U = q(1- \epsilon) + \epsilon$.
Let $b \in [L,U]$ be arbitrary.
Then $Q \in \cR_\R(P,\epsilon,\pobserved)$
if for all $z \in \R^d$
\begin{align*}
\frac{L}{b} \leq \frac{dQ}{dP}(z) \leq \frac{U}{b}\,.
\end{align*}
Furthermore, suppose $Q$ satisfies the condition above for some $b \in [L,U]$.
Consider the distribution $Q'$ over $\R_\star^d$ which outputs $\star^d$ with probability $1-b$ and $X \sim Q$ with probability $b$.
Then $Q' \in \cR(P,\epsilon,\pobserved)$.
\end{corollary}

\begin{lemma}[Mean estimation hard instance]\label{lem:univariate-it-hard-instance-mean}
    Let $q \in (0, 1]$, $\epsilon \in [0, 1)$, and $\tau = \frac{\epsilon}{q(1-\epsilon)}$.
    Let $v \in \bS^{d-1}$ and
    let $b, \gamma, R \in \mathbf{R}_{++}$ such that
    \begin{align*}
      b = q(1- \epsilon) \sqrt{1 + \tau},
      \quad
      \gamma^2 \leq 0.25 \log(1+ \tau), \quad R \geq 2\gamma, 
      \quad\textrm{and}\quad
      R \leq \frac{\log(1 + \tau)}{8 \gamma},
    \end{align*}
    With $\beta = \frac{\P(|G+\gamma|\leq R)}{\P(|G|\leq R)}$, let $A$ denote the distribution over $\mathbf{R}$ with density
    \begin{align}
        A(x) := 
        \begin{cases}
        \beta \phi(x) & |x| \leq R\\
         \phi(x;\gamma) & \text{otherwise}   \,,
        \end{cases}
        \label{eq:it-bound-mean-A}
    \end{align}
    and define $Q \in \mathcal{P}(\mathbf{R}^d_{\star})$ as 
    \[
    Q(\{\star^{d}\}) = 1 - b \quad \text{ and } \quad Q(x) = b \cdot P_{A, v}(x).
    \]
    Then $Q \in \mathcal{R}\bigl(\mathsf{N}(\gamma v, I_d), \epsilon, q\bigr)$. 
    Furthermore,  $|\beta - 1| \lesssim \gamma e^{-R^2/4}$.
\end{lemma}

\begin{proof}
First observe that $A$ is a valid distribution since it integrates to $1$ on $\R$.
Furthermore, we have that $\beta \leq 1$ and $\beta \geq \min_{|x| \leq R} \frac{\phi(x; \gamma)}{\phi(x)} = \min_{|x| \leq R}\exp( \gamma x - \gamma^2/2 ) = \exp(-\gamma R - \gamma^2/2)$.

Let $f(x)$ be the density of $P_{A,v}$ at a point $x$. \Cref{cor:conditional-distribution-realizable} states that $Q$ is a valid realizable contamination as long as 
$\P_Q(X \neq  \star^d) = b \in (L, U)$
and $\frac{f(x)}{\phi(x;\gamma)} \in (L',U')$ for $L' = L/b$ and $U' = U/b$ and $L = q(1- \epsilon)$ and $U = q(1- \epsilon) + \epsilon$.

Observe that our choice of $b$ satisfies $b^2 = LU$ and thus $|\log L'| = |\log U'| = \log \sqrt{1 + \tau}$.
Therefore, for $Q$ to be a valid realizable contamination of $\mathsf{N}(\gamma v,I)$,
we require that $\left|\log\frac{f(x)}{\phi(x;\gamma)}\right|\leq 0.5\log(1 + \tau)$ for all $x \in \R^d$.
In fact, we shall establish a stronger condition that $\left|\log\frac{f(x)}{\phi(x;\gamma)}\right|\leq 0.25\log(1 + \tau)$ for all $x \in \R^d$;
this stronger requirement will be useful in proving SQ lower bounds.
Observe that 
\begin{align*}
    \frac{f(x)}{\phi(x;\gamma)} = \frac{\phi_{v^\perp} (x) A(v^\top x)}{\phi_{v^\perp} (x) \phi(v^\top x; \gamma)} = \frac{A(v^\top x)}{\phi(v^\top x; \gamma)}\,.
\end{align*}
We denote $x' = v^\top x$ in the rest of the proof for brevity.
We thus require that $\left|\log \frac{A(x')}{\phi(x'; \gamma)}\right| \leq 0.5 \log(1 + \tau) $ for all $x' \in \R$.
By the definition of $A$ in \Cref{eq:it-bound-mean-A}, this obviously holds on the domain $|x'| > R$.
For $|x'| \leq R$,
this ratio is equal to
\begin{align*}
\max_{|x'|\leq R}\left|\log\frac{A(x')}{\phi(x';\gamma)}\right| &= \max_{|x'|\leq R}\left|\log\frac{\beta \phi(x')}{\phi(x';\gamma)}\right| = \max_{|x'|\leq R}\left|\log \beta e^{\gamma^2/2 - \gamma x'} \right| = \left|\log \beta e^{\gamma^2/2 + \gamma R} \right| \\
&\leq  \left|\log \beta e^{\gamma^2/2}\right|  +   \gamma R \,.
\end{align*}
It suffices if each term is less than $0.125\log(1 + \tau)$, and observe that the second term, $\gamma R$, indeed satisfies this under the condition $R \leq \log(1 + \tau)/ (8 \gamma )$.

For the first term, observe that $\log \beta e^{\gamma^2/2} \leq \gamma^2/2$ since $\beta \leq 1$, which is also appropriately bounded under the condition on $\gamma^2$ listed in the statement.
Finally, $\log \beta e^{\gamma^2/2} \geq \log e^{- \gamma R - \gamma^2/2} e^{\gamma^2/2} = - \gamma R$.
And thus $-\log \beta e^{\gamma^2/2} \leq \gamma R $, which as we argued earlier is indeed small.
Therefore,
\begin{align*}
\max_{x' \in \R}\left|\log\frac{A(x')}{\phi(x';\gamma)}\right| &\leq 0.25 \log(1+\tau).
\end{align*}

Finally, we state an upper bound on $|\beta - 1|$ which is equal to
\begin{align*}
   |\beta - 1| &= \frac{|\P(|G + \gamma| \leq R) - \P(|G|\leq R)|}{\P(|G|\leq R)} \lesssim \left| \P(G \in [-R - \gamma, R-\gamma])  - \P(G \in [-R,R])| \right| \\
   &\lesssim \P(G \in [R -\gamma, R]) \leq \gamma \phi(R/2)  \lesssim \gamma e^{-R^2/4}\,,
\end{align*} 
where we use that $R \geq 2\gamma$.
\end{proof}

\begin{lemma}[Covariance estimation hard instance]
\label{lem:univariate-it-hard-instance-cov}
Let $q \in (0, 1]$, $\epsilon \in [0, 1)$, $v \in \bS^{d-1}$ and
let $b, \gamma, R \in \mathbf{R}_{++}$ satisfy
\begin{align*}
  b = q(1- \epsilon) \sqrt{1 + \tau},
  \quad
  \gamma \leq \tau,
  \quad\textrm{and}\quad
  R^2 \leq \frac{1+\gamma}{\gamma}\log(1 + \tau),
\end{align*}
With $\beta = \frac{\P(\sqrt{1+\gamma}|G|\leq R)}{\P(|G|\leq R)}$, let $A$ denote the distribution over $\mathbf{R}$ with density
\begin{align}
    A(x) := 
    \begin{cases}
    \beta \phi(x) & |x| \leq R\\
     \phi(x;0, 1+\gamma) & \text{otherwise}   \,,
    \end{cases}
    \label{eq:it-bound-cov-A}
\end{align}
and define $Q \in \mathcal{P}(\mathbf{R}^d_{\star})$ as 
\[
Q(\{\star^{d}\}) = 1 - b \quad \text{ and } \quad Q(x) = b \cdot P_{A, v}(x).
\]
Then $Q \in \mathcal{R}\bigl(\mathsf{N}(0, I_d + \gamma vv^T), \epsilon, q\bigr)$. 
Furthermore,
\[
|\beta - 1| \lesssim \min\left\{1,R \cdot \left|1 - \frac{1}{\sqrt{1+\gamma}}\right|\right\} e^{-\Omega(R^2/(1+\gamma))}.
\]
\end{lemma}
\begin{proof}
    Observe that $\beta$ is chosen to ensure that $A$ is a valid probability distribution.  Moreover, by definition, 
    $\beta$ satisfies the bounds
    \begin{align} \label{ineq:beta-bounds-cov}
    \beta \leq 1 \quad \text{ and } \quad \beta \geq \min_{|x| \leq R} \frac{1}{\sqrt{1+\gamma}}\frac{e^{-x^2/2(1+\gamma)}}{e^{-x^2/2}} = \frac{1}{\sqrt{1+\gamma}}.
    \end{align}
    With $\Sigma = I_d + \gamma v v^{\top}$ and following the proof of \Cref{lem:univariate-it-hard-instance-mean},
    by our choice of $b$,
    it suffices to show that for all $x \in \R^d$: $\bigl \lvert\log \frac{f(x)}{\phi(x;0,\Sigma)}\bigr \rvert \leq \log \sqrt{1 + \tau}$, where $f$ denotes the density of the hidden direction distribution $P_{A, v}$.
    Using the definition of $P_{A,v}$, the likelihood ratio is equal to
    \begin{align*}
        \frac{f(x)}{\phi(x;0,\Sigma)} = \frac{A(x')}{\phi(x';0,1+\gamma)},
    \end{align*}
    where $x' = v^\top x$.
    From the definition of $A$, it is clear that the above likelihood ratio is equal to $1$ when $|x'| > R$, and thus the desired bound is satisfied.
    On the other hand, for $|x'| \leq R$, we have
    \begin{align*}
    \log \frac{A(x')}{\phi(x';0,1+\gamma)}
    &= \log \bigl(\beta \sqrt{1 +\gamma}\bigr) + \log e^{\frac{-x'^2\gamma}{2(1+\gamma)}} 
     =  \log \bigl(\beta \sqrt{1+\gamma}\bigr) - \frac{x'^2 \gamma}{2(1+\gamma)}.
    \end{align*}
    In order to show that $Q \in \mathcal{R}\bigl(\mathsf{N}(0, I_d + \gamma vv^T), \epsilon, q\bigr)$, it remains to establish the pair of inequalities
    \begin{align*}
      \log\bigl(\beta \sqrt{1+\gamma}\bigr) \leq \log \sqrt{1+\tau}
      \quad\textrm{ and }\quad
      \log \bigl(\beta\sqrt{1+\gamma} \bigr) - \frac{R^2\gamma}{2(1+\gamma)} \geq -\log \sqrt{1+\tau}.
    \end{align*}
    The first inequality follows by applying the upper bound $\beta \leq 1$ in~\eqref{ineq:beta-bounds-cov} in conjunction with the assumption $\gamma \leq \tau$.
    Turning to the second inequality, we have
    \[
    \log \bigl(\beta\sqrt{1+\gamma} \bigr) - \frac{R^2\gamma}{2(1+\gamma)} \geq - \frac{R^2\gamma}{2(1+\gamma)} \geq -\log \sqrt{1+\tau},
    \]
where the first step follows from the lower bound $\beta \geq \frac{1}{\sqrt{1+\gamma}}$ in~\eqref{ineq:beta-bounds-cov} and the final step follows from the assumption
    $R^2 \leq \frac{1+\gamma}{\gamma} \log(1+\tau)$.
 
    Finally, we obtain a bound on $|\beta -1|$, which holds under the condition $R  \gtrsim \sqrt{1+\gamma}$.  In particular, expanding yields
    \begin{align*}
    |\beta - 1| &= \frac{ \P(|G| \leq R) - \P(|G| \leq \frac{R}{\sqrt{1+\gamma}})}{\P(|G|\leq R)}
    = \frac{\P\bigl(|G| \in \bigl[\frac{R}{\sqrt{1+\gamma}}, R\bigr] \bigr)}{\P(|G|\leq R)} \\
    &\lesssim \P\left(|G| \in \left[\frac{R}{\sqrt{1+\gamma}}, R\right] \right) \\
    &\lesssim  \min\biggl\{\P\biggl(|G| \geq \frac{R}{\sqrt{1+\gamma}}\biggr), \biggl \lvert R - \frac{R}{\sqrt{1+\gamma}} \biggr \rvert \cdot \phi\biggl(\frac{R}{\sqrt{1+\gamma}}\biggr)\biggr\} \\
    &\lesssim \min\left( 1, R \cdot \left|1 - \frac{1}{\sqrt{1+\gamma}}\right| \right) e^{-\Omega(R^2/(1+\gamma))}.
    \end{align*}
\end{proof}

\begin{lemma}[Linear regression hard instance]
\label{lem:univariate-it-hard-instance-lr}
Let $q \in (0, 1]$, $\epsilon \in [0, 1)$, and $\tau = \frac{\epsilon}{q(1-\epsilon)}$.
Let $v \in \bS^{d-1}$ and
let $b, \gamma, r, R \in \mathbf{R}_{++}$ such that
\begin{align*}
  b = q(1- \epsilon) \sqrt{1 + \tau},
  \quad
  \gamma^2 r^2 \leq 0.5 \log(1+ \tau),
  \quad\textrm{and}\quad
  R \leq \frac{\log(1 + \tau)}{8 \gamma r}.
\end{align*}
Let $P_*$ denote the distribution over $(X,y)$: $X \sim \mathsf{N}(0,I_d)$ and $y |X \sim \mathsf{N}(\gamma v^\top X,1)$. 
For all $x \in \mathbf{R}^d$, with $\beta_x \in (0, 1]$ defined implicitly, let $A_x$ denote the distribution over $\mathbf{R}$ with density
\begin{align}
    A_x(y) := 
    \begin{cases}
    \beta_x \phi(y) & |v^\top x| \leq r \text{ and } |y|\leq R\\
     \phi(y; \gamma v^\top x,1) & \text{otherwise}   \,.
    \end{cases}
    \label{eq:it-bound-lr-A}
\end{align}
and define $Q \in \mathcal{P}(\mathbf{R}^{d+1}_{\star})$ as 
\[
Q(\{\star^{d+1}\}) = 1 - b \quad \text{ and } \quad Q(x, y) = b \cdot \phi(x)A_x(y).
\]
Then $Q \in \mathcal{R}\bigl(P_*, \epsilon, q\bigr)$.

\end{lemma}

\begin{proof}
For any $x$, $A_x$ is a valid distribution since it integrates to $1$: (i) this is obvious if $\gamma |v^\top x| > r$ and (ii) for $\gamma|v^\top x| \leq r$, this is satisfied if we choose
\begin{align*}
    \beta_x \P(|G|\leq R) + \P(|G+ \gamma x^\top v|\leq R) = 1 \iff \beta_x = \frac{\P(|G + \gamma x^\top v|\leq R)}{\P(|G|\leq R)}\,.
\end{align*}
Note that for all $x$, $\beta_x \in (0, 1]$.
Using the same argument as in the proof of \Cref{lem:univariate-it-hard-instance-mean} for a fixed $x$, we get that
\begin{align*}
    1 \geq \beta_x &\geq \min_{|y| \leq R} \frac{\phi(y; \gamma x^\top v,1)}{\phi(y)}
    = \min_{|y|\leq R} \frac{e^{-(y-\gamma x^\top v )^2/2}}{e^{-y^2/2}}
    = \min{|y|\leq R} e^{-\gamma^2 (x^\top v)^2/2 + \gamma y x^\top v}\\
    &= e^{-\gamma^2 (x^\top v)^2/2 - \gamma R |x^\top v|}\,.
\end{align*}

Let $f(x,y)$ be the conditional density of $Q$ under the event $\{(x,y) \neq \star^{d+1}\}$ and let $p(x,y)$ denote the density of the linear model $X \sim \mathsf{N}(0,I_d)$ and $y \sim \mathsf{N}(\gamma v^\top X, 1)$.
As earlier,
we want the following likelihood ratio to lie in the range as per \Cref{cor:conditional-distribution-realizable}.
\begin{align*}
    \frac{f(x,y)}{p(x,y)} = \frac{A_x(y)}{\phi(y; \gamma v^\top x, 1)} = \begin{cases}
        \frac{\beta_x \phi(y)}{\phi(y; \gamma v^\top x, 1)} & \text{ if } |y|\leq R \text{ and } \gamma|v^\top x| \leq r\\
        1 & \text{ otherwise}
    \end{cases}  
\end{align*}
By \cref{cor:conditional-distribution-realizable} and our choice of $b$, we are done once we establish that the absolute value of the logarithm of this density ratio is uniformly upper bounded by $0.5 \log(1+\tau)$.
As per the definition, this needs to be checked only when $|y|\leq R$ and $\gamma|v^\top x| \leq r$.
For such $(x,y)$,
the ratio is
\begin{align*}
    \frac{f(x,y)}{p(x,y)} = \frac{\beta_x e^{-y^2/2}}{e^{-(y- \gamma x^\top v)^2/2}} = \beta_x e^{\gamma^2 (x^\top v)^2/2} e^{-\gamma y x^\top v}
\end{align*}
And therefore, the logarithm satisfies
\begin{align*}
    \log \frac{f(x,y)}{p(x,y)} = \log \beta_x e^{\gamma^2 (x^\top v)^2/2} - \gamma y x^\top v\,.
\end{align*}
It would suffice if the absolute value of each term is at most $0.25 \log(1+\tau)$.
The second term is at most $\gamma Rr$, which by assumption is bounded appropriately.
For the first term,  we  use that $e^{-\gamma^2 (x^\top v)^2/2 - \gamma R |x^\top v|} \leq \beta_x \leq 1$ to get
\begin{align*}
    \log \beta_x e^{\gamma^2 (x^\top v)^2/2} \leq \gamma^2 r^2/2
    \quad\textrm{and}\quad
    \log \beta_x e^{\gamma^2 (x^\top v)^2/2} \geq \log e^{-2\gamma R |x^\top v|} = -2\gamma R |x^\top v|.
\end{align*}
Thus, $| \log \beta_x e^{\gamma^2 (x^\top v)^2/2}|\leq \max( \gamma^2 r^2/2, 2\gamma R r)$, which is also bounded appropriately.
\end{proof}

\subsection{Information-theoretic lower bounds}\label{app:it-lbs}
\subsubsection{Mean estimation: Proof of Theorem~\ref{thm:mean-est-it-ub-lb} lower bound}\label{pf:mean-est-it-lb}

We prove the claim for general $q \in (0, 1]$.
We introduce the quantities $\tau = \frac{\epsilon}{q(1-\epsilon)}$, $\tau' = \log(1+\tau)$, and $b = q(1-\epsilon)\sqrt{1+\tau}$.
We claim that it suffices to prove that
\begin{align}
\label{ineq:sufficient-cond-minimax-mean-lb}
\cM := \cM_n(\delta, \cP_{\rm mean}(\sigma^2, \epsilon, q),L)
\gtrsim T_{\mathrm{parametric}} \vee T_{\mathrm{dimension}} \vee T_{\mathrm{confidence}},
\end{align}
where 
\begin{align*}
    T_{\mathrm{parametric}} := &\sqrt{\frac{d + \log(\frac{1}{\delta})}{nq(1-\epsilon)}}, \quad T_{\mathrm{dimension}} := \sqrt{\tau'} \wedge \frac{\tau '}{\sqrt{\log\left(3 + \frac{cn b\tau'^2}{d}\right)}}\, , \\
    \quad &\text{ and } \quad T_{\mathrm{confidence}} := \sqrt{\tau'} \wedge \frac{\tau '}{\sqrt{\log\left(3 + \frac{C n b\tau'^2}{\log(1/2\delta)}\right)}},
\end{align*} 
for a sufficiently large constant $C$.
We defer the proof of this reduction to the end of the section.
We next prove each lower bound in turn.

\paragraph{Proof of $\cM \gtrsim T_{\mathrm{parameteric}}$:}
We consider the contamination that censors everything, so that the resulting distribution is MCAR with probability or observation equal to $q(1-\epsilon)$.
Then $\cM \gtrsim T_{\mathrm{parametric}}$ is essentially the standard minimax mean estimation lower bound applied with the fact that there are $O(nq(1-\epsilon))$ non-missing samples with high probability, see \cite[Prop.\ 47]{ma2024estimation} for a formal argument. 

\paragraph{Proof of $\cM \gtrsim T_{\mathrm{dimension}}$:}
This case forms the technical bulk of the proof and proceeds via Fano's inequality.  We begin by defining several parameters.  We will let $\gamma := \gamma(\epsilon, q, n, d)$ denote our separation parameter.  Throughout we will impose the constraint
\begin{align} \label{ineq:constraint-gamma-mean-est}
\gamma \leq \frac{\sqrt{\tau'}}{10},
\end{align}
and set a truncation threshold $R = \tau'/(10\gamma)$.  Note that from the constraint~\eqref{ineq:constraint-gamma-mean-est}, we deduce the bound
\begin{align} \label{ineq:ub-gamma-R}
    \gamma \leq \frac{\sqrt{\tau'}}{10} \leq \frac{R}{10} \leq \frac{R}{2}.
\end{align}
Equipped with these parameters, we let $\mathcal{N}$ denote a $1/2$--packing of $\mathbf{S}^{d-1}$, which by, e.g.,~\cite[Corollary 4.2.11]{vershynin2025high}, exists and satisfies the cardinality bound $|\mathcal{N}| \leq 5^d$.  For each $v \in \mathcal{N}$, it suffices to construct distributions $P_{v}$ which satisfy the pair of desiderata
\begin{enumerate}
    \item (Valid contamination) $P_v\in \cR(\mathsf{N}(\gamma v, I_\dimension), \epsilon, \pobserved)$.
    \item (Equal probability of missing mass) The missingness probability $P_v \bigl(\{\star\}^d\bigr) = 1-b$.
\end{enumerate}

We will additionally define a central distribution $H \in \mathcal{P}\bigl(\mathbf{R}^d \cup \{\star^d\}\bigr)$ as
\begin{align} \label{def:central-H-mean}
H\bigl(\{\star^{d}\}\bigr) = 1 - b \quad \text{ and } \quad H(x) = b \cdot \phi_d(x) \text{ for all } x \in \mathbf{R}^d.
\end{align}
By Fano's inequality (Lemma~\ref{lem:fano}), it suffices to construct distributions $P_{\theta}$ which satisfy the two desiderata above and such that
for all $v \in \cN$, it holds that $\mathrm{KL}(P_{v}, H) \leq \frac{d}{2n}$.
Indeed, if this KL bound holds, then by tensorization of the KL divergence,
\[
\frac{\frac{1}{|\cN|} \sum_{v \in \cN} \, \mathrm{KL}\bigl(P_{v}^{\otimes n}, H^{\otimes n}\bigr) - \log(2 - |\cN|^{-1})}{\log |\cN|} \leq \frac{\frac{1}{2}d}{d \log{5}}
< \frac{1}{2}.
\]

We take $P_v$ to be the hard distribution $Q$ from \Cref{lem:univariate-it-hard-instance-mean}, with all lemma parameter chosen identically to this theorem.
The remainder of this proof is dedicated to establishing the bound on $\mathrm{KL}(P_{v}, H)$.

Since the probability of observing $\{\star^d\}$ is identical under both $P_v$ and $H$, and since $P_v$ and $H$ are supported on $\{\star^d\} \cup \mathbf{R}^d$, we have
\begin{align}
   \label{eq:it-mean-kl-div-1}
    \mathrm{KL}(P_v,H) = b \cdot \mathrm{KL}( P_v',H'),
\end{align}
where $P_v' = P_{A,v}$ and $H'= \mathsf{N}(0,I_d)$ denote the respective conditional distributions over $\mathbf{R}^d$.
Denoting $x' := v^\top x$ and $\bar{x}$ the projection of $x$ onto $v^\perp$,
observe that the density of $P_v'$ is $p_v'(x) = \phi_{d-1}(\bar{x}) A(v^\top x)$ and that of
$H'$ is $h'(x) = \phi_{d-1}(\bar{x}) \phi(v^\top x)$. 
Applying the definition of $A$ in \Cref{eq:it-bound-mean-A} yields
\begin{align*}
    \frac{p_v'(x)}{h'(x)} &= \frac{A(x')}{\phi(x')} =  \begin{cases}
        \beta & \text{ if } |x'| \leq R \\
       e^{-0.5\gamma^2+ \gamma x'} & \text{ if } |x'| > R.   \end{cases}
\end{align*}
In turn, this implies that the log-likelihood ratio satisfies
\begin{align*}
    \log \frac{p_v'(x)}{h'(x)} &= (\log \beta )\1_{|x'|\leq R} + (-0.5 \gamma^2 + \gamma x' )\1_{|x'| > R} \leq  \gamma x'\1_{|x'| > R},  
    \end{align*}
    where we have used the bound $\beta \leq 1$.
The expectation of $\log \frac{p_v'(x)}{h'(x)}$ under $P'_{\theta}(x)$ is exactly equal to $\KL(P_v',H')$.
Since the upper bound depends only on $x' \sim A$, it is equivalent to take an expectation of this term over $x' \sim A$,
which as per \Cref{eq:it-bound-mean-A} is equal to 
\begin{align}
\nonumber
\KL(P_v',H') &\leq \E_{X \sim A}[\gamma X \1_{|X| > R}] =\gamma \E[ (G + \gamma) \1_{|G + \gamma| > R}] \\
&= \gamma^2 \P(|G + \gamma| > R) + \gamma\E [G \1_{|G + \gamma|>R}]\,.
\label{eq:it-mean-kl-div-2}
\end{align}
To calculate an upper bound on the first term, we use that $\gamma< R/2$ and obtain that $\P(|G + \gamma| > R) \leq \P(|G| > R/2) \lesssim e^{-\Omega(R^2)}$.
For the second term, we use the following inequalities that rely on symmetry:
\begin{align}
\nonumber
    \E [G \1_{|G + \gamma|>R}] &=  \E [G \1_{|G|>R}]  +  \E [G \left( \1_{|G + \gamma|>R} -\1_{|G| \geq R}\right)]  \\
 \nonumber
   &=\E [G \left( \1_{|G + \gamma|>R} -\1_{|G| \geq R}\right)] \\
\nonumber
    &\leq \E [|G| \left| \1_{|G + \gamma|>R} -\1_{|G| \geq R}\right|] \\
\nonumber
    &\leq \E [|G| \1_{|G| \in [R-\gamma, R + \gamma]}] \\ 
    \nonumber
\nonumber
    &\leq 4 \gamma \max_{x \in [-R/2,3R/2]} x \phi(x) \\
    &\lesssim \gamma e^{-\Omega(R^2)},\,
   \label{eq:it-mean-kl-div-3}
\end{align}
where we use that $x e^{-x^2/2} \lesssim e^{-x^2/4}$ for all $x>0$.
Combining \Cref{eq:it-mean-kl-div-1,eq:it-mean-kl-div-2,eq:it-mean-kl-div-3}, we deduce the inequality
\begin{align} \label{ineq:KL-bound-mean-est}
    \mathrm{KL}(P_v, H)
    \lesssim b\gamma^2 e^{-\Omega(R^2)}
    \lesssim b\gamma^2 e^{- c \frac{\tau'^2}{\gamma^2}}
    \lesssim b\tau'^2 y e^{-1/y}\,,
\end{align}
where we use the value of $R =\tau'/(10\gamma)$ and define $y:= \frac{\gamma^2}{c \tau'^2}$, for a sufficiently small constant $c$.
For this to be less than $\frac{d}{2n}$,
it suffices that $ye^{-1/y} \leq \rho':=\frac{d}{C b\tau'^2n}$ for a sufficiently large constant $C$.
That is, $y \leq \frac{1}{W_0(1/\rho')}$, where $W_0(x)$ is the Lambert W function (principal branch),
which is the unique positive value such that $W_0(x) e^{W_0(x)} = x$ for $x\geq 0$.
Observe that $\log(3+x) e^{\log(3+x)} > x + 3 > x$ and thus $W_0(x)\leq \log(3+x)$.
This condition is always satisfied whenever $ \gamma \lesssim \tau' \sqrt{y} = \frac{\tau '}{\sqrt{W_0\left(\frac{C n b\tau'^2}{d}\right)}}$.
And since $\gamma$ must also satisfy that $\gamma \lesssim \sqrt{\tau'}$
for all of these calculations to be valid,
we get the desired lower bound in mean estimation
\begin{align}
    \cM \gtrsim \gamma \gtrsim \sqrt{\tau'} \wedge \frac{\tau '}{\sqrt{\log\left(3 + \frac{C n b\tau'^2}{d}\right)}} = T_{\rm dimension},
    \label{eq:mean-estimation-dim-lb}
\end{align}
which concludes this part of the proof.

\paragraph{Proof of $\cM \gtrsim T_{\mathrm{confidence}}$:} We will use the same constructions as in the previous part of the proof.  To this end, note that $H \in \cR(\mathsf{N}(0,I_\dimension), \epsilon, \pobserved)$.  We let $v$ be any unit vector, let $Q$ be the distribution from Lemma~\ref{lem:univariate-it-hard-instance-mean}.
Applying the Bretagnolle--Huber inequality in conjunction with the KL bound from~\eqref{ineq:KL-bound-mean-est} yields
\[
\mathrm{TV}(Q^{\otimes n}, H^{\otimes n})
\leq \sqrt{1 - e^{-n\mathrm{KL}(Q,H)}} \leq \sqrt{1 - \mathrm{exp}\bigl\{-n \cdot b \tau'^2 y e^{-1/y}\bigr\}}.
\]
Following the same steps as in the previous part and ensuring $y e^{-y} \lesssim \frac{\log(1/2\delta)}{b\tau'^2n}$ implies that $\mathrm{TV}(P_{\theta}^{\otimes n}, H^{\otimes n}) \leq 1 - 2\delta$.  Hence, we apply~\cite[Lemma 5]{ma2024high} to deduce that 
\begin{align}
    \cM \gtrsim \sqrt{\tau'} \wedge \frac{\tau '}{\sqrt{\log\left(3 + \frac{C n b\tau'^2}{\log(1/2\delta)}\right)}}\,,
    \label{eq:mean-estimation-hp-lb}
\end{align}
as desired.  

\paragraph{Proof of the reduction~\eqref{ineq:sufficient-cond-minimax-mean-lb}:}
Re-phrasing \eqref{ineq:sufficient-cond-minimax-mean-lb}, we have the minimax lower bound
\begin{align*}
    \cM
    \gtrsim
    \alpha + \sqrt{\tau'} \wedge \frac{\tau '}{\sqrt{\log\left(3 + \frac{C \sqrt{1+\tau}\tau'^2}{\alpha^2}\right)}}\,,
\end{align*}
where $\alpha = \sqrt{\frac{d+\log(1/\delta)}{nq(1-\epsilon)}} \leq 1/\sqrt{C_1}$ and $C$ is a sufficiently large constant.
Let $C'$ be sufficiently large constant that we will specify later.
Then by considering several cases, we will show that this is equivalent to the claimed lower bound in Theorem~\ref{thm:mean-est-it-ub-lb}.

\begin{itemize}
\item 
\textbf{Case 1: $\tau \leq C'\alpha$.}
This regime is immediate because
\begin{align*}
  \frac{\log(1 + \tau)}{\sqrt{\log(1 + \tau^2/\alpha^2)}}
  \asymp \tau\sqrt{\alpha^2/\tau^2} = \alpha \lesssim \cM.
\end{align*}

\item
\textbf{Case 2: $C'\alpha \leq \tau \leq 1$.}
Then because $\tau' \leq \tau \leq 1$, we have that
\begin{align*}
  \sqrt{\tau'} \wedge \frac{\tau '}{\sqrt{\log\left(3 + \frac{C \sqrt{1+\tau}\tau'^2}{\alpha^2}\right)}}
  &= \frac{\tau '}{\sqrt{\log\left(3 + \frac{C \sqrt{1+\tau}\tau'^2}{\alpha^2}\right)}} \\
  &\geq \frac{\log(1+\tau)}{\sqrt{\log(3 + C\sqrt{2}\tau^2/\alpha^2)}} \\
  &\geq \frac{\log(1+\tau)}{\sqrt{2\log(1 + \tau^2/\alpha^2)}},
\end{align*}
where the final inequality follows from $\tau/\alpha \geq C'$ by picking $C'$ large enough (depending on $C$).

\item
\textbf{Case 3: $1 \leq \tau \leq \alpha^{-2}/C_1$.}
In this regime, $\log\left(3 + \frac{C \sqrt{1+\tau}\tau'^2}{\alpha^2}\right) \asymp \log(1/\alpha)$ so that
\begin{align*}
    \sqrt{\tau'} \wedge \frac{\tau '}{\sqrt{\log\left(3 + \frac{C \sqrt{1+\tau}\tau'^2}{\alpha^2}\right)}}
    \asymp
    \sqrt{\log(1+\tau)} \cdot \left[1 \wedge \sqrt{\frac{\log(1+\tau)}{\log(1/\alpha)}}\right].
\end{align*}
Then because $\tau \leq \alpha^{-2}/C_1$, we have $\log(1+\tau) \lesssim \log(1/\alpha)$, and hence
\begin{align*}
   \sqrt{\tau'} \wedge \frac{\tau '}{\sqrt{\log\left(3 + \frac{C \sqrt{1+\tau}\tau'^2}{\alpha^2}\right)}}
   \gtrsim \frac{\log(1+\tau)}{\sqrt{\log(1/\alpha)}}
   \asymp \frac{\log(1+\tau)}{\sqrt{\log(1 + \tau^2/\alpha^2)}},
\end{align*}
where the final equality (up to constant factors) follows because $1 \leq \tau \leq \alpha^{-2}/C_1$.

\item
\textbf{Case 4: $\tau > \alpha^{-2}/C_1$.}
In this regime, $\log\left(3 + \frac{C \sqrt{1+\tau}\tau'^2}{\alpha^2}\right) \asymp \log(1 + \tau)$ so that
\begin{align*}
    \sqrt{\tau'} \wedge \frac{\tau '}{\sqrt{\log\left(3 + \frac{C \sqrt{1+\tau}\tau'^2}{\alpha^2}\right)}}
    \asymp
    \sqrt{\log(1+\tau)}.
\end{align*}
Moreover, $\log(1 + \tau^2/\alpha^2) \asymp \log(1 + \tau)$ because $\tau > \alpha^{-2}/C_1$.
Thus,
\begin{align*}
   \sqrt{\tau'} \wedge \frac{\tau '}{\sqrt{\log\left(3 + \frac{C \sqrt{1+\tau}\tau'^2}{\alpha^2}\right)}}
   \asymp \frac{\log(1+\tau)}{\sqrt{\log(1 + \tau^2/\alpha^2)}}.
\end{align*}
\end{itemize}
This completes the reduction, proving the claim. \hfill $\qed$

\subsubsection{Covariance estimation: Proof of Theorem~\ref{thm:cov-est-it-ub-lb} lower bound}\label{pf:cov-est-it-lb}
We prove the claim for general $q \in (0, 1]$.
The proof will closely mirror the structure of the proof of the lower bound of \Cref{thm:mean-est-it-ub-lb}.
We introduce the quantities $\tau = \frac{\epsilon}{q(1-\epsilon)}$, $\tau' = \log(1+\tau)$, and $b = q(1-\epsilon)\sqrt{1+\tau}$.
We claim that it suffices to prove that
\begin{align}
\label{ineq:sufficient-cond-minimax-cov-lb}
\cM := \cM_n\bigl(\delta, \cP_{\rm cov}(\eps, q), L\bigr) \geq T_{\mathrm{parametric}} \vee T_{\mathrm{dimension}} \vee T_{\mathrm{confidence}},
\end{align}
where 
\begin{align*}
    T_{\mathrm{parametric}} := &\sqrt{\frac{d + \log(\frac{1}{\delta})}{nq(1-\epsilon)}}, \quad T_{\mathrm{dimension}} := \tau' \wedge \frac{\tau '}{\log\left(3 + \frac{cn q(1-\epsilon)\sqrt{1+\tau}\tau'}{d}\right)}\, , \\
    \quad &\text{ and } \quad T_{\mathrm{confidence}} := \tau' \wedge \frac{\tau '}{\log\left(3 + \frac{cn q(1-\epsilon)\sqrt{1+\tau}\tau'}{\log(1/4\delta)}\right)},
\end{align*}
for a sufficiently large constant $c$.  As in the mean estimation case, the lower bound $\cM \geq T_{\mathrm{parametric}}$ is standard, and we omit its proof.  We next prove each of the lower bounds $\cM \geq T_{\mathrm{dimension}}$ and $\cM \geq T_{\mathrm{confidence}}$ in turn.

\paragraph{Proof of $\cM \geq T_{\mathrm{dimension}}$:}
Let $\gamma \in [0, 1]$ denote the target minimax error, to be specified later.
Let $\cN$ be a $1/2$-packing of unit vectors as in the proof of Theorem~\ref{thm:mean-est-it-ub-lb}(b) satisfying $|\cN| \leq 5^d$.
For each $v \in \cN$, we will construct hypotheses $P_v$ which satisfy the pair of desiderata
\[
P_v \in \cR\bigl(\mathsf{N}(0, I_d + \gamma vv^\top), \eps, q\bigr) \quad \text{ and } \quad P_v\bigl(\{\star^d\}\bigr) = 1-b.
\]
By \cref{lem:psd-separation}, for any distinct $v,v' \in \cN$, we have that for all matrices $M$,
$L(M, I_d + \gamma vv^\top) \vee L(M, I_d + \gamma \bar{v}\bar{v}^\top) > \gamma/16$.
Hence, any estimator $\widehat{\Sigma}$ with worst-case risk at most $\gamma/64$
must satisfy $\P_{P_v}(L(\widehat{\Sigma}, I_d + \gamma vv^\top) \geq \gamma/16) \leq 1/4$ for all $v \in \cN$,
and can be used to recover $v \in \cN$ with probability at least $3/4$.

To prove that recovering $v \in \cN$ with probability at least $3/4$ is impossible,
it will suffice by Fano's inequality (\cref{lem:fano}) to show that there exists a distribution $H$ such that
$\KL(P_v, H) \lesssim \frac{d}{n}$ for all $v \in \cN$.
We will take $P_v$ to be the distribution $Q$ from \cref{lem:univariate-it-hard-instance-cov} with $R^2 = \frac{1+\gamma}{\gamma}\log(1+\tau)$ and the common distribution $H$ to be defined as
\begin{align} \label{def:central-H-mean-cov}
H\bigl(\{\star^{d}\}\bigr) = 1 - b \quad \text{ and } \quad H(x) = b \cdot \phi_d(x) \text{ for all } x \in \mathbf{R}^d.
\end{align}
Since $P_v(\{\star^d\})=H(\{\star^d\}) = 1-b$ and both distributions are supported on $\R^d \cup \{\star^d\}$, we deduce that
\begin{align}
\label{eq:it-cov-kl-div-1}
    \mathrm{KL}(P_v,H) = b \cdot \mathrm{KL}(P_v',H') \,,
\end{align}
where $P_v'$ and $H'$ denote the distributions of $P_v$ and $H$ conditioned on belonging to $\R^d$.  We next introduce the shorthand $x' := v^{\top} x$ and let $p_v'$ and $h'$ denote the densities of $P_v'$ and $H'$.  Expanding their definitions, we simplify the likelihood ratio as
\begin{align*}
    \frac{p_v'(x)}{h'(x)} &= \frac{A(x')}{\phi(x')}
    =  \begin{cases}
          \beta & \text{ if } |x'| \leq R \\
         \frac{1}{\sqrt{1+\gamma}} e^{\frac{\gamma |x'|^2}{2(1+\gamma)}} & \text{ if } |x'| > R.
       \end{cases}
\end{align*}
Hence, we find that the log-likelihood ratio is bounded as
\begin{align}
\log \frac{p_v'(x)}{h'(x)} &= (\log \beta )\1_{|x'|\leq R} + \left( \frac{\gamma \lvert x' \rvert^2}{2(1+\gamma)} - \log\sqrt{1+\gamma}\right)\1_{|x'| > R}
\leq \frac{\gamma|x'|^2}{2(1+\gamma)}  \1_{|x'| > R},
\label{eq:it-cov-kl-div-2}
\end{align}
where we have additionally used the inequality $\beta \leq 1$, which holds by construction.  Moreover, our construction implies that if $X' \sim A$,
\[
X' \cdot \1_{\lvert X' \rvert > R} \overset{\mathsf{dist}}{=} \sqrt{1 + \gamma} \cdot G \1_{\sqrt{1 + \gamma} \lvert G \rvert > R}.
\]
Taking expectations with respect to the random variable $X' \sim A$, it thus follows from the upper bound~\eqref{eq:it-cov-kl-div-2} that
\begin{align}
   \label{eq:it-cov-kl-div-2a}
      \KL(P_v',H') \leq \frac{\gamma}{2}\E\bigl[G^2 \1_{\sqrt{1+\gamma}|G| > R}\bigr]
      \overset{(a)}{\lesssim} \gamma \sqrt{\P(\sqrt{1+\gamma} |G| > R)}
      \lesssim \gamma e^{-\Omega(R^2/(1+\gamma))},
\end{align}
where step $(a)$ follows upon applying the Cauchy--Schwarz inequality.  
Combining \Cref{eq:it-cov-kl-div-1,eq:it-cov-kl-div-2a} then yields the upper bound
\begin{align} \label{eq:it-cov-kl-div-2b}
\KL(P_v,H) \lesssim b \gamma e^{-\Omega(R^2/(1+\gamma))}\, \lesssim b \gamma e^{-\Omega(\log(1+\tau)/\gamma)},
\end{align}
where the final inequality follows since by construction $R^2 = \frac{1+\gamma}{\gamma}\log(1+\tau)$.  

To obtain the desired lower bound, we set
\[
\gamma
= \min\Biggl\{c, \frac{\log(1+\tau)}{\log\Bigl(\frac{q(1-\eps)n}{cd}\Bigr)}\Biggr\},
\]
for a sufficiently small constant $c \leq 1$.
Note that from the numeric inequality $\log(1 + x) \leq x$, this setting ensures that $\gamma \leq \tau$.  Hence, substituting this into the inequality~\eqref{eq:it-cov-kl-div-2b}, we obtain
\[
\KL(P_v,H) \lesssim b \gamma e^{-\Omega(\log(1+\tau)/\gamma)}
\lesssim q(1-\eps)\sqrt{1 + \tau} \exp\biggl\{-\Omega\biggl(\frac{\log(1+\tau)}{c} \vee \log\Bigl(\frac{q(1-\eps)n}{cd}\Bigr)\biggr)\biggr\} \lesssim \frac{d}{n},
\]
where the last inequality follows by taking $c$ sufficiently small and because $\frac{d}{nq(1-\epsilon)} \leq 1$.
It thus follows from Fano's inequality that $\cM \geq T_{\mathrm{dimension}}$.

\paragraph{Proof of $\cM \geq T_{\mathrm{confidence}}$:}
We will use the same constructions as in the previous part of the proof.  To this end, note that $H \in \cR(\mathsf{N}(0, I_d), \epsilon, \pobserved)$.  We let $v$ be any unit vector, and let $Q$ be the distribution from Lemma~\ref{lem:univariate-it-hard-instance-cov}.  In this case, we take
\[
\gamma = \min\Biggl\{\tau', \frac{\tau'}{\log\Bigl(\frac{b \cdot \tau' n}{C\log(1/(4\delta))}\Bigr)}\Biggr\} = \min\Biggl\{\tau', \frac{\tau'}{\log\Bigl(\frac{q(1-\eps)\sqrt{1+\tau}\tau' n}{C\log(1/(4\delta))}\Bigr)}\Biggr\}.
\]
Applying the KL bound~\eqref{eq:it-cov-kl-div-2b} and substituting the setting of $\gamma$ above yields
\[
\mathrm{KL}(Q,H) \lesssim b \gamma e^{-\Omega(\tau'/\gamma)} = q(1-\epsilon)\sqrt{1 + \tau} \gamma e^{-\Omega(\tau'/\gamma)} \leq 
\frac{c\log(1/(4\delta))}{n}.
\]
Hence, adjusting constants and applying the Bretagnolle--Huber inequality, we obtain
\[
\mathrm{TV}(Q^{\otimes n}, H^{\otimes n})
\leq \sqrt{1 - e^{-n\mathrm{KL}(Q,H)}} \leq 1 - 2\delta.
\]
We conclude by applying a high-probability variant of Le Cam's method~\cite[Lemma 5]{ma2024high}, which yields the lower bound
\[
\cM \geq \gamma = T_{\mathrm{confidence}},
\]
as desired.

\paragraph{Proof of the reduction~\eqref{ineq:sufficient-cond-minimax-cov-lb}:}
Re-phrasing \eqref{ineq:sufficient-cond-minimax-cov-lb}, we have the minimax lower bound
\begin{align*}
    \cM
    \gtrsim
    \alpha + \min\Biggl\{c, \frac{\log(1+\tau)}{\log(1/c\alpha^2)}\Biggr\}
    \asymp \alpha + \frac{\log(1+\tau)}{\log(1+\tau) + \log(1/\alpha)},
\end{align*}
where $\alpha = \sqrt{\frac{d+\log(1/\delta)}{nq(1-\epsilon)}} \leq 1/\sqrt{C_1}$ and $c$ is a sufficiently small constant.
By considering several cases, we will show that this is equivalent to the claimed lower bound in Theorem~\ref{thm:cov-est-it-ub-lb}.
\begin{itemize}
\item 
\textbf{Case 1: $\tau \lesssim \alpha\log(1/\alpha)$.}  This regime is immediate because
\begin{align*}
  \frac{\log(1 + \tau)}{\log(1 + \tau/\alpha^2)}
  \asymp \alpha \lesssim \cM.
\end{align*}

\item 
\textbf{Case 2: $\alpha \log(1/\alpha) \lesssim \tau \lesssim \alpha^{-2}$.}
This regime follows because both $\log(1+\tau)$ and $\log(1+\tau^2/\alpha^2)$ are within constant factors of $\log(1/\alpha)$.

\item
\textbf{Case 3: $\tau \gtrsim \alpha^{-2}$.}
This regime is immediate because $\log(1+\tau) \asymp \log(1+\tau/\alpha^2) \gtrsim \log(1/\alpha)$.
\end{itemize}
Because we have covered all cases of $\tau$, the claim follows.\qed

\subsubsection{Linear regression: Proof of Theorem~\ref{thm:lin-reg-ub-lb} lower bound}\label{sec:proof-linreg-est-it-lb}
\begin{proof}
We prove the claim for general $q \in (0, 1]$.
We can assume without loss of generality that $\sigma^2=1$, as otherwise we can scale the responses by $\sigma^{-1}$.
We introduce the quantities $\tau = \frac{\epsilon}{q(1-\epsilon)}$, $\tau' = \log(1+\tau)$, and $b = q(1-\epsilon)\sqrt{1+\tau}$.
By the same argument as in the proof of the lower bound of \cref{thm:mean-est-it-ub-lb}, it suffices to show that
\begin{align}
\label{ineq:sufficient-cond-minimax-lr-lb}
\cM := \cM(\delta, \cP_{\rm LR}(1, \epsilon),L) \gtrsim T_{\mathrm{clean}} \vee T_{\mathrm{dimension}} \vee T_{\mathrm{confidence}},
\end{align}
where 
\begin{align*}
    T_{\mathrm{clean}} := &\sqrt{\frac{d + \log(\frac{1}{\delta})}{nq(1-\epsilon)}}, \quad T_{\mathrm{dimension}} := \sqrt{\tau'} \wedge \frac{\tau '}{\sqrt{\log\left(3 + \frac{C n b\tau'^2}{d}\right)}}\, , \\
    \quad &\text{ and } \quad T_{\mathrm{confidence}} := \sqrt{\tau'} \wedge \frac{\tau '}{\sqrt{\log\left(3 + \frac{Cn b\tau'^2}{\log(1/2\delta)}\right)}},
\end{align*}
for a sufficiently large constant $C$.
The lower bounds $T_{\mathrm{clean}}$ and $T_{\mathrm{confidence}}$ follow by appropriate modifications to the arguments in the proof of the lower bound of \cref{thm:mean-est-it-ub-lb}, so we focus on establishing $T_{\mathrm{dimension}}$.
The final claim then follows from \eqref{ineq:sufficient-cond-minimax-lr-lb} by an identical argument as in \cref{thm:mean-est-it-ub-lb}.

Let $\gamma:=\gamma(\epsilon,q,n,\dimension)$ denote the resulting error bound.
Let $r \in \R_+$ be the $x$-truncation parameter such that $\gamma^2r^2 \leq 0.5 \tau'$. 
Let $R \in \R_+$ be the $y$-truncation parameter such that $R := \frac{\tau'}{10 \gamma r}$ and observe that $\gamma r \leq R/2$.
We shall choose $r = 1$ in the remainder of the proof. Thus, the following arguments will be applicable as long as $\gamma \lesssim \sqrt{\tau'}$.

Let $F_v$ be the distribution over $(X,y)$ such that $X \sim \mathsf{N}(0,I_d)$ and $y\sim \mathsf{N}(\theta^\top X,1)$.
Let $\cN$ be a $1/2$-packing of unit vectors satisfying $|\cN| \leq 5^d$.
For each $v \in \cN$, we will define $P_v \in \cR(F_{\gamma v}, \eps, q)$ with the additional property that $P_v(\star^d) = 1-b$.
To prove our desired lower bound of $\Omega(\gamma)$, it will suffice by Fano's inequality (\cref{lem:fano}) to show that there exists a distribution $H$ such that  $\KL(P_v, H) \lesssim \frac{d}{n}$ for all $v \in \cN$.
We will take $P_v$ to be the distribution $P_*$ from \cref{lem:univariate-it-hard-instance-lr} with parameters $b, \gamma, r, R$ (and noting that the parameter requirements are satisfied)
and $H$ to be defined as
\begin{align} \label{def:central-H-mean-cov}
H\bigl(\{\star^{d}\}\bigr) = 1 - b \quad \text{ and } \quad H(x) = b \cdot F_0(x, y) \text{ for all } x \in \mathbf{R}^d \text{ and } y \in \mathbf{R}.
\end{align}
Because $P_v(\star^d)=H(\star^d) = b$ and both distributions are supported on $\R^d \cup \{\star^d\}$, we have that
\begin{align}
\label{eq:it-lr-kl-div-1}
    \mathrm{KL}(P_v,H) = b \cdot \mathrm{KL}(P_v',H') \,,
\end{align}
where $P_v'$ and $H'$ are the conditional distributions of $\R^d$.

Observe that the density ratio satisfies (with $p_0$ being the density of $F_0$)
\begin{align*}
    \frac{p_v'(x,y)}{p_0'(x,y)} = \frac{\phi(x)A_x(y)}{\phi(x)\phi(y)} 
    &= \begin{cases}
        \beta_x & \text{ if } |x^\top v|\leq r \text{ and } |y| \leq R\\ 
        \frac{e^{-(y - \gamma v^\top x)^2/2}}{e^{-y^2/2}} & \text{otherwise}
    \end{cases}\\
&= \begin{cases}
        \beta_x & \text{ if } |x^\top v|\leq r \text{ and } |y| \leq R\\ 
        e^{\gamma y x^\top v -\gamma^2 (v^\top x)^2/2} & \text{otherwise}
    \end{cases}\end{align*}
Therefore,
\begin{align}
\nonumber
    \log \frac{p_v'(x,y)}{p_0'(x,y)} & = \begin{cases}
        \log \beta_x & \text{ if } |x^\top v|\leq r \text{ and } |y| \leq R\\ 
        \gamma y x^\top v -\gamma^2 (v^\top x)^2 & \text{otherwise}
    \end{cases} \\
    &\leq \gamma (y x^\top v - \gamma (v^\top x)^2)1_{|x^\top v| >  r \text{ or } |y| > R }\,,
   \label{eq:it-lr-kl-div-2}
\end{align}
where we use that $\beta_x \leq 1$.
Taking expectation under $P_v'$ we get that
 the KL divergence is upper bounded by
 \begin{align*}
     \KL(P_v',F_0)/\gamma
     &\leq  \E_{(x,y)\sim P_v'}\left[(y x^\top v -\gamma (v^\top x)^2)1_{|x^\top v| >  r \text{ or } |y| > R } \right] \\
     &= \int_{ \left\{|x^\top v| >  r\right\} \cup \left\{  |y| > R\right\} } (y x^\top v - \gamma (v^\top x)^2) \phi(x) A_x(y)dx dy \\
     &= \int_{\left\{|z| >  r\right\} \cup \left\{  |y| > R\right\}} (yz  -\gamma z^2)\phi(z) \phi(y; \gamma z,1 ) dz dy \\
     &\tag*{(definition of $A_x$ and $z:= x^\top v$)}\\
     &=\E_{G , Z}[((\gamma Z + G)Z - \gamma Z^2 ) 1_{|Z|> r/\gamma \text { or } |\gamma Z + G|> R}]\\
     &=\E_{G , Z}[GZ 1_{|Z|> r \text { or } |\gamma Z + G|> R}],
 \end{align*}
 where $Z$ and $G$ are independent standard normals, and the equality follows from the linear model representation  $Y = \gamma Z + G $ for the Gaussian linear model, which is the same as $A_x$ in that interval.
Continuing further we get
\begin{align*}
    \KL(P_v',F_0)/\gamma &\leq \E_{G , Z}[ GZ1_{|Z|> r \text { or } |\gamma Z + G|> R}]\\
    &=\left(\E_{G,Z}[  GZ 1_{|Z|> r \text { or } |\gamma Z + G|> R}] -   \E_{G , Z}[GZ 1_{|Z|> r \text { or } |G|> R}] \right)\\
    &     \tag*{(by symmetry $\E_{G , Z}[GZ 1_{|Z|> r/\gamma \text { or } |G|> R}]=0$ )}\\
    &=\E_{G,Z}\left[  GZ \left( 1_{|Z|> r \text { or } |\gamma Z + G|> R}] -    1_{|Z|> r \text { or } |G|> R}] \right)\right] \\
    &=\E_{G,Z}\left[  GZ 1_{|Z|\leq  r} \left( 1_{|\gamma Z + G|> R} -    1_{ |G|> R} \right)\right] \\
    &\leq\E_{G,Z}\left[  |G| \cdot |Z| 1_{|Z|\leq  r} \left| 1_{|\gamma Z + G|> R} -    1_{ |G|> R} \right|\right]
\end{align*}
Since $|\1_{|x + y| \geq a} - \1_{|x|\geq a}| \leq \1_{x \in [-|a| - |y|, -|a|+|y|]} + \1_{x \in [|a| - |y|, |a|+|y|]}$, we get that
\begin{align*}
    \KL(P_v',F_0)/\gamma
    &\leq\E_{G,Z}\left[  |G| \cdot |Z| 1_{|Z|\leq  r} \left( 1_{G \in [-R - \gamma |Z|, -R + \gamma |Z|] \cup [R-\gamma|Z|, R + \gamma|Z|]}  \right)\right] \\
    &=2\E_{G,Z}\left[  |G| \cdot |Z| 1_{|Z|\leq  r}  1_{G \in [R-\gamma|Z|, R + \gamma|Z|]}  \right] 
    \tag*{(by symmetry)}\\
    &\lesssim \E_{G,Z}\left[  |R +  \gamma r| \cdot |Z| 1_{|Z|\leq  r} 1_{G \in [R-\gamma|Z|, R + \gamma|Z|]}  \right] \\
    &\lesssim \max(\gamma r,R)\E_{G,Z}\left[   |Z| 1_{|Z|\leq  r} \P\{G \in [R-\gamma|Z|, R + \gamma|Z|]\}  \right] \\
    &\leq\max(r\gamma,R)\E_{G,Z}\left[   |Z| 1_{|Z|\leq  r} \cdot  \gamma|Z| e^{-(R - \gamma|Z|)^2/2}   \right] \\
    &\lesssim R\E_{G,Z}\left[   \gamma Z^2 e^{-\Omega(R^2)}  \right] \tag*{($r \gamma \leq R/2$)} \\
    &\lesssim  R\gamma e^{-\Omega(R^2)}  \\
    &\lesssim  \gamma e^{-\Omega(R^2)}  \,,
    \numberthis \label{eq:it-lr-kl-div-2}
\end{align*}
where we use
the fact that $xe^{-\Omega(x^2)} \lesssim e^{-\Omega(x^2)}$ for all $x > 0$.
Combining \Cref{eq:it-lr-kl-div-1,eq:it-lr-kl-div-2},
the KL divergence is upper bounded by
\begin{align*}
    \mathrm{KL}(P_v,H) \leq b \gamma^2e^{-\Omega(R^2)} \leq b \gamma^2 e^{- c \tau'^2/\gamma^2}\,.
\end{align*}
where we use that $R \asymp \tau'/\gamma$. 
Since the expressions and parameter restrictions are identical to those encountered 
in the proof of \cref{thm:mean-est-it-ub-lb} (see \Cref{pf:mean-est-it-lb}),
we will get a lower bound of
\begin{align*}
    \min\left( \sqrt{\tau'},  \frac{\tau'}{\sqrt{\log\left(3 + \frac{C nb \tau'^2}{d}\right)}} \right)\,.
\end{align*}
\end{proof}

\section{Proofs of sum-of-squares upper bounds}

\subsection{Mean estimation: Proof of Theorem~\ref{thm:mean-est-sos-upper-bound}}\label{app:pf-mean-est-sos-ub}
    We will prove both claims hold given $\mc{E}$. The result then follows from Lemma~\ref{lem:lower-bound-E-prob}.
    
    \par\medskip\noindent\textbf{Satisfiability.}\quad
    We first show that the constraint~\eqref{eqn:sos-constraint-mean-est} is satisfiable.
    To this end, let $\theta = \theta_\star$. 
    Let $\xi = \E_{x \sim \Tstar + \Tprime} [x^{\otimes 2k}] - \E_{z \sim \mathsf{N}(0,I_d)} [z^{\otimes 2k}]$ (recall $\Tstar + \Tprime$ consists of $n$ i.i.d.\ standard Gaussian observations).
    Applying Lemma~\ref{lem:sos-mean-est-gauss} (with $t = {d^{k}\norm{\xi}_\infty}$) gives
    \begin{align*}
        \bigl\{\norm{v}_2^2 = 1\bigr\}\sos{v}{4k} \E_{x \sim \Tstar + \Tprime} \bigl[\langle v, x-\theta \rangle^{2k}\bigr] \leq \E_{G \sim \mathsf{N}(0,1)}[G^{2k}] + d^k \norm{\xi}_\infty,
    \end{align*}
    from which $\mc{E}$ implies 
    \begin{align*}
        \bigl\{\norm{v}_2^2 = 1\bigr\}\sos{v}{4k} \E_{x \sim \Tstar + \Tprime} \bigl[\langle v, x-\theta \rangle^{2k}\bigr] 
        &\leq \E_{G \sim \mathsf{N}(0,1)}[G^{2k}] + \eps \\
        &\leq (1+\eps) \E_{G \sim \mathsf{N}(0,1)}[G^{2k}].
    \end{align*}
    
    Because the difference (up to normalization) between 
    \[
    \E_{x \sim T} \bigl[\langle v, x-\theta \rangle^{2k}\bigr] \quad \text{ and } \quad \E_{x \sim \Tstar + \Tprime} \bigl[\langle v, x-\theta \rangle^{2k}\bigr]
    \]
    is a sum-of-squares, i.e.,
    \begin{align*}
        \frac{|T|}{n}\E_{x \sim T} \bigl[\langle v, x-\theta \rangle^{2k}\bigr] +\frac{|\Tprime - \Tdprime|}{n}&\E_{x \sim \Tprime - \Tdprime} \bigl[\langle v, x-\theta \rangle^{2k}\bigr] \\
        \qquad\qquad&= \E_{x \sim \Tstar + \Tprime} \bigl[\langle v, x-\theta \rangle^{2k}\bigr],
    \end{align*}
    it then follows that
    \begin{align*}
        \E_{x \sim T} \bigl[\langle v, x-\theta \rangle^{2k}\bigr] \leq \frac{|\Tstar + \Tprime|}{|T|} (1+\eps) \E_{G \sim \mathsf{N}(0,1)}[G^{2k}].
    \end{align*}
    Then $\mc{E}$ implies 
    \begin{align*}
        \E_{x \sim T} \bigl[\langle v, x-\theta \rangle^{2k}\bigr] 
        &\leq \frac{|\Tstar + \Tprime|}{|T|} (1+\epsilon) \E_{G \sim \mathsf{N}(0,1)}[G^{2k}] \\
        &\leq \frac{(1+\epsilon)^2}{1-\epsilon} \E_{G \sim \mathsf{N}(0,1)}[G^{2k}].
    \end{align*}
    
    \par\medskip\noindent\textbf{Accuracy.}\quad
    Let $\pE$ be a degree-$4k$ pseudoexpectation over $\theta \in \R^d$ satisfying the constraint~\eqref{eqn:sos-constraint-mean-est}.  Note that 
    \begin{align*}
        \left(\frac{(1+\epsilon)^2}{1-\epsilon}\right) \E_{G \sim \mathsf{N}(0,1)}[G^{2k}] 
        &\geq \pE\bigl[\E_{x \sim T} \langle v, x-\theta\rangle^{2k} \bigr]\\
        &\geq \pE\biggl[\frac{|\Tstar|}{|T|}  \E_{x \sim \Tstar} \langle v, x - \theta \rangle^{2k} \biggr] \\
        &\geq \pE\bigl[(1- \epsilon) \E_{x \sim \Tstar} \langle v, x - \theta\rangle^{2k}\bigr] \\
        &\geq (1- \epsilon) \E_{x \sim \Tstar} \langle v, x - \pE[\theta]\rangle^{2k},
    \end{align*}
    where the last step follows by combining Lemmas~\ref{lem:pe-holder} and \ref{lem:pe-jensen}.

    For convenience, let 
    \[
    \widehat{p}_\ell = \inf_{v \in \mathbf{S}^{d-1}} \E_{x \sim \Tstar} \bigl[\langle v, x - \theta_\star\rangle^\ell \bigr] \quad \text{ and } \quad \Delta = \bigl \| \theta_\star - \pE[\theta]\bigr \|_2.
    \]
    Then, taking $v = \frac{\pE[\theta] - \theta_\star}{\Delta}$ and dividing the previous display by $1-\epsilon$ on both sides gives
    \begin{align*}
        \left(\frac{1+\epsilon}{1-\epsilon}\right)^2 \E_{G \sim \mathsf{N}(0,1)}[G^{2k}] \geq \widehat{p}_{2k} + \sum_{i=1}^{k - 1} \left\{ {2k \choose 2i } \widehat{p}_{2k-2i} \Delta^{2i} -  {2k \choose 2i -  1} \widehat{p}_{2k-2i + 1} \Delta^{2i-1} \right\} + \Delta^{2k}.
    \end{align*}
    The event $\mc{E}$ implies via Lemma~\ref{lem:tensor-concentration-all} that $|\widehat{p}_{j}  - \E_{G \sim \mathsf{N}(0,1)}[G^{j}]| \leq d^{j/2 - k} \eps$ for $j \in [2k]$. Recall $\gamma = \left(\frac{(1+\epsilon)^2 - (1-\epsilon)^3}{(1-\epsilon)^2}\right)$.
    Rearranging the previous display by moving $\widehat{p}_{2k}$ to the other side gives
    \begin{align}
        \sum_{i=1}^{k - 1} \left\{ {2k \choose 2i } \widehat{p}_{2k-2i} \Delta^{2i} -  {2k \choose 2i -  1} \widehat{p}_{2k-2i + 1} \Delta^{2i-1} \right\} + \Delta^{2k} \leq \gamma \E_{G \sim \mathsf{N}(0,1)}[G^{2k}]. \label{eqn:sos-mean-est-util-1}
    \end{align}
    Let $\Delta_0 = \max_{i \in [k]} 2 \bigl| \frac{\widehat{p}_{2k-2i+1}}{\widehat{p}_{2k-2i}} \frac{2i}{2k-2i+1} \bigr|$
    and suppose that $\Delta > \Delta_0$.
    Then each parenthesized summand in the display~\eqref{eqn:sos-mean-est-util-1} is at least ${2k \choose 2i } \widehat{p}_{2k-2i} \Delta^{2i} \left(1 -\frac{\Delta_0}{2\Delta}\right)$.
    Recall $\mc{E}$ implies $\widehat{p}_{2k-2i} \geq 0$ for all $i \in [k]$ (since these are even-degree polynomials with expectation at least $1$ regardless of $i$), so taking $i=1$ and discarding all other terms gives
    \begin{align*}
    \frac{2k(2k-1)\widehat{p}_{2k-2}}{16} \Delta^2 \leq \gamma \E_{G \sim \mathsf{N}(0,1)}[G^{2k}]\,.  
    \end{align*}
    Thus, we obtain
    \begin{align*}
        \Delta 
        &\leq \max \left(\Delta_0, \frac{4}{k} \sqrt{ \frac{\gamma \E_{G \sim \mathsf{N}(0,1)}[G^{2k}]}{\widehat{p}_{2k-2}}}\right) \\
        &\leq \max \left(\Delta_0, 8 \sqrt{\frac{\gamma}{k}}\right),
    \end{align*}
    where the last step follows from that fact that the event $\mc{E}$
    implies $\widehat{p}_{2k-2} \geq 0.5(2k-3)!!$ (recall $\E_{G \sim \mathsf{N}(0,1)}[G^{2k}] = (2k-1)!!$).

    All that remains is to upper bound $\Delta_0$. Recall $\mc{E}$ implies $|\widehat{p}_{2k-2i+1}| \leq \frac{\epsilon d^{1/2-i}}{2k}$ for all $i \in [k]$ (since these are odd-degree polynomials with expectation zero). Likewise, we can lower bound $\widehat{p}_{2k-2i} \geq 0.5 (2k-2i-1)!!$. Finally, recall $k \leq \sqrt{d}$ by assumption.
    Putting these facts together gives $\Delta_0 \leq 8 \frac{\epsilon}{k} \leq 8 \frac{\gamma}{k}$. \hfill $\qed$

    \subsection{Covariance estimation: Proof of Theorem~\ref{thm:cov-est-sos-upper-bound}}
    We will prove both claims hold given $\mc{E}$. The result then follows from Lemma~\ref{lem:lower-bound-E-prob}. Throughout the proof, we let $\widetilde{\Sigma} = \E_{x \sim T}[xx^T]$ be the empirical covariance of the observed data.

    \par\medskip\noindent\textbf{Satisfiability.}\quad %
    We first show that $\mc{A}_{\text{deviation}}$ is satisfiable f by $M = (\widetilde{\Sigma}^{1/2} \Sigma^{-1} \widetilde{\Sigma}^{1/2})^{1/2}$ and $B = M - I_d$ for $\alpha = 10\eps$. First, observe both matrices are symmetric and $M \succeq 0$ (recall that we assume $\Sigma \succ 0$).
    The event $\mc{E}$ implies for $S \in \{\Tstar,\Tstar + \Tprime\}$ that
    \begin{align*}
        \bigl\|\Sigma^{-1/2}\E_{x \sim S}[ x x^T ] \Sigma^{-1/2} -  I_d\bigr\|_{\rm op} \leq d \bigl \|\Sigma^{-1/2} \E_{x \sim S}[ x x^T ] \Sigma^{-1/2} -  I_d\bigr\|_\infty \leq \eps,
    \end{align*}
    and thus $(1-\eps) \Sigma \preceq \E_{x \sim S}[ x x^T ] \preceq (1+\eps)\Sigma$. Also, $\mc{E}$ implies $m \geq (1-10\epsilon)n$. Thus, after renormalization we have
    \begin{align*}
        \left(\frac{1-\eps}{1+\eps}\right) \Sigma  \preceq \E_{x \sim T}[ x x^T ] \preceq \left(\frac{1+\eps}{1-\eps}\right)\Sigma,
    \end{align*}
    which after whitening becomes
    \begin{align*}
        \left(\frac{1-\eps}{1+\eps}\right) \widetilde{\Sigma}^{-1/2}\Sigma \widetilde{\Sigma}^{-1/2}  \preceq I_d \preceq \left(\frac{1+\eps}{1-\eps}\right)\widetilde{\Sigma}^{-1/2} \Sigma \widetilde{\Sigma}^{-1/2}
    \end{align*}
    and after inverting becomes
    \begin{align*}
        \left(\frac{1-\eps}{1+\eps}\right) \widetilde{\Sigma}^{1/2}\Sigma^{-1} \widetilde{\Sigma}^{1/2}  \preceq I_d \preceq \left(\frac{1+\eps}{1-\eps}\right)\widetilde{\Sigma}^{1/2} \Sigma^{-1} \widetilde{\Sigma}^{1/2}.
    \end{align*}
    Because $\eps < 1/100$, we can bound $\sqrt{\frac{1+\eps}{1-\eps}} \leq 1+5\eps$. So, from the fact that $f(t) \to \sqrt{t}$ is operator monotone it follows by taking square roots of the above display that $M$ and $B$ satisfy the constraints in $\mc{A}$ for $\alpha=10\eps$.

    Now all that remains is to show satisfiability of the $\mc{A}_{\mathrm{moments}}$ for the same choice of $M$.
    Let $\ell \in [k]$ and denote $\xi_S =  \E_{x \sim S} [(\Sigma^{-1/2}x)^{\otimes 2\ell}] - \E_{z \sim \mathsf{N}(0,I_d)} [z^{\otimes 2\ell}]$. We prove the upper bound by taking $S = \Tstar + \Tprime$; the lower bound follows analogously by taking $S = \Tstar$.
    Because $\widetilde{\Sigma}^{-1/2} x \sim \cN(0,M^{-2})$,
    applying Lemma~\ref{lem:sos-mean-est-gauss} (with $t = d^\ell\norm{\xi_{\Tstar + \Tprime}}_\infty$) gives 
    \begin{align*}
        \bigl\{\norm{v}_2^2 = 1\bigr\}\sos{v}{4k} \E_{x \sim \Tstar + \Tprime} \bigl[\langle v, M \widetilde{\Sigma}^{-1/2} x \rangle^{2\ell}\bigr] 
        & \leq \E_{G \sim \mathsf{N}(0,1)}[G^{2\ell}] + d^\ell \norm{\xi_{\Tstar + \Tprime}}_\infty,
    \end{align*}
    from which $\mc{E}$ implies 
    \begin{align*}
        \bigl\{\norm{v}_2^2 = 1\bigr\}\sos{v}{4k} \E_{x \sim \Tstar + \Tprime} \bigl[\langle v, M \widetilde{\Sigma}^{-1/2} x \rangle^{2\ell}\bigr] 
        & \leq \E_{G \sim \mathsf{N}(0,1)}[G^{2\ell}] + \eps \\
        & \leq (1+\eps) \E_{G \sim \mathsf{N}(0,1)}[G^{2\ell}].
    \end{align*}
    
    Because the difference between 
    \[
    \E_{x \sim T} \bigl[\langle v, M \widetilde{\Sigma}^{-1/2} x \rangle^{2\ell}\bigr] \quad \text{ and } \quad \frac{|T|}{n} \E_{x \sim \Tstar + \Tprime} \bigl[\langle v, M \widetilde{\Sigma}^{-1/2} x \rangle^{2\ell}\bigr],
    \]
    is a sum-of-squares, it then follows that 
    \begin{align*}
        \left\{\norm{v}_2^2 = 1\right\}\sos{v}{4k} \E_{x \sim T} \left[\langle v, M \widetilde{\Sigma}^{-1/2} x \rangle^{2\ell}\right] 
        &\leq \frac{n}{|T|} (1+\eps) \E_{G \sim \mathsf{N}(0,1)}[G^{2\ell}] \\
        & \leq \frac{(1+\eps)^2}{1-\eps} \E_{G \sim \mathsf{N}(0,1)}[G^{2\ell}].
    \end{align*}
    Using our assumption that $\eps < 1/100$, we can bound $\frac{(1+\eps)^2}{1-\eps} \leq 1+10\eps$. Thus, $M$ and $B$ satisfy the upper bound in the constraints~\eqref{eqn:sos-constraint-cov-est} over $\ell \in [k]$.

    \par\medskip\noindent\textbf{Accuracy.}\quad 
    Let $\beta = 10\eps$ and $A_{\rm OPT} = \widetilde{\Sigma}^{-1/2} \Sigma \widetilde{\Sigma}^{-1/2}$ for convenience.
    We will show (Lemma~\ref{lem:cov-sos-upper-bound-acc-helper}) the event $\mc{E}$ implies 
    \begin{align*}
        \Bigl\{\bigl\{\norm{v}_2^2 = 1\bigr\} \sos{v}{4k} \E_{x \sim T} \langle v, M \widetilde{\Sigma}^{-1/2} x \rangle^{2k} \in \left(1\pm 8\eps \right)\bigl \|\Sigma^{1/2} \widetilde{\Sigma}^{-1/2} M v\bigr\|_2^{2k}\E_{G \sim \mathsf{N}(0,1)} [G^{2k}] + 2\eps \Bigr\}.
    \end{align*}
    Together with the fact that $\pE$ satisfies constraint~\eqref{eqn:sos-constraint-cov-est}, this implies for $\beta = 10\eps$ that
    \begin{align*}
         \pE\left[\bigl\|\Sigma^{1/2} \widetilde{\Sigma}^{-1/2} M v\bigr\|_2^{2k}\right] \in \left[\frac{1-\beta}{1+\beta}, \frac{1+\beta}{1-\beta}\right],
    \end{align*}
    i.e., $\pE\bigl[\bigl\|\Sigma^{1/2} \widetilde{\Sigma}^{-1/2} M v\bigr\|_2^{2k}\bigr] \in 1 \pm 10\beta$ (since $\beta < 1/2$).
    Combining lemmas~\ref{lem:pe-1-m-alpha} and \ref{lem:pe-1-p-alpha} then gives
    $\pE\bigl[\bigl\|\Sigma^{1/2} \widetilde{\Sigma}^{-1/2} M v\bigr\|_2^2\bigr] \in 1 \pm \gamma$. Because this holds for all $v \in \mathbf{S}^{d-1}$, we obtain $\bigl\|\pE[M A_{\rm OPT} M] - I_d\bigr\|_{\rm op} \leq \gamma$.
    Thus, there exists $A_*$ satisfying $\bigl\|\pE[M A_* M] - I_d\bigr\|_{\rm op} \leq \gamma$, and moreover by the triangle inequality any such $A_*$ must satisfy $\bigl\|\pE[M (A_*-A_{\rm OPT}) M]\bigr\|_{\rm op} \leq 2\gamma$. Applying Lemma~\ref{lem:cov-est-sos-upper-bound-helper} with $H = A_* - A_{\rm OPT}$ then gives $\norm{A_*-A_{\rm OPT}}_{\rm op} \lesssim  \frac{\eps}{k}$. It thus follows for any $v \in \mathbf{S}^{d-1}$ that
    \begin{align*}
        \bigl\|\Sigma^{-1/2} \widetilde{\Sigma}^{1/2} ( A_*-A_{\rm OPT})  \widetilde{\Sigma}^{1/2} \Sigma^{-1/2} v\bigr\|_2 
        & \lesssim \frac{\eps}{k} \bigl\|\Sigma^{-1/2} \widetilde{\Sigma}^{1/2}\bigr\|_{\rm op} \bigl\|\widetilde{\Sigma}^{1/2} \Sigma^{-1/2}\bigr\|_{\rm op}\\
        & \lesssim \frac{\eps}{k},
    \end{align*}
    where the last step follows from the fact that $\mc{E}$ implies $(1-\frac{1}{2}) \Sigma \preceq \widetilde{\Sigma} \preceq 2 \Sigma$. The main claim then follows immediately.

    \begin{lemma}\label{lem:cov-sos-upper-bound-acc-helper}
        The event $\mc{E}$ implies 
        \begin{align*}
            \Bigl\{\bigl\{\norm{v}_2^2 = 1\bigr\} \sos{v}{4k} \E_{x \sim T} \langle v, M \widetilde{\Sigma}^{-1/2} x \rangle^{2k} \in \left(1\pm \beta \right)\bigl\|\Sigma^{1/2}  \widetilde{\Sigma}^{-1/2} M v\bigr\|_2^{2k}\E_{G \sim \mathsf{N}(0,1)} [G^{2k}] \Bigr\}.
        \end{align*}
    \end{lemma}
    \begin{proof}
        The proof will follow a similar structure to the proof of satisfiability. As before, let $\ell \in [k]$ and denote $\xi_S =  \E_{x \sim S} [(\Sigma^{-1/2}x)^{\otimes 2\ell}] - \E_{z \sim \mathsf{N}(0,I_d)} [z^{\otimes 2\ell}]$. We prove the upper bound by taking $S = \Tstar + \Tprime$; the lower bound follows analogously by taking $S = \Tstar$.

        Let $u = \Sigma^{1/2} \widetilde{\Sigma}^{-1/2} M v$.
        Lemma~\ref{lem:sos-mean-est-gauss} implies for all $t > 0$ that
        \begin{align*}
            \bigl\{\norm{v}_2^2 = 1\bigr\}\sos{v}{4k} & \ \E_{x \sim \Tstar + \Tprime} \bigl[\langle v, M \widetilde{\Sigma}^{-1/2} x \rangle^{2\ell}\bigr] \\
            & = \E_{x \sim \Tstar + \Tprime} \bigl[\langle u, \Sigma^{-1/2} x \rangle^{2\ell}\bigr] \\
            & \leq \norm{u}_2^{2\ell}\E_{G \sim \mathsf{N}(0,1)}[G^{2\ell}] + \frac{t}{2}\norm{u}_2^{4\ell} + \frac{d^{2\ell}}{2t} \norm{\xi_{\Tstar + \Tprime}}_\infty^2.
        \end{align*}
        Using the fact that $M \preceq 2I$ and taking $t = \left(\frac{d}{4}\right)^\ell \norm{\xi_{\Tstar + \Tprime}}_\infty$, we then obtain 
        \begin{align*}
            \bigl\{\norm{v}_2^2 = 1\bigr\}\sos{v}{4k} & \ \E_{x \sim \Tstar + \Tprime} \bigl[\langle v, M \widetilde{\Sigma}^{-1/2} x \rangle^{2\ell}\bigr]  \\
            & \leq \norm{u}_2^{2\ell}\E_{G \sim \mathsf{N}(0,1)}[G^{2\ell}] + {2^{4\ell-1}t} + \frac{d^{2\ell}}{2t} \norm{\xi_{\Tstar + \Tprime}}_\infty^2 \\
            & \leq \norm{u}_2^{2\ell}\E_{G \sim \mathsf{N}(0,1)}[G^{2\ell}] 
            + (4d)^\ell \norm{\xi_{\Tstar + \Tprime}}_\infty \\
            & \leq \norm{u}_2^{2\ell}\E_{G \sim \mathsf{N}(0,1)}[G^{2\ell}] + \eps,
        \end{align*}
        where the last step follows from $\mc{E}$. 
        Because the difference between $\E_{x \sim T} \bigl[\langle v, M \widetilde{\Sigma}^{-1/2} x \rangle^{2\ell}\bigr]$ and $\frac{|T|}{n} \E_{x \sim \Tstar + \Tprime} \bigl[\langle v, M \widetilde{\Sigma}^{-1/2} x \rangle^{2\ell}\bigr]$ is a sum-of-squares, it then follows that 
        \begin{align*}
            \bigl\{\norm{v}_2^2 = 1\bigr\}\sos{v}{4k} \E_{x \sim T} \bigl[\langle v, M \widetilde{\Sigma}^{-1/2} x \rangle^{2\ell}\bigr] 
            & \leq \frac{n}{|T|} \bigl(\norm{u}_2^{2\ell}\E_{G \sim \mathsf{N}(0,1)}[G^{2\ell}] + \eps\bigr) \\
            & \leq \frac{1+\eps}{1-\eps} \bigl((\norm{u}_2^{2\ell}\E_{G \sim \mathsf{N}(0,1)}[G^{2\ell}] + \eps\bigr) \\
            & \leq (1+8\eps)\bigl(\norm{u}_2^{2\ell}\E_{G \sim \mathsf{N}(0,1)}[G^{2\ell}]\bigr) + 2\eps,
        \end{align*}
        where the last step follows from our assumption that $\eps < 1/100$.
    \end{proof}

    \subsubsection{Proof of Lemma~\ref{lem:cov-est-sos-upper-bound-helper}} \label{sec:proof-lem-cov-est-sos-upper-bound-helper}
    Since $\pE$ satisfies $B = M - I_d$, $\pE$ must satisfy $MHM = (I+B)H(I+B)$. Thus for any unit vector $v \in \mathbf{S}^{d-1}$ we have
    \begin{align}
        \left| v^T H v \right|
        &= \left| v^T\pE\left[MHM - (BH + HB) + BHB \right] v\right| \label{eqn:sos-cov-helper-eqn}\\
        &\leq \beta + \left|\pE[v^T (BH + HB) v]\right| + \pE[v^T BHB v]. \nonumber
    \end{align}
    We will proceed to separately bound the two terms in the last display.
    
    First, we have
    \begin{align*}
        \pE[v^T BH v] 
        &\leq \sqrt{\pE[(v^T BH v)^2]} \\
        &\leq \sqrt{\pE[\norm{Bv}_2^2] \pE[\norm{Hv}_2^2]} \\
        &\leq  \alpha \norm{H}_{\rm op},
    \end{align*}
    where the first step follows from Jensen's inequality (Lemma~\ref{lem:pe-jensen}); the second from Cauchy-Schwarz (Lemma~\ref{lem:pe-cauchy-schwarz}); and the last from our constraint that $B^T B \preceq \alpha^2 I$.
    We can similarly bound $\pE[v^T BH v]$ and both negations (i.e., $-\pE[v^T BH v]$ and  $-\pE[v^T BH v]$), implying 
    \begin{align*}
        \left|\pE[v^T (BH + HB) v]\right| \leq 2 \alpha \norm{H}_{\rm op}.
    \end{align*}
    
    For the second term, observe $\norm{H}_{\rm op}^2 \norm{Bv}_2^2 -\norm{H Bv}_2^2$ is a sum of squares, i.e.,  for some $C = \norm{H}_{\rm op}^2 I - H^T H \succeq 0$ and $DD^T = C$
    \begin{align*}
        \norm{H}_{\rm op}^2 \norm{Bv}_2^2 - \norm{H Bv}_2^2 = v^T D D^T v.
    \end{align*}
    Thus, we can bound
    \begin{align*}
        \pE [v^T B H B v]
        & \leq \sqrt{\pE [(v^T B H B v)^2]} \\
        & \leq \sqrt{\pE [\norm{Bv}_2^2] \pE [\norm{H Bv}_2^2]} \\
        & \leq \alpha^2 \norm{H}_{\rm op}.
    \end{align*}

    Combining our bounds on the two terms and taking the supremum over $v$ of the left-hand side of equation~\eqref{eqn:sos-cov-helper-eqn} (recall $H$ is symmetric) gives
    \begin{align*}
        \norm{H}_{\rm op} \leq \beta + 2\alpha \norm{H}_{\rm op} + \alpha^2 \norm{H}_{\rm op},
    \end{align*}
    from which the claim follows immediately. \hfill \qed

\section{Proofs of computational lower bounds}
\label{app:sq-lower-bound}
\subsection{Preliminaries}

The following lemma controls the difference in the moments between the hard univariate distributions for the information-theoretic rate and the standard Gaussian.
\begin{lemma}[Approximate moment matching]
\label{lem:it-approximate-moments}
For a distribution $A$ on $\R$ and a $k \in \mathbb N$, define  
\begin{align*}
\delta_{A,k} := \left|\E_{X \sim A}[X^k] - \E_{G \sim \mathsf{N}(0,1)}[G^k]\right|   
\end{align*}
\begin{enumerate}
    \item If $A$ is the instance from \Cref{eq:it-bound-mean-A} and $|\gamma|\leq R/2$ and $R \gtrsim  1$, then
    \begin{align*}
        \delta_{A,k} \lesssim |\gamma|e^{-\Omega(R^2)} e^{O(k \log \max(k, R))}\,.
    \end{align*}
    \item If $A$ is the instance from \Cref{eq:it-bound-cov-A} with  $R \gtrsim 1$,
    then
    \begin{align*}
        \delta_{A,k} \lesssim \gamma e^{-\Omega(R^2) + O(k \log \max(k,R))}\,.
    \end{align*} 
\end{enumerate}
\end{lemma}
\begin{proof}
We first start with $A$ as in \Cref{eq:it-bound-mean-A}.
    Observe that the difference in the moments is exactly equal to
    \begin{align*}
        ( \beta - 1)\E[G^k \1_{|G| \leq R}] + \E[(G + \gamma)^k\1_{|G+\gamma| > R} - G^k\1_{|G| > R}]\,,
    \end{align*}
    The absolute value of the first term is at most $|\beta - 1| (k-1)!! \lesssim |\gamma|e^{-\Omega(R^2)} e^{O(k \log k)}$, where we use \Cref{lem:univariate-it-hard-instance-mean} for controlling $|\beta - 1|$.
    For the second term, we upper bound its absolute value
    as follows (where we shall use using that $|a^k - b^k| \leq k|a-b|\max(|a|^{k-1},|b|^{k-1})$):
    \begin{align*}
      \E[(G + \gamma)^k&\1_{|G+\gamma| > R} - G^k\1_{|G| > R}]  \\
      &\leq \left|\E[((G + \gamma)^k - G^k)\1_{|G| > R}]\right| + \left|\E[(G + \gamma)^k \left(\1_{|G+\gamma|>R} - \1_{|G| > R}\right)]\right|\\
    &\leq \left|\E[k|\gamma| (|G| + |\gamma|)^{k-1}\1_{|G| > R}]\right| + \E[(|G| + \gamma|)^k \1_{|G| \in [R-|\gamma|,R + |\gamma|]}]\\
    &\leq \left|\E[k|\gamma| (|G| + |\gamma|)^{k-1}\1_{|G| > R}]\right| + \E[(O(R + |\gamma|))^k \1_{|G| \in [R-|\gamma|,R + |\gamma|]}]
 \end{align*}
 Now, the second term is at most $(R + |\gamma|)^{O(k)} \P(|G| \in [R - |\gamma|, R + |\gamma|])$, which is at most $|\gamma|R^{O(k)}e^{-\Omega(R^2)}$,
 where we use that $\gamma \lesssim R$.
 For the first term, we apply Cauchy-Schwarz inequality to get it is at most $|\gamma| (k + |\gamma|)^{O(k)} e^{-\Omega(R^2)}$ and then use that $\gamma \leq R$.
 Combining everything, we have shown the desired upper bound on $\delta_{A,k}$.

    Next we consider the case when $A$ is from \Cref{eq:it-bound-cov-A}.
    The difference in the moments is 
    \begin{align*}
        (\beta - 1 )\E[G^k \1_{|G| \leq R}] + \E[ (G\sqrt{1+\gamma})^k \1_{|G \sqrt{1+\gamma}| > R } - G^k\1_{|G| > R}]\,.
    \end{align*}
    The first term is at most $\gamma e^{-\Omega(R^2)} e^{ O(k \log k)}$.
    For the second term, we shall do a similar decomposition as before
    \begin{align*}
       &\left|  \E[ (G\sqrt{1+\gamma})^k \1_{|G \sqrt{1+\gamma}| > R } - G^k\1_{|G| > R}] \right| \\
        &\qquad\leq  \left|  \E[\left( (G\sqrt{1+\gamma})^k - G^k \right)\1_{|G| > R}] \right| 
+  \left|  \E[ (G\sqrt{1+\gamma})^k \cdot \left( \1_{|G \sqrt{1+\gamma}| > R} - \1_{|G| > R}\right)] \right| \\
        &\qquad\leq \left|(\sqrt{1 + \gamma})^k - 1\right| \E G^k \1_{|G| \geq R} + O(R)^k \left|\E[\1_{|G \sqrt{1+\gamma}| > R} - \1_{|G| > R}]\right| \\
        &\leq O(k \gamma) e^{k} \sqrt{\E[G^{2k}]} \sqrt{\P(|G|> R)} + R^k \P(G \in [R/\sqrt{1+\gamma},R]) ,
\end{align*}    
which is at most $\gamma e^{k \log k - \Omega(R^2)} + R^k O(\gamma) e^{-\Omega(R^2)}$. 
\end{proof}

We will also use the following technical result that perturbs a distribution $A$ so that it can turn a distribution that approximately matches moments into one that exactly matches moments.
\begin{lemma}[Fixing the moments~{\cite[Exercise 8.3]{DiaKan22-book}}]
\label{lem:fix-the-moments}
Then there exists a constant $c>0$ such that the following holds.
Let $k \in \bN$ and let $A$ be any distribution over $\R$ with $k$ finite moments such that $|\E_{A}[X^i] - \E_{G \sim \mathsf{N}(0,1)}[G^i]| \leq \alpha$ for all $i \in [k]$.
Suppose that $\inf_{|x| \leq 1} A(x) \geq 2\alpha k^{c}$.
Then there is a distribution $A'$ such that 
\begin{align*}
    |A(x) - A'(x)| \leq \begin{cases}
        \alpha k^{c} & \text{if } |x| \leq 1\\
        0  & \text{otherwise}\,.
    \end{cases}
\end{align*}
and $A'$ matches $k$ moments with $\mathsf{N}(0,1)$ exactly. In particular, $|A(x) - A'(x)|\leq \frac{1}{2}|A(x)|$ for all $x$.
\end{lemma}

\begin{claim}
\label{claim:chi-square}
    Let $Q \in \cP_\R(P,q,\eps)$.
    Let $R$ be an arbitrary distribution with finite $\chi^2(P,R)$.
    Then $\chi^2(Q,R) \lesssim         \max(1,\tau^2)\left(1 + \chi^2(P,R)\right)$.
\end{claim}
\begin{proof}
    By \Cref{cor:conditional-distribution-realizable},
    we have that  for all $x \in \R^d$:
    \begin{align*}
 \frac{L}{U}\leq  \frac{L}{b} \leq        \frac{q(x)}{p(x)} \leq  \frac{U}{b} \leq \frac{U}{L},
    \end{align*}
    where $b = \P(X \neq \star) \in (L,U)$, $L = q(1-\eps)$, and $U = q(1-\eps) + \eps$.
    Therefore, $q(x)/p(x) \leq (1 + \tau)$.
    \begin{align*}
        \chi^2(Q,R) &= \E_R\left[\left( \frac{q(X)}{r(X)}\right)^2\right] - 1 
        = \E_R\left[\left( \frac{q(X)}{p(X)}\right)^2 \left(\frac{p(X)}{r(X)}\right)^2\right] - 1 \\
        &\leq (1+\tau)^2 \left( 1 + \chi^2(P,R) \right) - 1 \lesssim
        \max(1,\tau^2)\left(1 + \chi^2(P,R)\right)\,.
    \end{align*}
\end{proof}

\paragraph{High-level proof sketch for \Cref{thm:sq-lower-bound-mean,thm:sq-lower-bound-cov}}
Both of our lower bound instances 
would be based on NGCA.
In particular, the null instance would not have any contamination, and we would set $P' = \mathsf{N}(0,I_d)$ for both \Cref{def:testing-mean-lb,def:testing-cov-lb}.
For the alternate instance,
we would set $P':= P_{A,v}$ for some univariate distribution $A$, chosen separately for the mean and the covariance testing problem.
Observe that the corresponding clean distribution in the mean case is equal to $\mathsf{N}(\rho v, I_d)$ and in the covariance setting is equal to $\mathsf{N}(0, I+\rho vv^\top)$.
These could be written compactly as $P_{F,v}$ for $F = \mathsf{N}(\delta, 1 + \gamma)$, where $\gamma = 0, $ $\delta = \rho$ for the mean case and $\delta = 0, \gamma = \rho$ for the covariance case. 

The following claim is immediate.
\begin{claim}
\label{claim:ngca-realizable}
    If $A \in \cR_ \R(F,\eps,\pobserved)$ for a univariate distribution $F$, then $P_{A,v} \in \cR_\R(P_{H,v}, \eps,\pobserved)$.
\end{claim}

We get the desired SQ hardness from \Cref{lem:ngca-lower-bound} if $A \in \cR_ \R(H,\eps,\pobserved)$ and matches $m$ moments with $\mathsf{N}(0,1)$,
where $m = \widetilde\Theta(\eps^2/\rho^2)$ for the task of mean estimation and $m =\widetilde\Theta(\eps/\rho)$ for the task of covariance estimation (while ensuring that the $\chi(A,\mathsf{N}(0,1)) \lesssim m$).

\subsection{Mean estimation: Proof of Theorem~\ref{thm:sq-lower-bound-mean}}
\label{sec:sq-lower-bound-mean}
As mentioned above, we shall use \Cref{claim:ngca-realizable} to reduce our problem to an NGCA instance.
We make a change of variable and use $\gamma =\rho$ in the following.
Given the generic SQ hardness in \Cref{lem:ngca-lower-bound},
our goal reduces to finding a univariate distribution $A' \in \cR_\R(\mathsf{N}(\gamma, 1),\eps, q )$ that matches $m$ moments and satisfies $\chi^2(A',\mathsf{N}(0,1))\lesssim 1$.
We start by establishing the existence of a moment-matching distribution, providing its proof in \Cref{sec:sq-mean-univariate-A-exact-matc}.

\begin{lemma}
\label{lem:sq-mean-univariate-A-exact-match}
Let $\eps$ and $\pobserved$ be such that $\tau:= \frac{\eps}{q(1-\eps)} \leq c$ for a small enough constant.
    There exists a univariate distribution $A' \in \cR_\R(\mathsf{N}(\gamma, 1),\eps, q )$ that matches $m = \widetilde\Theta(\frac{\eps^2}{\gamma^2})$ many moments with $\mathsf{N}(0,1)$. 
\end{lemma}

To complete the proof using \Cref{lem:ngca-lower-bound}, it remains to show that $\chi^2$ divergence is at most $O(m)$.
Applying \Cref{claim:chi-square},
we get that
\begin{align*}
    \chi^2(A',\mathsf{N}(0,1)) \lesssim 1 + \chi^2(\mathsf{N}(\gamma,1), \mathsf{N}(0,1 ) ) \lesssim 1 +  e^{\gamma^2} \lesssim 1\,,
\end{align*}
where we use that $\gamma \lesssim 1$ and $\tau \lesssim 1$.

\subsubsection{Proof of Lemma~\ref{lem:sq-mean-univariate-A-exact-match}}
\label{sec:sq-mean-univariate-A-exact-matc}
\begin{proof}[Proof of \Cref{lem:sq-mean-univariate-A-exact-match}]
Recall that we are in the regime where $\tau = \frac{\eps}{q(1-\eps)} \leq c_0$ for a small enough constant.
In this regime, $\tau = \Theta(\tau')$.
Our starting point to constructing $A'$ will be the distribution $A$ from \Cref{lem:univariate-it-hard-instance-mean},
which is a valid realizable contamination as per the lemma statement.
The parameter choices will be the same as in the proof of the lower bound in \Cref{thm:mean-est-it-ub-lb}: we choose $R$ to be $\Theta(\tau'/\gamma)= \Theta(\tau/\gamma)$.
Since $\gamma/\tau \leq  c'$ for a small enough enough absolute constant, we have that $R$ is at least a large enough constant.
Furthermore, $\R \leq \sqrt{\tau}/10$ since $\gamma \leq \tau/c$.
Observe that the target number of moments, $m = \widetilde\Theta(\tau^2/\gamma^2)$, is  $\Theta(R^2/\log R)$.

Recall that to get the desired SQ hardness, we want an $A$ that matches $m$ moments for $m\gtrsim \frac{\tau^2}{\gamma^2} \frac{1}{\log(\tau/\gamma)}$.
As shown in \Cref{lem:it-approximate-moments}, it matches roughly $m_0 \asymp  R^2/\log R$
many moments \emph{approximately}.
To be precise, there exists a small  constant $c$ such that if $m \leq cR^2/\log R$, then for all $i \in [m]$:
\begin{align*}
|\E_{A}[X^i] - \E[G^i]| \lesssim |\gamma|e^{-\Omega(R^2)}.
\end{align*}

To match $m$ moments exactly, we 
perturb this distribution to another $A'$ using \Cref{lem:fix-the-moments}.
For this lemma to be applicable, we want that 
\begin{align}
\label{eq:lower-bound-A-mass}
 \inf_{x:|x|\leq 1}A(x) \gtrsim |\gamma| e^{-\Omega(R^2)} \poly(m)  \,. 
\end{align}
To establish \Cref{eq:lower-bound-A-mass}, we shall show that, for $R\gtrsim 1$, the left hand side is at least an absolute constant,
while the right side is upper bounded by $O(e^{-\Omega(R^2)})$, and thus \Cref{eq:lower-bound-A-mass}  is satisfied for $R$ large enough.
For the left hand side, the definition of $A$ tells us that $\inf_{x:|x|\leq 1}A(x) = \beta \phi(1) \gtrsim 1$ since $\beta \geq 1/2$ for $R$ at least a large enough constant (by \Cref{lem:univariate-it-hard-instance-mean}).
For the right hand side, observe that since $m \leq R$, the right hand side is at most $O(e^{-\Omega(R^2)})$.
Thus, the resulting $A'$ will match $m$ moments exactly.
However, we still need to show that $A' \in \cR_\R(\mathsf{N}(\gamma,1),\eps,q)$ is a valid realizable contamination of $\mathsf{N}(\gamma,1)$.
By \cref{cor:conditional-distribution-realizable}, it suffices  to establish that
\begin{align}
\max_{x \in \R}\left|\log\frac{A'(x)}{\phi(x;\gamma,1)}\right| &\leq 0.5 \log(1+\tau)\,.
\label{eq:desired-A-bound-sq}
\end{align}

For $|x|>1$, this is true because $A'$ is exactly equal to $A$.
For $|x| \leq 1$, we need to ensure that the perturbations are small enough.
We have that for any $|x| \leq 1$,
\begin{align*}
\log \frac{A'(x)}{\phi(x;\gamma,1)}
&= \log \left(\frac{A(x)}{\phi(x;\gamma)}\left( 1 + \frac{A'(x) - A(x)}{ A(x)}\right)\right)\\
&= \log \frac{A(x)}{\phi(x;\gamma)} + \log \left( 1 + \frac{A(x) - A'(x)}{ A(x)}\right).
\end{align*}
By \Cref{lem:fix-the-moments}, we have that $\left|\frac{A(x) - A'(x)}{ A(x)}\right|\leq \frac{1}{2}$ and thus the inequality $|\log(1+y)|\lesssim |y|$ for $|y| \leq \frac{1}{2}$ implies that
\begin{align*}
    \left|\log \frac{A'(x)}{\phi(x;\gamma)}\right| &\leq \left|\log \frac{A(x)}{\phi(x;\gamma)} \right| + \left|\frac{A'(x)-A(x)}{A(x)}\right| \\
    &\leq 0.25 \log(1+\tau) +  O(|\gamma|e^{-\Omega(R^2)})\,,
\end{align*}
where we use that in the proof of \cref{lem:univariate-it-hard-instance-mean}, we had established the stronger inequality $\left|\log \frac{A(x)}{\phi(x;\gamma)} \right| \leq 0.25 \log(1+\tau)$ as opposed to a bound of $\left|\log \frac{A(x)}{\phi(x;\gamma)} \right|\leq 0.5 \log(1+\tau)$.
To establish \Cref{eq:desired-A-bound-sq}, we thus want that $\gamma e^{-\Omega(R^2)} \lesssim \log(1+\tau) \lesssim \tau$, meaning $e^{-\Omega(R^2)} \lesssim \tau/\gamma = \Theta(R)$,
which is satisfied as long as $R\gtrsim 1$. 
\end{proof}

\subsubsection{Covariance estimation: Proof of Theorem~\ref{thm:sq-lower-bound-cov}}
We now provide the proof of \Cref{thm:sq-lower-bound-cov}.
\begin{proof}[Proof of \Cref{thm:sq-lower-bound-cov}]
We again make a change of variable and use $\gamma =\rho$ in the remainder of the proof.
Using the same arguments in the proof of \Cref{thm:sq-lower-bound-mean},
it suffices to establish a moment-matching distribution that is a valid realizable contamination.
\begin{lemma}
\label{lem:sq-cov-univariate-A-exact-match}
Let $\eps$ and $\pobserved$ be such that $\tau:= \frac{\eps}{q(1-\eps)} \leq c$ for a small enough absolute constant.
Let $\gamma \in (0,1/2)$ such that $\tau/\gamma \gtrsim 1$.
    There exists a univariate distribution $A' \in \cR_\R(\mathsf{N}(0, 1+\gamma),\eps, q )$ that matches $m = \widetilde\Theta(\frac{\tau}{\gamma})$ many moments with $\mathsf{N}(0,1)$. 
    \end{lemma}
\begin{proof}

    We start from $A$ given in \Cref{lem:univariate-it-hard-instance-cov}, where we take $R^2 = \frac{1+\gamma}{8\gamma}\log(1+\tau) \asymp \frac{\tau}{\gamma}$.
    According to \Cref{lem:it-approximate-moments}, the resulting $A$ matches roughly $m_0=\widetilde\Theta(R^2)$ many moments with error at most $O(\gamma e^{-\Omega(R^2)})$.

    Following the same arguments as in the proof of \Cref{lem:sq-mean-univariate-A-exact-match}, we take this approximate moment-matching distribution $A$ and perturb it using \Cref{lem:fix-the-moments} to get an $m_0$ moment-matching distribution $A'$.
    This can be done as long as the density of $A$ is bounded by below by $O(|\gamma| e^{-\Omega(R^2)})\poly(m)$.
    As in the proof of \Cref{lem:sq-mean-univariate-A-exact-match}, this is satisfied as long as $R \gtrsim 1$.
    Next we need to verify that $A'$ is a valid realizable contamination of $\mathsf{N}(0,1+\gamma)$ in the sense of \Cref{eq:desired-A-bound-sq}.
    For this we need to show that the likelihood ratio between $A'$ and $\mathsf{N}(0,1+\gamma)$ is bounded appropriately, meaning that $\left|\log \frac{A'(x)}{\phi(x)}\right| \leq 0.5 \log(1+\tau)$ for all $x \in \R$.
    Doing the same calculations as in \Cref{lem:sq-mean-univariate-A-exact-match}, we get that
\begin{align*}
        \left|\log \frac{A'(x)}{\phi(x)}\right| &\leq \left|\log \frac{A(x)}{\phi(x)} \right| + \left|\frac{A'(x)-A(x)}{A(x)}\right| \\
    &\leq 0.25 \log(1+\tau) +  O(|\gamma|e^{-\Omega(R^2)})\,,
\end{align*}   
where use that $\left|\log \frac{A(x)}{\phi(x)} \right| \leq 0.25\log(1+\tau)$ if (i.) $\log(1+\gamma) \leq 0.5 \log (1+\tau)$, which is equivalent to $\gamma \lesssim \tau \lesssim 1$ in our regime and (ii.) $R^2 \leq \frac{1 + \gamma}{2\gamma} \log(1 + \tau)$; this can be seen by a simple inspection of the proof in  \Cref{lem:univariate-it-hard-instance-cov}.
For the second term to be less than $0.25 \log(1+\tau)$, it suffices that
$e^{-\Omega(R^2)} \lesssim \tau/\gamma \asymp R$,
which happens as long as $R \gtrsim 1$. As per our parameter choice, this is equivalent to requiring $\tau/\gamma \gtrsim 1$.
 \end{proof}
So far, we have established that $A'$ matches the desired number of moments, while being a valid realizable contamination.
Finally, it remains to establish an upper bound on $\chi^2(A',\mathsf{N}(0,1))$.
Using \Cref{claim:chi-square},
we have that
\begin{align*}
    \chi^2(A',\mathsf{N}(0,1)) \lesssim 1 + \chi^2(\mathsf{N}(0,1+\gamma),\mathsf{N}(0,1)) \lesssim 1,
\end{align*}
whenever $\gamma\leq 1/2$ and $\tau \lesssim 1$.
\end{proof}

\section{Extension to multiple missingness patterns for mean and covariance estimation}\label{sec:multiple-patterns}

In this section, we demonstrate how our lower bounds and algorithms can be extended to the multiple pattern setting depicted in Figure~\ref{fig:multiple-pattern}. Let us begin with a generalization of the contamination model in~\eqref{def:realizable-model}.  To this end, let $\mathbb{S}$ be a collection of subsets of $[d]$.  We extend the definitions of MCAR~\eqref{def:MCAR-aon} and MNAR~\eqref{def:MNAR-aon} as

\begin{subequations}
    \begin{align}
    \mathsf{MCAR}_{(P, \mathbb{S}, \pi)} &:= \Bigl\{\law\bigl(X \ostar \Omega\bigr):\, X \sim P \text{ and } \Omega \in \1_{\mathbb{S}}, \; \Omega \indep X, \; \mathbb{P}(\Omega = \1_{S}) = \pi_S \Bigr\} \label{def:MCAR-multiple}\\
    \mnar_{P, \mathbb{S}} &:= \Bigl\{\law\bigl(X \ostar \Omega\bigr):\, X \sim P \text{ and } \law(\Omega) \in \mathcal{P}\bigl(\1_{\mathbb{S}})\bigr)\Bigr\},\label{def:MNAR-multiple}
    \end{align}
\end{subequations}
where above we have let $\1_{\mathbb{S}} = \{\1_{S}\}_{S \in \mathbb{S}}$ and $\1_{S} \in  \{0,1\}^d$ denote the vector which takes the value $1$ on the index set $S$ and $0$ elsewhere.  We then define the natural extension of the all-or-nothing realizable contamination model~\eqref{def:realizable-model} as
\begin{align} \label{def:realizable-model-multiple}
\mathcal{R}(P, \epsilon, \mathbb{S}, \pi) := (1- \epsilon) \mcar_{(P, \mathbb{S}, \pi)} + \epsilon \mnar_{P, \mathbb{S}}.
\end{align}
In the subsequent sections, we show that, at the cost of multiplicative factors scaling polynomially in the cardinality $\lvert \mathbb{S} \rvert$, our algorithmic guarantees and lower bounds can be extended to this more general setting.  

\subsection{Lower bounds}

In order to extend our lower bounds, we rely on the intuition that revealing more data in each observation makes the problem easier.  We make this notion precise through an inflation map, defined presently.
\begin{definition}[Inflation map] \label{def:inflation}
    For $\mathbb{S} \subseteq 2^{[d]}$, we call the map $f: \mathbb{S} \rightarrow 2^{[d]}$ an \emph{inflation map} if, for all $S \in \mathbb{S}$, it holds that $S \subseteq f(S)$.  
\end{definition}
Equipped with this definition, we have the following lemma. 

\begin{lemma}[Enlarging observations makes the problem easier]\label{lem:multi-lb}
  Let $\mathbb{S} \subseteq 2^{[d]}$ and let $f: \mathbb{S} \rightarrow 2^{[d]}$ be an inflation map as in Definition~\ref{def:inflation}.  Let $\mathcal{A}$ denote an algorithm designed for observations from a distribution in $\mathcal{R}(P, \epsilon, \mathbb{S}, \pi)$.  Given observations from $\mathcal{R}(P, \epsilon, f(\mathbb{S}), f_* \pi)$, there exists an efficient algorithm which simulates $\mathcal{A}$.  
\end{lemma}
\begin{proof}
Let $f_* \pi$ denote the pushforward measure of $\pi$ under the inflation map $f$ so that for any set $S' \in \mathrm{Image}(f)$, $(f_{*} \pi)(S') = \pi\bigl(f^{-1}(S')\bigr)$.  In turn, let $g$ denote a potentially randomized mapping which satisfies $g(S') \sim \pi$ if $S' \sim f_{*}\pi$, and note that it is possible to construct this map so that $g(S') \subseteq S'$, almost surely.  Now, consider any measure $R$ contained in the inflated realizable contamination set $\mathcal{R}(P, \epsilon, f(\mathbb{S}), f_* \pi)$.  By definition, there exist random variables $X \sim P$, $\Omega^{\mnar}, \Omega^{\mcar}$ and $B$ such that 
  \begin{align*}
    X \ostar \bigl\{(1-B)\Omega^{\mcar} + B\Omega^{\mnar}\bigr\} \sim R.
  \end{align*}
Now, since $g(S') \subseteq S'$ and $g(S') \sim \pi$, we see that the random variable 
  \begin{align*}
    X \ostar \bigl\{(1-B)g\bigl(\Omega^{\mcar}\bigr) + Bg\bigl(\Omega^{\mnar}\bigr)\bigr\} \sim R_0 \in \mathcal{R}(P, \epsilon, \mathbb{S}, \pi).
  \end{align*}
It follows that any algorithm designed for observations in $\mathcal{R}(P, \epsilon, \mathbb{S}, \pi)$ can be simulated by masking observed data from $\mathcal{R}(P, \epsilon, f(\mathbb{S}), f_*\pi)$ using $g$.  
\end{proof}
Let us use this reduction to provide information-theoretic lower bounds as well as SQ lower bounds.  Both corollaries use the following instantiation of inflation map $f_I$
\begin{align*}
  f_I : S \mapsto
    \begin{cases}
      S \cup I & \text{ if } \quad I \cap S \neq \emptyset \\
      S & \textrm{ otherwise},
    \end{cases}
\end{align*}
where $I \subseteq [d]$.  Note that for any set of missingness pattern $\mathbb{S}$ and any $I$, either all or nothing of the subset $I$ is observed in any given sample.  Next, consider the error metrics $\| \theta_1 - \theta_2 \|_2$ and $\| \Sigma_1^{-1/2} \Sigma_2 \Sigma_1^{-1/2} - I_d \|_{\mathrm{op}}$ and note that for any index set $I' \subseteq [d]$, it holds that
\[
\| \theta_1 - \theta_2 \|_2 \geq \| (\theta_1)_{I'} - (\theta_2)_{I'} \|_2 \quad \text{ and } \quad \| \Sigma_1^{-1/2} \Sigma_2 \Sigma_1^{-1/2} - I_d \|_{\rm op} \geq \| (\Sigma_1)_{I', I'}^{-1/2} (\Sigma_2)_{I',I'} (\Sigma_1)_{I',I'}^{-1/2} - I_{\lvert I'\rvert} \|_{\rm op}.
\]
Hence, taking the maximum over all subsets $I \subseteq [d]$ over both the minimax quantile lower bounds in Theorems~\ref{thm:mean-est-it-ub-lb} and~\ref{thm:cov-est-it-ub-lb} and the SQ lower bounds in Theorems~\ref{thm:sq-lower-bound-mean} and~\ref{thm:sq-lower-bound-cov} and applying the reduction in Lemma~\ref{lem:multi-lb}, we conclude that the same lower bounds hold in the multiple pattern setting.  

\subsection{Upper bounds}
In the previous section, we showed that our lower bounds for mean and covariance estimation continue to hold in the multiple pattern setting.
Here, we use our algorithms to develop quantitative upper bounds for mean and covariance estimation in turn, beginning with mean estimation.
In order to obtain quantitative guarantees, we require the following pair of regularity conditions
\begin{assumption}[Constant observation probability] \label{asm:constant-observation}
    The pair $(\mathbb{S}, \pi_{\mathbb{S}})$ satisfies the constant observation probability assumption if there exists a constants $c_0 > 0$ such that 
        \[
        \min_{S \in \mathbb{S}}\; \pi_{S} > c_0.
        \]
\end{assumption}
Assumption~\ref{asm:constant-observation} ensures that each pattern is observed a constant fraction of the time.  The second assumption concerns the coverage of the patterns in $\mathbb{S}$.
\begin{assumption}[Coverage regularity]\label{asm:coverage-regularity}
The set of patterns $\mathbb{S}$ has regular coverage if both
\[
\bigcup_{S \in \mathbb{S}} S = [d].  
\]
\end{assumption}
Note that Assumption~\ref{asm:coverage-regularity} is necessary for non-trivial mean estimation.  

\subsubsection{Mean estimation}
We begin by describing how to adapt estimators designed for the all-or-nothing setting to the multiple pattern setting.  We follow the following steps:
\begin{enumerate}
\item For each $S \in \mathbb{S}$, consider the modified input that replaces all inputs whose missingness pattern is not $S$ with $(\star^{d})$ and discard all coordinates in $[d] \setminus S$. This corresponds to an all-or-nothing dataset in $\mathbb{R}^S$.
\item Run an all-or-nothing estimator (e.g.\ the estimator implicit in Theorem~\ref{thm:mean-est-it-ub-lb} or Algorithm~\ref{alg:mean-estimator}) on each dataset to obtain an estimate $\widehat{\theta}_S \in \mathbb{R}^S$ with a high probability error bound $r_S$, $\| \widehat{\theta}_S - (\theta_{\star})_S \|_2 \leq r_S$.  This implies that the true mean $\theta_\star$ lies in the cylinder 
\[
C_S := \Bigl\{ \theta \in \R^d: \; \bigl \| \widehat{\theta}_S - \theta_S  \bigr \|_2 \leq 2r_S \Bigr\}.
\]

\item Since the guarantee above holds for all patterns $S \in \mathbb{S}$ (after taking a union bound), we deduce that
\[
\theta_{\star} \in \bigcap_{S \in \mathbb{S}}\, C_{S} =: \mathcal{C}.  
\]
Moreover, since each of the cylinders $C_{S}$ are convex, we can find some element $\widehat{\theta} \in \mathcal{C}$ via convex optimization.

\end{enumerate}
After producing an estimator $\widehat{\theta}$ using the steps above, we note that since both $\widehat{\theta} \in \mathcal{C}$ and $\theta_{\star} \in \mathcal{C}$, we can produce an error bound by computing a bound on the diameter of the set $\mathcal{C}$.  

In order to make the above steps precise, three items remain.  First, in order to use our all-or-nothing algorithms, we must show that the reduction in Step 1 preserves realizability.  Second, we provide more detail on how to produce an element in $\mathcal{C}$ via convex optimization.  Finally, we provide a bound on the diameter of the set $\mathcal{C}$.  We provide the details for each of these three items in order.

\paragraph{Projection preserves realizability.} Let us begin with some notation.  For $S \subseteq [d]$, let $\Pi^S : \mathbf{R}_\star^d \rightarrow \mathbf{R}_\star^{|S|}$ be the projection onto the coordinates in $S$.  Additionally, define the censoring function $F^S : \mathbf{R}_\star^d \rightarrow \mathbf{R}^{\lvert\mathbb{S}\rvert}$ as 
\begin{align*}
  F^S(x) = \begin{cases}
      x_S & \textrm{if } x_{j} \neq \star \text{ for all } j \in S \text{ and } x_{j} = \star \text{ for all } j \in S^{c} \\
      (\star)^{|S|} & \textrm{otherwise};
  \end{cases}
\end{align*}
i.e., take the coordinates when the missingness pattern is exactly $S$ and fully censor otherwise.  The following straightforward lemma demonstrates that this censoring map preserves realizability. 
\begin{lemma}
    Suppose that $Q \in \mathcal{R}(P, \epsilon, \mathbb{S}, \pi)$.  Then, for every $S \in \mathbb{S}$, the pushforward measure $(F^S)_* Q$ satisfies the inclusion
   $(F^S)_* Q \in \mathcal{R}(\Pi^S_\star P, \epsilon, \pi(S))$.
\end{lemma}
\begin{proof}
Note that by definition there exists a tuple of random variables $X \sim P, \Omega^{\mcar}, \Omega^{\mnar}$ and $B \sim \Bern(\epsilon)$ such that
\[
X \ostar \bigl\{ (1 - B) \Omega^{\mcar} + B \Omega^{\mnar}\bigr\} \sim Q. 
\]
Now, consider the map $g: \{0, 1\}^{d} \rightarrow \{0, 1\}^{\lvert S \rvert}$ defined so that $g(\1_{S}) = \1_{\lvert S\rvert}$ (the all ones vector) and $g(v) = 0$ (the all zeros vector) for all other $v$.  Note that 
\[
\Pi^{S} X \ostar \bigl\{ (1 - B) g\bigl(\Omega^{\mcar}\bigr) + B g\bigl(\Omega^{\mnar}\bigr)\bigr\} \sim (F^S)_* Q. 
\]
Further, it holds that $B, g(\Omega^{\mcar})$, and $(\Pi^S X, g(\Omega^{\mnar}))$ are mutually independent and where $\mathbb{P}(g(\Omega^{\mcar}) = \1_{\lvert S \rvert}) = \pi(S)$.
This proves the claim.
\end{proof}

\paragraph{Computing an element $\widehat{\theta} \in \mathcal{C}$.}  We will show that an element of $\mathcal{C}$ can be computed in polynomial time by the ellipsoid method.  In order for this to hold, we require inner and outer radii $r$ and $R$ such that
\[
\mathbf{B}_2(\theta, r) \subseteq \mathcal{C} \subseteq \mathbf{B}_2(0, R),
\]
for some element $\theta \in \mathcal{C}$ and a separation oracle for the set $\mathcal{C}$.  Let us begin with a separation oracle.
Note that if $\theta \notin \mathcal{C}$, then there exists $S \in \mathbb{S}$ such that $\theta \notin C_S$.
Note that a separation oracle for balls in $\bR^{\lvert S \rvert}$ provides a separation oracle for $C_S$ by filling the coordinates in $[d] \setminus S$ with zeros.
This holds for all $S \in \mathbb{S}$ and in turn provides a separation oracle for the intersection $\mathcal{C}$.  

Next, we show that $\mathcal{C} \subseteq \mathbf{B}_2(0, R)$.  Note that by Assumption~\ref{asm:coverage-regularity}, it holds that $\cup_{S \in \mathbb{S}}\, S = [d]$.  Hence, by the triangle inequality, we deduce that for any $\theta \in \mathcal{C}$,
\[
\| \theta \|_2 \leq \sum_{S \in \mathbb{S}} \| \theta_S \|_2 \leq \sum_{S \in \mathbb{S}} 2r_S + \| \widehat{\theta}_S \|_2 =: R.
\]

Finally, we show that there exists an $r > 0$ such that $\mathsf{B}_2(\theta_{\star}, r) \subseteq \mathcal{C}$.  To this end, we take $r := r_{\min} = \min_{S \in \mathbb{S}}\, r_S$.  We then see that for any $\theta \in \mathbf{B}_2(\theta_{\star}, r)$, 
\[
\| \theta_S - \widehat{\theta}_S \|_2 \leq \| (\theta_{\star})_S - \widehat{\theta}_S \|_2 + \|  \theta_S - (\theta_{\star})_S \|_2 \leq r_S + r_{\min} \leq 2 r_S.
\]
Since the above inequality holds for all $S \in \mathbb{S}$, it holds that $\theta \in \mathcal{C}$, whence we deduce that $\mathbf{B}_2(\theta, r) \subseteq \mathcal{C}$.

It remains to bound the diameter of the set $\mathcal{C}$.

\paragraph{Bounding the diameter of the set $\mathcal{C}$.}
Let $w: \mathbb{S} \rightarrow [0,1]$ denote a weighting function such that for all $j \in [d]$, 
\[
\sum_{S: i \in S}\; w(S) \geq 1.
\]
Note that under the coverage condition in Assumption~\ref{asm:coverage-regularity}, such a weighting function exists.  Let $I_S = \mathrm{diag}(\1_S) \in \mathcal{C}^d_+$ denote the diagonal matrix which contains ones on the indices contained in the set $S$ and note that the previous display ensures that
\[
I_d \preceq \sum_{S \in \mathbb{S}}\, w(S) I_S.
\]
It follows that
\begin{align*}
    \bigl \| \widehat{\theta} - \theta_{\star} \bigr\|_2^2
    &= \bigl(\widehat{\theta} - \theta_{\star}\bigr)^{\top} I_d \bigl(\widehat{\theta} - \theta_{\star}\bigr)
    \leq \sum_{S \in \mathbb{S}}\, w(S) \cdot \bigl(\widehat{\theta} - \theta_{\star}\bigr)^{\top} I_S \bigl(\widehat{\theta} - \theta_{\star}\bigr) \\
    &\leq \sum_{S \in \mathbb{S}}\, w(S) \cdot \bigl \| \widehat{\theta} - \theta_{\star}\bigr \|_2^2 \leq 4 \sum_{S \in \mathbb{S}}\; w(S) r_S^2.  
\end{align*}
Note that it is always valid to taking the weighting function $w(S) = 1$ for all $S \in \mathbb{S}$ so that 
\[
\bigl \| \widehat{\theta} - \theta_{\star} \bigr\|_2^2 \leq 4 \lvert \mathbb{S} \rvert \cdot \max_{S \in \mathbb{S}} r_S^2.
\]
Combining this bound with the lower bound in the previous section, we see that under Assumption~\ref{asm:constant-observation} (which ensures $\lvert \mathbb{S} \rvert$ is of constant order), our algorithms are optimal in the multi-pattern setting up to a multiplicative factor of $\lvert \mathbb{S} \rvert$.

\subsubsection{Covariance estimation}
We propose a very similar algorithm as in the mean estimation setting.  Although in the main text, we provide guarantees for our estimators in relative operator norm, in order to simplify the development here, we consider guarantees in the (weaker) operator norm.  In particular, for each $S \in \mathbb{S}$, our algorithms yield estimators $\widetilde{\Sigma}_{S} \in \mathcal{C}^{\lvert S \rvert_{+}}$ which satisfy the guarantee
\[
\bigl \| \widetilde{\Sigma}_{S} - (\Sigma_{\star})_{S,S} \bigr \|_{\mathrm{op}} \leq r_S, \quad \text{ for all } \quad S \in \mathbb{S}.
\]
Proceeding as before, we define the cylinders 
\[
C_S := \Bigl\{\Sigma \in \mathcal{C}^d_{+}: \; \bigl \| \widetilde{\Sigma}_S - P_S \Sigma P_S^{\top} \bigr \|_{\mathrm{op}} \leq r_S \Bigr\},
\]
where the matrices $P_S \in \mathbf{R}^{\lvert S \rvert \times d}$ take the form $P_S = \bigl((e_{j})_{j \in S}\bigr)^{\top}$ for $e_j$ the $j$th elementary vector. We again define the intersection $\mathcal{C} = \cap_{S \in \mathbb{S}}\, C_S$ and produce an estimate $\widehat{\Sigma}$ as any element of $\mathcal{C}$.  As in the mean estimation setting, we produce such an estimate via the ellipsoid method.  Let us note that the inner radius $r$, outer radius $R$, and separation oracle can be obtained in a parallel manner to in the mean estimation setting.  

The primary departure from the setting of mean estimation comes when we bound the diameter of the set $\mathcal{C}$.  In particular, we require the following coverage condition.
\begin{assumption}[Coverage regularity for covariance]\label{asm:coverage-regularity-cov}
The set of patterns $\mathbb{S}$ has regular coverage if, for every pair $i, j \in [d]$, there exists $S \in \mathbb{S}$ such that $\{i, j\} \subseteq S$.  
\end{assumption}
Henceforth, we will assume that Assumption~\ref{asm:coverage-regularity-cov} is in force.  Next, let $I_1, \ldots, I_k$ be a partition of the coordinate set $[d]$ such that for every pair $\ell_1, \ell_2 \in [k]$, there exists a pattern $S \in \mathbb{S}$ such that the union $I_{\ell_1} \cup I_{\ell_2} \subseteq S$.  Note that under Assumption~\ref{asm:coverage-regularity-cov}, such a partition always exists.  For instance, one can take the partition into singleton sets.  Given such a partition, let $S_{\ell_1,\ell_2} \in \mathbb{S}$ satisfy $I_{\ell_1} \cup I_{\ell_2} \subseteq S_{\ell_1, \ell_2}$.  Now, for any $\widehat{\Sigma} \in \mathcal{C}$ and $v \in \mathbf{S}^{d-1}$, we have
\begin{align*}
    v^{\top} \bigl(\widehat{\Sigma} - \Sigma_{\star} \bigr)v &= v^{\top} \biggl\{ \biggl(\sum_{\ell=1}^{k} P_{I_\ell}\biggr) \bigl( \widehat{\Sigma} - \Sigma_{\star} \bigr) \biggl( \sum_{\ell = 1}^{k} P_{I_{\ell}}\biggr) \biggr\} v = \sum_{\ell_1, \ell_2 \in [k]} v^{\top} P_{I_{\ell_1}} \bigl(\widehat{\Sigma} - \Sigma_{\star}\bigr) P_{I_{\ell_1}}  v\\
    &= \sum_{\ell_1, \ell_2 \in [k]} v^{\top} P_{I_{\ell_1}} P_{S_{\ell_1, \ell_2}} \bigl(\widehat{\Sigma} - \Sigma_{\star}\bigr)  P_{S_{\ell_1, \ell_2}} P_{I_{\ell_1}} v\\
    &\leq \sum_{\ell_1, \ell_2 \in [k]} \bigl \| P_{I_{\ell_1}} v \bigr\|_2 \bigl \| P_{I_{\ell_2}} v \bigr\|_2 r_{S_{\ell_1, \ell_2}} \\
    &\leq r_{\max} \cdot \biggl\{ \sum_{\ell \in [k]} \bigl \| P_{I_\ell} v \bigr\|_2\biggr\}^{2} \leq k r_{\max},
\end{align*}
where $r_{\max} = \max_{S \in \mathbb{S}}\, r_{S}$.  We thus see that minimizing the number of partitions yields better rates.  

Let us conclude by providing a coarse upper bound on the size of a minimum partition.  As mentioned, the singleton partition always satisfies the desired condition.  This, however, gives the dimension dependent upper bound $k \leq d$.  When the number of patterns $\lvert \mathbb{S} \rvert$ is a constant, we can provide a smaller partition by taking $k = 2^{\lvert \mathbb{S} \rvert}$.  Let $\ell \in 2^{\lvert \mathbb{S} \rvert}$ and consider its binary representation.  This representation corresponds to which sets in $\mathbb{S}$ each element is a member of.  By construction, this is a partition of $[d]$.  Moreover, it satisfies the pairwise union condition.  To see this, consider any element $x_{\ell_1} \in I_{\ell_1}$ and $x_{\ell_2} \in I_{\ell_2}$ and note that by Assumption~\ref{asm:coverage-regularity-cov}, there exists some set $S \in \mathbb{S}$ such that both $x_{\ell_1}, x_{\ell_2} \in S$.  It thus follows that both $I_{\ell_1}, I_{\ell_2} \subseteq S$ so that $I_{\ell_1} \cup I_{\ell_2} \subseteq S$.  Hence, we deduce the bound $k \leq 2^{\lvert \mathbb{S} \rvert}$ so that
\[
\bigl \| \widehat{\Sigma} - \Sigma_{\star} \bigr \|_{\mathrm{op}} \leq \min\bigl\{d, 2^{\lvert \mathbb{S} \rvert}\bigr\} r_{\max}\leq C_{\lvert \mathbb{S} \rvert} r_{\max},
\]
where the final inequality follows from Assumption~\ref{asm:constant-observation} and  $C_{\lvert \mathbb{S} \rvert}$ is a constant depending only on the number of patterns.

\section{Auxiliary lemmas}

\subsection{Reduction to $q=1$: Proof of Lemma~\ref{lem:gen-q-conversion}} \label{sec:proof-lem-gen-q-conversion}
\begin{proof}
  First, suppose that $Q \in \cR(P, \epsilon, q)$, so that by Lemma~\ref{lem:all-or-nothing-characterization}, for all $z \in \bR^d$,
  \begin{align}\label{ineq:sandwich-Q}
      q(1-\epsilon) \leq \frac{\mathrm{d}Q}{\mathrm{d}P}(z) \leq q(1-\epsilon)+\epsilon = q'.
  \end{align}
  We then define the distribution $Q'$ on $\R^d \cup \{\star^{d}\}$ so that
  \[
  \frac{\mathrm{d} Q'}{\mathrm{d} P}(z) = \frac{1}{q'} \frac{\mathrm{d}Q}{\mathrm{d}P}(z), \;\; \text{ for } z \in \R^d, \quad \text{ and } \quad Q'\bigl(\{\star^{d}\}\bigr) = 1 - \frac{1}{q'} \cdot Q\bigl(\R^d\bigr).
  \]
  Note that by construction, $Q'$ is a probability distribution.  Moreover, observe that the invariant relation
  \[
  q(1 - \epsilon) = q' \cdot (1 - \epsilon'),
  \]
  holds by construction.  Consequently, applying the sandwich relation~\eqref{ineq:sandwich-Q}, we obtain the inequalities
  \[
  \frac{\mathrm{d}Q'}{\mathrm{d}P}(z)  = \frac{1}{q'} \cdot \frac{\mathrm{d}Q}{\mathrm{d}P}(z) \leq 1 \quad \text{ and } \quad \frac{\mathrm{d}Q'}{\mathrm{d}P}(z)  = \frac{1}{q'} \cdot \frac{\mathrm{d}Q}{\mathrm{d}P}(z) \geq 1 - \epsilon'.
  \]
  By Lemma~\ref{lem:all-or-nothing-characterization}, this implies that $Q' \in \mathcal{R}(P, \epsilon', 1)$.  To conclude this part, note that $q' \cdot Q'$ and $Q$ agree on the entirety of $\R^d$ so that $Q = q' \cdot Q' + (1 - q')\delta_{\{\star^{d}\}}$, since $Q$ must integrate to $    1$ on $\R^d \cup \{\star^d\}$.  

  To obtain the reverse implication, suppose that $Q' \in \mathcal{R}(P, \epsilon', 1)$.  Applying~\ref{lem:all-or-nothing-characterization} once more yields
  \begin{align*}
     q(1-\epsilon) = q' \cdot (1-\epsilon') \leq q' \cdot \frac{\mathrm{d}Q'}{\mathrm{d}P}(z) \leq q' = q(1-\epsilon) + \epsilon.
  \end{align*}
  Now, consider $Q = q' \cdot Q' + (1 - q') \cdot \delta_{\{\star\}^{d}}$.  By construction, $Q$ is a probability distribution.  Hence, from the above display, we conclude that $Q \in \cR(P, \epsilon, q)$.
\end{proof}

\subsection{Useful general lemmas}

\begin{lemma}\label{lem:tensor-concentration-helper-2}
    Let $a,b \geq 0$ and $\ell \leq 1$. Then
    \begin{align*}
        (a+b)^\ell \leq a^\ell + b^\ell.
    \end{align*}
\end{lemma}
\begin{proof}
    The case where $a,b = 0$ is trivial, so for the remainder we will assume $a + b > 0$. Observe
    \begin{align*}
        a^\ell + b^\ell 
        & = (a+b)^\ell \left(\left(\frac{a}{(a+b)}\right)^\ell + \left(\frac{b}{(a+b)}\right)^\ell\right) \\
        & \geq (a+b)^\ell,
    \end{align*}
    where the last step follows from the fact that $x^\ell + (1-x)^\ell \geq 1$ for $x \in [0,1]$.
\end{proof}

\begin{fact}\label{fact:stirling}
For all $k \in \bN$,
\begin{align*}
    \sqrt{2\pi k}\left(\frac{k}{e}\right)^k \leq k! \leq \sqrt{2\pi k}\left(\frac{k}{e}\right)^k e^{1/(12k)}.
\end{align*}
\end{fact}

\begin{lemma}\label{lem:binom-concentration}
For $n \in \bN$ and $\eps \in [0,1)$, let $m \sim \Bin(n, 1- \epsilon)$. For $\delta \in (0,1)$, suppose
\begin{align*}
  n \geq \frac{(1+\epsilon)^2}{\epsilon^2(1-\epsilon)^2} \log\frac{1}{\delta}.
\end{align*}
Then with probability at least $1-\delta$ we have
\begin{align*}
  m \geq \frac{(1-\epsilon)n}{1+\epsilon}.
\end{align*}
\end{lemma}
\begin{proof}
    Observe $\E[m] = (1-\epsilon)n$. Hoeffding's inequality thus implies for any $t>0$ that
    \begin{align*}
    \P(m-\mathbb (1-\epsilon)n\le -t) \leq \exp\left(-\frac{2t^2}{n}\right).
    \end{align*}
    The result then follows by taking $t = (1 - 1/(1+\epsilon))(1-\epsilon)n = (\epsilon/(1+\epsilon))(1-\epsilon)n$.
\end{proof}

\begin{lemma}\label{lem:psd-separation}
Let $L(\Sigma, \Sigma_\star) = \|\Sigma_\star^{-\frac{1}{2}}\Sigma \Sigma_\star^{-\frac{1}{2}} - I\|_{\rm op}$.
Let $\gamma \in [0, 1]$ and $v, \bar{v} \in \bS^{d-1}$.
Then for all $M \in \mathcal{C}_+^d$ we have
\begin{align*}
    L(M, I_d + \gamma vv^\top)
    \vee
    L(M, I_d + \gamma \bar{v}\bar{v}^\top)
    > \frac{\gamma}{16} \|v - \bar{v}\|_2^2.
\end{align*}
\end{lemma}
\begin{proof}
    Suppose $L(M, I_d + \gamma vv^\top) < \frac{\gamma}{16} \|v - \bar{v}\|_2^2$,
    so that $M = I_d + \gamma vv^\top + E$ for some $E$ with
    $\frac{1}{1+\gamma}\|E\|_{\rm op}
    \leq \|(I+\gamma vv^\top)^{-\frac{1}{2}}E(I+\gamma vv^\top)^{-\frac{1}{2}}\|_{\rm op} < \frac{\gamma}{16} \|v - \bar{v}\|_2^2$.
    Then,
    \begin{align*}
    &\left\|
        \left(I_d + \gamma \bar{v}\bar{v}^\top\right)^{-\frac{1}{2}} M \left(I_d + \gamma \bar{v}\bar{v}^\top\right)^{-\frac{1}{2}} - I_d
    \right\|_{\rm op} \\
    & = \left\|
        \left(I_d + \gamma \bar{v}\bar{v}^\top\right)^{-\frac{1}{2}} \left(I_d + \gamma vv^\top + E\right) \left(I_d + \gamma \bar{v}\bar{v}^\top\right)^{-\frac{1}{2}} - I_d
    \right\|_{\rm op} \\
    & \geq \gamma \left\|
        \left(I_d + \gamma \bar{v}\bar{v}^\top\right)^{-\frac{1}{2}} \left(vv^\top - \bar{v}\bar{v}^\top\right) \left(I_d + \gamma \bar{v}\bar{v}^\top\right)^{-\frac{1}{2}}
    \right\|_{\rm op} - \left\|
        \left(I_d + \gamma \bar{v}\bar{v}^\top\right)^{-\frac{1}{2}} E \left(I_d + \gamma \bar{v}\bar{v}^\top\right)^{-\frac{1}{2}}
    \right\|_{\rm op} \\
    & \geq \frac{\gamma}{1+\gamma} \| vv^\top - \bar{v}\bar{v}^\top \| - \frac{\gamma(1+\gamma)}{8}\|v-\bar{v}\|^2.
    \end{align*}

    Now let $u \in \bS^{d-1}$ be orthogonal to $v$ and such that $\bar{v} = \alpha v + \sqrt{1-\alpha^2} u$ for some $\alpha \in [0,1]$.
    Then $\| vv^\top - \bar{v}\bar{v}^\top \|_{\rm op} \geq (\bar{v} \cdot u)^2 = 1 - \alpha^2$ and $\|v-\bar{v}\|_2^2 = 2(1-\alpha)$.
    Thus, we have the lower bound
    \begin{align*}
        \| vv^\top - \bar{v}\bar{v}^\top \| \geq 1 - \alpha^2 \geq 1-\alpha = \frac{1}{2}\|v-\bar{v}\|_2^2.
    \end{align*}
    Plugging this in gives
    \begin{align*}
        L(M, I_d + \gamma \bar{v}\bar{v}^\top)
        \geq \left(\frac{\gamma}{2(1+\gamma)} - \frac{\gamma(1+\gamma)}{16} \right)\|v-\bar{v}\|^2
        \geq \frac{\gamma}{8}\|v-\bar{v}\|^2
        > \frac{\gamma}{16}\|v-\bar{v}\|^2,
    \end{align*}
    where the last step follows from $\gamma \in [0,1]$.
    This proves the claim.
\end{proof}

Below we present a slightly modified version of \cite[Lemma 3.2.3]{BlumerEhAnWa89}.
The only difference is that we allow the $s$-fold union to be across different classes, but this does not change the result.
\begin{lemma}[VC-dimension of $s$-fold intersections]\label{lem:vc-dim-intersection}
  Let $C_1,\dots, C_s \subseteq 2^X$ be concept classes such that $\mathrm{VC}(C_i) \leq d$ for all $i \in [s]$.
  Then the concept class $C_\cap = \{\cap_{i=1}^s c_i \mid c_i \in C_i \forall i \in [s]\}$ has VC-dimension less than
  $2ds\log(3s)$.
\end{lemma}

\subsection{Univariate Gaussian concentration and moment properties}

Throughout, recall that denote the density of a standard univariate Gaussian by the function $\phi$.
We will also use $G$ to denote a standard Gaussian.

\begin{fact}[Gaussian moments]\label{fact:gaussian-moments}
  For all $k \in \bN$,
  \begin{align*}
    \E[|G|^{2k-1}] = \sqrt{\frac{2}{\pi}} (2k-2)!!
    \quad\textrm{and}\quad
    \E[G^{2k}] = (2k-1)!!.
  \end{align*}
\end{fact}

\begin{lemma}\label{lemma:trunc-kmom-lb}
  Let $k \geq 1$ and $T \geq 2\sqrt{k}\log(3k)$.
  Then $\E[G^{2k-2}\1\{G \in [1, T]\}] \geq \frac{(2k-3)!!}{3}$.
\end{lemma}
\begin{proof}
  From \cref{fact:gaussian-moments}, we have $\E[G^{2k-2}] = (2k-3)!!$ and by symmetry of the Gaussian distribution, we have that
  \begin{align*}
    \frac{(2k-3)!!}{2} = \int_0^\infty t^{2k-2} \phi(t) dt
    = \underbrace{\int_0^1 t^{2k-2}\phi(t) dt}_{=:A} + \underbrace{\int_1^T t^{2k-2}\phi(t) dt}_{=:B} + \underbrace{\int_T^\infty t^{2k-2}\phi(t) dt}_{=:C}.
  \end{align*}
  We want to lower bound $B$, which we will do by upper bounding $A$ and $C$.

  To upper bound $A$,
  \begin{align*}
    A = \int_0^1 t^{2k-2}\phi(t) dt
    \leq \int_0^1 \frac{e^{-t^2/2}}{\sqrt{2\pi}} dt
    \leq \frac{1}{\sqrt{2\pi}}.
  \end{align*}

  To upper bound $C$,
  \begin{align*}
    C = \int_T^\infty t^{2k-2}\phi(t) dt
    \leq \frac{1}{T} \int_T^\infty t^{2k-1} \phi(t) dt.
  \end{align*}
  Noting that $\frac{d}{dt}(-\phi(t)) = t\phi(t)$, we apply integration by parts with $u = t^{2k-2}$ and $v = -\phi(t)$ to obtain
  \begin{align*}
    \int_T^\infty t^{2k-1} \phi(t) dt
    &= \int_T^\infty (t^{2k-2}) (t\phi(t)) dt
    \\&= \left(-t^{2k-2} \phi(t)\right)\big|_{T}^\infty + (2k-2)\int_T^\infty t^{2k-3}\phi(t) dt
    \\&\leq T^{2k-2}\phi(T) + \frac{2k-2}{T}C.
  \end{align*}
  Thus,
  \begin{align*}
    C \leq \frac{1}{T}\left[T^{2k-2}\phi(T) + \frac{2k-2}{T}C\right]
    \Longrightarrow C \leq \frac{1}{\sqrt{2\pi}}\left[1 - \frac{2k-2}{T^2}\right]^{-1} T^{2k-3}e^{-T^2/2}
  \end{align*}
  
  Observe that $t \mapsto t^{2k-3}e^{-t^2/2}$ is decreasing for $t \geq \sqrt{2k-3}$, so for $T \geq 2\sqrt{k}\log (3k)$ we have that
  \begin{align*}
    T^{2k-3}e^{-T^2/2}
    &\leq (2\sqrt{k}\log (3k))^{2k-3}e^{-(4k\log^2(3k))/2)}
    \leq \exp\left(-2k\log^2 (3k) + 2k\log(3k)\right)
    \leq 1
  \end{align*}
  Also for $T \geq 2\sqrt{k-1}$, $\left[1 - \frac{2k-2}{T^2}\right]^{-1} \leq 2$.
  Thus, $C \leq \frac{2}{\sqrt{2\pi}}$.

  This implies $B \geq \frac{(2k-3)!!}{2}-\frac{3}{\sqrt{2\pi}} \geq \frac{(2k-3)!!}{3}$.
\end{proof}

\begin{lemma}\label{lem:tensor-concentration-helper-3}
Let $\ell \in \bN$. Then $\E[G^{\ell}] \leq 2\left(\frac{\ell}{e}\right)^{\ell/2}$.
\end{lemma}
\begin{proof}
    The claim is trivial when $\ell$ is odd, so for the remainder let $\ell = 2k$ for some $k \in \bN$.
    Stirling's formulae (\cref{fact:stirling}) imply that
    \begin{align}
        \sqrt{2\pi k}\left(\frac{k}{e}\right)^k \leq k! \leq \sqrt{2\pi k}\left(\frac{k}{e}\right)^k e^{1/(12k)}. \label{eqn:stirling-factorial}
    \end{align}
    Thus, we have
    \begin{align*}
        \E[G^{2k}] 
        &= \frac{(2k)!}{2^k k!} \\
        &\leq \frac{\sqrt{4\pi k}(2k/e)^{2k} e^{1/(24k)}}{2^k \sqrt{2\pi k}(k/e)^k} \\
        &= \sqrt{2} e^{1/24} \left(\frac{2k}{e}\right)^k \\
        &\leq 2\left(\frac{2k}{e}\right)^k.
    \end{align*}
\end{proof}

\begin{lemma}\label{lem:gaussian-power-concentration-truncation}
    Let $G_1,\dots,G_n \simiid \mathsf{N}(0,1)$ and $m \in \bN$. Then with probability at least $1-\delta$,
    we have that 
    \begin{align*}
      \sum_{i=1}^n  |G_i|^{m}  \leq n(m-1)!! + \sqrt{n2^{m-1}\log(2/\delta)\log(2n/\delta)^{\frac{m}{2}}}.
    \end{align*}
\end{lemma}
\begin{proof}

  Let $T > 0$ be a threshold we will specify later and
  define $A_i = |G_i|^m\1\{|G_i| \leq T\}$.
  Then by a union bound
  \begin{align*}
    \P\left(\frac{1}{n}\sum_{i=1}^n |G_i|^m \geq \mu_m + \Delta_1\right)
    \leq \P\left(\frac{1}{n}\sum_{i=1}^n A_i \geq \mu_m + \Delta_1\right)
      + \P\left(\max_{i \in [n]} |G_i| > T\right).
  \end{align*}
  
  For the first term, note that because $A_i$ is supported on $[0, T]$ that $\frac{1}{4}T^{2m}$-subgaussian.
  Furthermore, $\E[A_1] \leq \mu_m$.
  Thus,
  \begin{align*}
    \P\left(\frac{1}{n}\sum_{i=1}^n A_i \geq \mu_m + \Delta_1\right)
    &\leq \P\left(\frac{1}{n}\sum_{i=1}^n A_i - \mu_m \geq \Delta_1\right) \\
    &\leq \exp\left(-2n\Delta_1^2 T^{-2m}\right)
  \end{align*}
  This is bounded by $\delta/2$ if we choose $T$ such that
  \begin{align*}
    T^2 \leq \left[\frac{2 n \Delta_1^2}{\log(2/\delta)}\right]^{\frac{1}{m}}.
  \end{align*}

  For the second term, we have by union bound that
  \begin{align*}
     \P\left(\max_{i \in [n]} |G_i| > T\right) \leq 2n \exp(-\tfrac{1}{2}T^2),
  \end{align*}
  which is at most $\delta/2$ for $T \geq 2\sqrt{\log(2n/\delta)}$.

  We can always find some $T$ satisfying both constraints as long as
  \begin{align*}
    \frac{2 n \Delta_1^2}{\log(2/\delta)}\geq 2^m \log(2n/\delta)^{\frac{m}{2}}
    \Longrightarrow \Delta_1 \geq \sqrt{\frac{2^{m-1}\log(2/\delta)\log(2n/\delta)^{\frac{m}{2}}}{n}},
  \end{align*}
  proving the claim.
\end{proof}

The next two lemmas provide basic properties of the Gaussian survivor function $\bar{\Phi}$.
\begin{lemma}[Mills' ratio bound]\label{lem:mills}
  For all $t \geq 0$,
  \[
    \frac{t}{t^2 + 1} \cdot \phi(t) \leq \bar{\Phi}(t) \leq \frac{1}{t} \cdot \phi(t).
  \]
\end{lemma}

\begin{lemma} \label{lem:Gaussian-Q-prop}
    Let $\bar{\Phi}: \R \rightarrow [0,1]$ be defined as $\bar{\Phi}(x) = \mathbb{P}(G \geq x)$, where $G \sim \mathsf{N}(0, 1)$.  The following hold:
    \begin{itemize}
        \item[(a)] The inverse function $\bar{\Phi}^{-1}$ is differentiable and its derivative is given by
        \[
        \frac{\mathrm{d}}{\mathrm{d} x}\, \bar{\Phi}^{-1}(x) = -\frac{1}{\phi(\bar{\Phi}^{-1}(x))}.
        \]
        \item[(b)] For $0 \leq x \leq 1/2$, 
        \[
        t \cdot \bar{\Phi}^{-1}(t) \leq \phi\bigl(\bar{\Phi}^{-1}(t)\bigr) \leq t \cdot \biggl\{\bar{\Phi}^{-1}(t) + \frac{1}{\bar{\Phi}^{-1}(t)}\biggr\}.
        \]
    \end{itemize}
\end{lemma}

\begin{proof}
We prove each part in turn.

\medskip
\noindent \underline{Proof of part (a).} Since $\bar{\Phi}$ is continuously differentiable, it follows from the inverse function theorem that 
\[
\frac{\mathrm{d}}{\mathrm{d} x}\, \bar{\Phi}^{-1}(x) = \frac{1}{\bar{\Phi}'\bigl(\bar{\Phi}^{-1}(x)\bigr)} = -\frac{1}{\phi(\bar{\Phi}^{-1}(x))},
\]
as desired.

\medskip
\noindent \underline{Proof of part (b).}  The desired sandwich relation follows from the Mills' ratio bounds (\cref{lem:mills})
applied to $t = \bar{\Phi}^{-1}(x)$ for $0 \leq x \leq 1/2$.
\end{proof}

\subsection{Multivariate Gaussian concentration and moment properties}
\begin{lemma}\label{lem:tensor-concentration-helper-1}
    Let $z \sim \mathsf{N}(0,I_d)$ and $\ell \in \bN$. For an index $\alpha = (i_1,...,i_{\ell}) \in [d]^{\ell}$, let $X = (z^{\otimes \ell})_\alpha$. Then for any $\gamma > 1/2$ we have $\E [\exp(\frac{1}{8\gamma}|X|^{2/\ell})] \leq e^{1/2\gamma}$.
\end{lemma}
\begin{proof}
    We can write $X = \Pi_{i \in \supp(\alpha)} z_i^{m_i}$ for multiplicities $m_i \in \bN$.
    Observe
    \begin{align*}
        \frac{1}{8\gamma} |X|^{2/\ell} & = \frac{1}{8\gamma}\left(\Pi_{i \in \supp(\alpha)} z_i^{2m_i}\right)^{1/\ell} \\
        & \leq \frac{1}{8\gamma \ell} \sum_{i \in \supp(\alpha)} m_i z_i^2,
    \end{align*}
    where the last step follows from the AM-GM inequality.

    Recall for $t < 1/4$ we have $\E_{G \sim \mathsf{N}(0,1}[\exp(tG^2)] \leq (1-2t)^{-1} \leq \exp(4t)$. Our assumption on $\gamma$ implies $\frac{m_i}{8\gamma \ell} < \frac{1}{4}$ for all $i \in \supp(\alpha)$, and so it follows that
    \begin{align*}
        \E \left[\exp\left(\frac{1}{8\gamma}|X|^{2/\ell}\right)\right]
        &= \Pi_{i \in \supp(\alpha)} \E\left[\exp\left(\frac{1}{8\gamma \ell} m_i z_i^2\right)\right] \\
        & \leq \Pi_{i \in \supp(\alpha)} \exp\left( \frac{m_i}{2\gamma \ell}\right) \\
        & = e^{1/2\gamma}.
    \end{align*}
\end{proof}

\newcommand{\orlnrm}{8\ell}
\begin{lemma}\label{lem:tensor-concentration-helper}
    Let $z \sim \mathsf{N}(0,I_d)$ and $\ell \in \bN$. For an index $\alpha = (i_1,...,i_{\ell}) \in [d]^{\ell}$, let $X = (z^{\otimes \ell})_\alpha$. Then we have $\E [\exp(\frac{1}{\orlnrm}|X - \E[X]|^{2/\ell})] \leq 2$.
\end{lemma}
\begin{proof}
    Lemma~\ref{lem:tensor-concentration-helper-2} implies 
    \begin{align}
        \E \left[\exp\left(\frac{1}{\orlnrm}|X - \E[X]|^{2/\ell}\right)\right] 
        &\leq \E \left[\exp\left(\frac{1}{\orlnrm}\left(|X|^{2/\ell} + |\E[X]|^{2/\ell})\right)\right)\right] \nonumber \\
        &= \exp\left(\frac{1}{\orlnrm}|\E[X]|^{2/\ell}\right) \E \left[\exp\left(\frac{1}{\orlnrm}|X|^{2/\ell}\right)\right]. \label{eqn:tensor-concentration-helper-eqn}
    \end{align}
    Lemmas~\ref{lem:tensor-concentration-helper-3} and \ref{lem:tensor-concentration-helper-1} respectively imply we can bound both terms in equation~\eqref{eqn:tensor-concentration-helper-eqn} by $e^{1/4}$. Combining, we can bound their product by $e^{1/2} \leq 2$.
\end{proof}

\begin{lemma}\label{lem:tensor-concentration} %
    Let $\ell \in \bN$. Let $z_1,...,z_n \iid \mathsf{N}(0,I_d)$ and $\xi = \frac{1}{n} \sum_{i=1}^n z_i^{\otimes \ell} - \E_{z \sim \mathsf{N}(0,I_d)} z^{\otimes \ell}$. Then for $\delta \in (0,1)$ we have
    \begin{align*}
        \norm{\xi}_\infty \leq (100 \ell)^{\ell/2} \left(
            \sqrt{\frac{\ell\log d+\log(1/\delta)}{n}}
            +
            \frac{(\ell\log d+\log(1/\delta))^{\ell/2}}{n}
        \right)
    \end{align*}
    with probability at least $1 - \delta$.
\end{lemma}
\begin{proof}
    For any index $\alpha = (i_1,...,i_{\ell}) \in [d]^{\ell}$, combining Lemma~\ref{lem:tensor-concentration-helper} (setting $\ell = 2k$) with Theorem~3.1 of \cite{kuchibhotla2018movingbeyond} gives 
    \begin{align*}
        \P\left(|\xi_\alpha| \leq (100 \ell)^{\ell/2} \left(\sqrt{\frac{t}{n}} + \frac{t^{\ell/2}}{n}\right)\right) \leq 2e^{-t}.
    \end{align*}
    The result then follows from taking a union bound over all $d^{\ell}$ elements.
\end{proof}

\begin{corollary}\label{lem:tensor-concentration-corollary}
    Let $\ell \in \bN$. Let $z_1,...,z_n \simiid \mathsf{N}(0,I_d)$ and $\xi = \frac{1}{n} \sum_{i=1}^n z_i^{\otimes \ell} - \E_{z \sim \mathsf{N}(0,I_d)} z^{\otimes \ell}$. For $\beta > 0$, suppose
    \begin{align*}
        n \geq \max \left\{
            \frac{4(100 \ell)^{\ell} (\ell\log d+\log(1/\delta))}{\beta^{2}},
            \frac{2(100 \ell)^{\ell/2}(\ell\log d+\log(1/\delta))^{\ell/2}}{\beta}
        \right\}.
    \end{align*}
    Then $\norm{\xi}_\infty \leq \beta$ with probability at least $1-\delta$.
\end{corollary}
\begin{proof}
    The result follows immediately from inverting the bound of Lemma~\ref{lem:tensor-concentration}, setting both summands in the bound to $\beta/2$.
\end{proof}

\begin{lemma}\label{lem:tensor-concentration-all}
    Let $z_1,...,z_n \iid \mathsf{N}(0,I_d)$. For $\ell \in \bN$, let 
    $\xi = \frac{1}{n} \sum_{i=1}^n z_i^{\otimes \ell} - \E_{z \sim \mathsf{N}(0,I_d)} z^{\otimes \ell}$. Then  
    \begin{align*}
        \sup_{\norm{v}_2 \leq 1} \left|\frac{1}{n} \sum_{i=1}^n \langle z_i, v\rangle^\ell - \E_{z \sim \mathsf{N}(0,I_d)} [\langle z, v\rangle^\ell]\right| \leq d^{\ell/2}\norm{\xi}_\infty.
    \end{align*}
\end{lemma}
\begin{proof}
    Observe for any $x,y \in \R^d$ we have $\langle x,y \rangle^\ell = \langle x^{\otimes \ell}, y^{\otimes \ell} \rangle$. Moreover, for any unit vector $v \in \R^d$ we have
    \begin{align*}
        \norm{v^{\otimes \ell}_2} 
        & = \sum_{\alpha \in [d]^\ell} \left(\Pi_{i \in \alpha} v_i\right)^2 \\
        & = \left(\sum_{i=1}^d v_i^2\right)^k = 1.
    \end{align*}
    It thus follows for any such $v$ that
    \begin{align*}
        \left|\frac{1}{n} \sum_{i=1}^n \langle z_i, v\rangle^\ell - \E_{z \sim \mathsf{N}(0,I_d)} [\langle z, v\rangle^\ell]\right| 
        &= \left| \left\langle \frac{1}{n} \sum_{i=1}^n z_i^{\otimes \ell} - \E_{z \sim \mathsf{N}(0,I_d)} [ z^{\otimes \ell}], v^{\otimes \ell} \right\rangle \right| \\
        & \leq \norm{\xi}_2 \norm{v^{\otimes \ell}}_2 \\
        & \leq d^{\ell/2} \norm{\xi}_\infty.
    \end{align*}
\end{proof}

\subsubsection{Proof of Lemma~\ref{lem:lower-bound-E-prob}}\label{sec:proof-lem-sample-complexity}
\begin{proof}
    Our assumption on $n$ satisfies the conditions of Lemma~\ref{lem:binom-concentration}, implying $\frac{n}{|T|} \leq \frac{1+\epsilon}{1-\epsilon}$ with probability at least $1-\delta$.
    This then implies both $|T|$ and $n$ satisfy the conditions of Lemma~\ref{lem:tensor-concentration-corollary} with probability $\delta/2k$ and $\beta = \frac{\epsilon}{d^{k} k}$. Taking a union bound over $\ell \in [2k]$ gives the desired result.
\end{proof}

\subsection{Sum-of-squares facts}

\begin{lemma}\label{lem:sos-higher-powers}
    Let $r \in \bN$. Let $x,y$ by polynomials of degree at most $k$. Then
    \begin{align}
        \{x,y,y-x \geq 0\} \sos{x,y}{rk} y^r - x^r \geq 0.
    \end{align} 
\end{lemma}
\begin{proof}
    We can write $y^r - x^r$ as the sum of products of non-negative polynomials as follows:
    \begin{align*}
        y^r - x^r 
        & = \sum_{i=0}^{r-1} y^{r-i}x^i - y^{r-i-1}x^{i+1} \\
        &= (y-x) \sum_{i=0}^{r-1} y^{r-i-1} x^i.
    \end{align*} 
\end{proof}

\begin{lemma}\label{lem:sos-cauchy-schwarz}
    Let $x,y$ be vector-valued polynomials of degree at most $k$. Then for any $t > 0$ we have $\sos{x,y}{2k} \langle x,y \rangle \leq \frac{t}{2} \norm{x}_2^2 + \frac{1}{2t} \norm{y}_2^2$.
\end{lemma}
\begin{proof}
    Observe 
    \begin{align*}
        \frac{t}{2} \norm{x}_2^2 + \frac{1}{2t} \norm{y}_2^2 = \langle x , y \rangle + \frac{1}{2}\norm{\sqrt{t} x + \frac{1}{\sqrt{t}} y}_2^2.
    \end{align*}
\end{proof}

\begin{corollary}\label{lem:pe-cauchy-schwarz}
    Let $x,y$ be vector-valued polynomials of degree at most $k$ and let $\pE$ be a pseudoexpectation of degree at least $2k$. Then $\pE [\langle x, y\rangle] \leq \sqrt{\pE[\norm{x}_2^2]\pE[\norm{y}_2^2]}$.
\end{corollary}
\begin{proof}
    The result follows immediately from Lemma~\ref{lem:sos-cauchy-schwarz} with $t = \sqrt{\frac{\pE[\norm{y}_2^2]}{\pE[\norm{x}_2^2]}}$.
\end{proof}

\begin{lemma}\label{lem:pe-baby-holder}
     Let $c \in \bN$ with $r = 2c$ and $r' = 2(c+1)$. Let $x$ be a polynomial of degree at most $k$ and let $\pE$ be a pseudoexpectation of degree at least $2k (c+1)$. Then $\pE[x^{r'}] \geq \pE[x^r]^{\frac{r'}{r}}$.
\end{lemma}
\begin{proof}
    We proceed by induction on $c$. The base case $c=1$ follows from Cauchy-Schwarz (i.e., combining Lemmas~\ref{lem:pe-cauchy-schwarz} and \ref{lem:sos-higher-powers}). Now suppose the claim holds for $c' = c-1$, i.e., 
    \begin{align*}
        \pE[x^{2(c'+1)}] \geq \pE[x^{2c'}]^{\frac{c'+1}{c'}}.
    \end{align*}
    We will show the claim holds for $c = c'+1$. Applying Cauchy-Schwarz to $x^{c'+2}$ and $x^{c'}$ gives
    \begin{align*}
        \pE[x^{2(c'+2)}] 
        &\geq \frac{\pE[x^{2(c'+1)}]^2}{\pE[x^{2c'}]} \\
        & \geq \frac{\pE[x^{2(c'+1)}]^2}{\pE[x^{2(c'+1)}]^{\frac{c'}{c'+1}}} \\
        & = \pE[x^{2(c'+1)}]^{\frac{c'+2}{c'+1}}.
    \end{align*}
    with the second step following from the induction hypothesis.
\end{proof}

\begin{corollary}\label{lem:pe-holder}
    Let $r,r' \in \bN$ be even with $r' \geq r$ and let $x$ be a polynomial of degree at most $k$. Let $\pE$ be a pseudoexpectation of degree at least $kr'$. Then $\pE[x^{r'}] \geq \pE[x^r]^{\frac{r'}{r}}$.
\end{corollary}
\begin{proof}
    The result follows from repeatedly applying Lemma~\ref{lem:pe-baby-holder}.
\end{proof}

\begin{lemma}\label{lem:pe-jensen}
    Let $x$ be a polynomial of degree at most $k$ and let $\pE$ be a pseudoexpectation of degree at least $2k$. Then $\pE[x]^2 \leq \pE[x^2]$.
\end{lemma}
\begin{proof}
    Observe
    \begin{align*}
        \pE[x^2] - \pE[x]^2 = \pE[(x - \pE[x])^2] \geq 0.
    \end{align*}
\end{proof}

\begin{lemma}\label{lem:sos-mean-est-gauss}
    Let $z_1,...,z_n \iid \mathsf{N}(0,I_d)$, with $\xi = \frac{1}{n}\sum_{i=1}^n z_i^{\otimes 2k} - \E_{z \sim \mathsf{N}(0,I_d)} z^{\otimes 2k}$.
    Then it follows for $k \in \mathbf{N}$ that
    for all $t > 0$ that 
    \begin{align*}
        \sos{v}{4k} \frac{1}{n} \sum_{i=1}^n \langle v, z_i\rangle^{2k} \in \norm{v}_2^{2k} \E_{G \sim \mathsf{N}(0,1)} [G^{2k}] \pm \left( \frac{t}{2}\norm{v}_2^{4k}+ \frac{d^{2k}}{2t}\norm{\xi}_\infty^2\right).
    \end{align*}
\end{lemma}
\begin{proof}
    For the upper bound, observe
    \begin{align}
        \sos{v}{4k} \frac{1}{n} \sum_{i=1}^n \langle v, z_i\rangle^{2k}
        &= \norm{v}_2^{2k} \E_{G \sim \mathsf{N}(0,1)} [G^{2k}] + \langle v^{\otimes 2k}, \xi \rangle \label{eqn:gauss-sos-step1}\\
        &\leq \norm{v}_2^{2k} \E_{G \sim \mathsf{N}(0,1)} [G^{2k}] + \frac{t}{2}\langle v^{\otimes 2k}, v^{\otimes 2k} \rangle + \frac{1}{2t}\langle \xi,\xi\rangle
        \label{eqn:gauss-sos-step2}\\
        &= \norm{v}_2^{2k} \E_{G \sim \mathsf{N}(0,1)} [G^{2k}] + \frac{t}{2}\norm{v}_2^{4k} + \frac{1}{2t}\langle \xi,\xi\rangle \nonumber \\
        &\leq \norm{v}_2^{2k} \E_{G \sim \mathsf{N}(0,1)} [G^{2k}] + \frac{t}{2} \norm{v}_2^{4k} +\frac{d^{2k}}{2t}\norm{\xi}_\infty^2, \nonumber
    \end{align}
    where equation~\eqref{eqn:gauss-sos-step1} follows from the fact that $\E_{G \sim \mathsf{N}(0,1)} [\langle v, G \rangle^{2k}] = \norm{v}_2^{2k} \E_{G \sim \mathsf{N}(0,1)} [G^{2k}]$ for all $v \in \R^d$ and inequality~\eqref{eqn:gauss-sos-step2} follows from SoS Cauchy--Schwarz (Lemma~\ref{lem:sos-cauchy-schwarz}). 
    
    The lower bound follows analogously.
    
\end{proof}

\begin{lemma}\label{lem:pe-1-p-alpha}
    Let $x \in \R$ be a polynomial of degree at most $k$. Let $r \in \bN$ be even and let $\pE$ be a pseudoexpectation of degree at least $kr$ satisfying $\pE[x^r] \leq 1+ \alpha$ for some $\alpha > 0$.
    Then it follows that $\pE\left[x^2\right] \leq 1 + \frac{2\alpha}{r}$.
\end{lemma}
\begin{proof}
    Holder's inequality for pseudoexpectations (Lemma~\ref{lem:pe-baby-holder}) implies
    \begin{align*}
        \pE\left[x^2\right] \leq \pE\left[x^r\right]^{2/r} \leq \left(1 + \alpha \right)^{2/r} \leq 1 + \frac{2\alpha}{r}.
    \end{align*}
\end{proof}

\begin{lemma}\label{lem:pe-1-m-alpha-helper}
    Let $x \in \R$ be a polynomial of degree at most $k$. Let $r \in \bN$ be even and let $\pE$ be a pseudoexpectation of degree at least $2kr$ satisfying both $\pE[x^r] \geq 1 - \beta$ and $\pE[x^{2r}] \geq 1 - \alpha$ for some $\alpha,\beta > 0$.
    Then it follows that $\pE\left[x^r\right] \geq 1 - \frac{\alpha}{2-\beta}$.
\end{lemma}
\begin{proof}
    Observe
    \begin{align*}
        \pE[1 - x^{2r}]
        & = \pE[(1-x^r)(1+x^r)] \\
        & \geq (2-\beta) \pE[1-x^r].
    \end{align*}
    Thus, $\pE[1 - x^{2r}] \leq \alpha$ implies $\pE[1-x^r] \leq \frac{\alpha}{2-\beta}$.
\end{proof}

\begin{corollary}\label{lem:pe-1-m-alpha}
    Let $x \in \R$ be a polynomial of degree at most $k$. Let $r = 2^p$ for some $p \in \bN$ and $\alpha \in (0,1/2)$. Let $\pE$ be a pseudoexpectation of degree at least $2kr$ satisfying $\pE[x^{2^\ell}] \geq 1 - \alpha$ for $\ell \in [p]$. 
    Then we have $\pE[x^2] \geq 1 - \frac{2 \exp(3/2) \alpha}{r}$.
    
\end{corollary}
\begin{proof}
    Let $\beta_1 = \frac{\alpha}{2-\alpha}$ and $\beta_i = \frac{\beta_{i-1}}{2-\alpha}$ for $2 \leq i \leq p$. Applying Lemma~\ref{lem:pe-1-m-alpha-helper} recursively over $\ell \in [p]$ then gives 
    \begin{align*}
        \pE[x^2] 
        &\geq 1 - \frac{\alpha}{\Pi_{i=1}^{p-1} (2-\beta_i)} \\
        &\geq  1 - \frac{\alpha}{\Pi_{i=1}^{p-1} 2 \exp(-\beta_i)} \\
        &= 1 - \frac{2\alpha}{r \exp \left(-\sum_{i=1}^{p-1}\beta_i\right)}.
    \end{align*}
    where the second step follows from the fact that $\exp(-2x) \leq 1-x$ for $x < 1/2$ (taking $x = \frac{\beta_i}{2}$).

    Our premise that $\alpha < 1/2$ implies $\beta_{i+1} \leq \frac{2}{3} \beta_{i}$ for $i \in [p-1]$, and so $\Pi_{i=1}^p \sum_{i=1}^{p-1}\beta_i \leq 3\alpha$. The desired result follows immediately from combining this fact with the previous display.
\end{proof}

\end{document}